\documentclass[hidelinks, 11pt]{article}

\usepackage[UKenglish]{babel}
\usepackage[utf8]{inputenc}
\usepackage[T1]{fontenc}
\usepackage{erewhon}
\usepackage{amsmath}
\usepackage{mathrsfs}
\usepackage{amsthm}
\usepackage{amsfonts}
\usepackage{amssymb}
\usepackage{mathtools}
\usepackage{dsfont} 

\usepackage{quiver}
\usetikzlibrary{nfold}
\usepackage{adjustbox}

\usepackage[dvipsnames]{xcolor}
\usepackage{hyperref}
\usepackage{fullpage}
\usepackage{microtype}
\usepackage{booktabs}
\usepackage{enumitem} 
\usepackage{comment}
\usepackage{authblk} 
\usepackage{bussproofs}

\usepackage[autostyle,italian=guillemets]{csquotes}
\usepackage[backend=biber]{biblatex}
\tikzset{node distance=2cm, auto}
\tikzcdset{row sep/normal=2.7em,column sep/normal=3.5em}

\allowdisplaybreaks

\setlist[description]{font=\normalfont\bfseries\:\!}

\renewcommand\labelenumi{\normalfont\textbf{[\arabic{enumi}]}}
\renewcommand\theenumi\labelenumi

\hypersetup{
    colorlinks=true,
    linkcolor=gray,
    filecolor=gray,      
    urlcolor=gray,
    citecolor=gray
    }
\numberwithin{equation}{section}

\newtheorem{theorem}{Theorem}[section]
\newtheorem{proposition}[theorem]{Proposition}
\newtheorem{lemma}[theorem]{Lemma}
\newtheorem{corollary}[theorem]{Corollary}

\theoremstyle{definition}
\newtheorem{definition}[theorem]{Definition}
\newtheorem{example}[theorem]{Example}

\newtheorem{remark}[theorem]{Remark}

\newenvironment{contentX}[1][blank]{\noindent{\large\bfseries\scshape\MakeLowercase{#1}}.\thinspace\noindent}{}
\newenvironment{content}[1][blank]{\vskip .5cm\begin{contentX}[#1]}{\end{contentX}}
\newenvironment{acknowledgements}{\begin{content}[Acknowledgements]}{\end{content}}

\newenvironment{conventions}{\begin{content}[Conventions]}{\end{content}}

\renewenvironment{abstract}{\begin{contentX}[Abstract]}{\end{contentX}}

\DeclareSymbolFont{sfoperators}{OT1}{cmss}{m}{n}

\makeatletter
\def\operator@font{\mathgroup\symsfoperators}
\makeatother

\newcommand{\CategoryFont}[1]{\mathsf{\uppercase{#1}}}
\newcommand{\FunctorFont}[1]{\mathsf{#1}}

\DeclareMathOperator{\tr}{tr}
\DeclareMathOperator{\LHS}{LHS}
\DeclareMathOperator{\RHS}{RHS}

\newcommand{\x}{\otimes}

\renewcommand{\.}{\cdot}

\newcommand{\1}{\mathds{1}}

\renewcommand{\=}{\mathrel{\mathop:}=}
\renewcommand{\epsilon}{\varepsilon}
\renewcommand{\o}{\circ}
\renewcommand{\b}{\bullet}
\newcommand{\<}{\langle}
\renewcommand{\>}{\rangle}
\newcommand{\op}{\FunctorFont{op}}

\renewcommand{\*}{\ast}
\newcommand{\nto}{\mathrel{\mkern3mu\vcenter{\hbox{$\scriptscriptstyle|\:$}}\mkern-12mu{\to}}}

\newcommand{\Tag}[1]{\tag{$#1$}}

\newcommand{\Z}{\mathbb{Z}}

\newcommand{\R}{\mathbb R}

\newcommand{\Alg}{\CategoryFont{Alg}}
\newcommand{\cAlg}{\CategoryFont{cAlg}}
\newcommand{\Mod}{\CategoryFont{Mod}}

\newcommand{\KL}{\FunctorFont{KL}}
\newcommand{\id}{\FunctorFont{id}}

\newcommand{\X}{\mathbb{X}}

\newcommand{\End}{\FunctorFont{End}}

\newcommand{\Set}{\CategoryFont{Set}}
\newcommand{\Cat}{\CategoryFont{Cat}}

\DeclareMathOperator{\VF}{VF}
\DeclareMathOperator{\DO}{DO}
\DeclareMathOperator{\DB}{DB}

\newcommand{\p}{\mathbf{p}}

\renewcommand{\d}{\FunctorFont{d}}
\newcommand{\T}{\mathrm{T}}
\newcommand{\TT}{\mathbb{T}}

\newcommand{\SMan}{\CategoryFont{SMan}}

\newcommand{\Aff}{\CategoryFont{Aff}}

\newcommand{\TngCat}{\CategoryFont{TngCat}}

\DeclareMathOperator{\LIVF}{LIVF}
\DeclareMathOperator{\LI}{LI}
\DeclareMathOperator{\Lie}{Lie}
\DeclareMathOperator{\GL}{GL}
\DeclareMathOperator{\SL}{SL}

\DeclareMathOperator{\LIDer}{LIDer}
\DeclareMathOperator{\Der}{Der}
\DeclareMathOperator{\Sym}{Sym}

\DeclareMathOperator{\AD}{AD}
\DeclareMathOperator{\Ad}{Ad}
\DeclareMathOperator{\ad}{ad}

\newcommand{\gl}{\mathfrak{gl}}
\newcommand{\lnr}{\mathfrak{lnr}}
\renewcommand{\sl}{\mathfrak{sl}}
\renewcommand{\L}{\mathsf{L}}
\newcommand{\Simple}{\FunctorFont{S}}

\newcommand{\g}{\mathfrak{g}}
\newcommand{\h}{\mathfrak{h}}
\newcommand{\LieAlg}{\CategoryFont{LieAlg}}
\newcommand{\LieGr}{\CategoryFont{LieGrp}}
\newcommand{\mc}{\FunctorFont{mc}}
\newcommand{\e}{\FunctorFont{e}}
\newcommand{\m}{\FunctorFont{m}}
\renewcommand{\i}{\FunctorFont{i}}
\newcommand{\Gr}{\CategoryFont{Grp}}
\newcommand{\SQ}[1]{[\![#1]\!]}
\newcommand{\PLIVF}{PLIVF }
\newcommand{\PLIVFs}{PLIVFs }
\renewcommand{\P}{\mathscr{P}}

\title{The Lie group-Lie algebra correspondence in tangent categories}
\author{\uppercase{Marcello Lanfranchi}}
\affil{\normalsize\textit{Macquarie University, School of Mathematical and Physical Sciences}}
\date{}

\begin{document}

\maketitle

\begin{abstract}
Classic Lie theory establishes a correspondence between Lie groups, which are internal group objects in the category of smooth manifolds, and Lie algebras. An analogous correspondence exists also for group objects in affine schemes. Both smooth manifolds and affine schemes form tangent categories, which provide a categorical context for differential geometry. Therefore, it is natural to ask whether the Lie correspondence can be constructed entirely from the tangent structure.

Building on work of Cockett and Schwarz, we develop an internal Lie group-Lie algebra correspondence in tangent categories. We introduce Lie group objects as group objects in a tangent category which admit a tangent space at the unit. For any group object we construct a Lie functor and show it is tangentially representable exactly when the group object is a Lie group. Representability then yields an internal Lie algebra whose underlying object is the tangent space at the unit. We prove this Lie algebra is a differential Lie algebra, with a bilinear Lie bracket induced by the adjoint representation, and we extend the construction to a functor from Lie groups to differential Lie algebras. Finally, we recover the usual Lie correspondences in both differential and algebraic geometry.
\end{abstract}

\begin{acknowledgements}
Mathematics is an emergent phenomenon that arises from the collaboration between researchers and is not the achievement of single individuals. Even though the main ideas of this paper are part of the author's original work, they would not exist without the numerous discussions and suggestions of various people of the growing community of tangent and differential categories. In particular, we would like to thank JS Lemay for corrections, suggestions and support. We are also grateful to Geoff Cruttwell, Richard Garner, Robin Cockett, Florian Schwarz, and Edmund Heng for useful discussions and suggestions. For this work, the author was supported by the AFOSR under award number FA9550-24-1-0008.
\end{acknowledgements}

\par\noindent\rule{\textwidth}{1pt}


\tableofcontents

\section{Introduction}
\label{section:introduction}
The mathematical language to study the symmetries of a system is the language of groups. While discrete symmetries, such as the symmetries of a regular polygon, are captured by finite groups, continuous and smooth symmetries, such as the rotation of a spherical object, are captured by Lie groups. A Lie group consists of a group whose underlying set is equipped with a topology and a differential structure that turns it into a smooth manifold. Furthermore, the multiplication and the inverse maps of a Lie group are smooth functions.

The smooth structure enables the application of powerful tools of differential geometry to study the algebraic properties of the group. One of the most important consequences of this is the fundamental theorem of Lie theory, which establishes a correspondence between Lie groups and Lie algebras: \textbf{every Lie group admits an associated Lie algebra}.

Classically, this construction is presented in at least two equivalent ways. On the one hand, the Lie algebra $\g$ of a Lie group $G$ is the Lie algebra of \textbf{left-invariant vector fields} of $G$, which are sections $X\colon G\to\T G$ of the tangent bundle $p_G\colon\T G\to G$ of $G$, invariant under the left action of the group, that is, $X_{gh}=\d L_gX_h$, where $L_g(h)=gh$. In this definition, the Lie bracket is the restriction of the usual Lie bracket of vector fields.

On the other hand, $\g$ is equivalently defined as the tangent space $\T_\e G$ of $G$ at the unit $\e\in G$. In this second definition, the Lie bracket is defined by the differential $\d_\e\hat\Ad_G\colon\g\to\lnr_\g$ at the unit of the adjoint representation $\hat\Ad_G\colon G\to\GL_\g$ of $G$ onto $\g$.

The Lie algebra of a Lie group retains local information about the group itself. In particular, if two Lie groups have isomorphic Lie algebras, they are \emph{locally} isomorphic. Furthermore, the reconstruction theorem establishes that every finite-dimensional Lie algebra $\g$ is in fact the Lie algebra of a simply-connected Lie group $\Gamma(\g)$, so every (connected and locally connected) Lie group $G$ is locally isomorphic to a simply connected Lie group $\Gamma(\g)$.

Categorically, a Lie group corresponds to an internal group object of the category $\SMan$ of finite-dimensional smooth manifolds and smooth maps. The category $\SMan$ carries a \textbf{tangent structure}.

\begin{content}[Tangent categories]
Tangent categories (Definition~\ref{definition:tangent-category}) are a well-established categorical framework for differential geometry. Introduced by Rosick\'y in~\cite{rosicky:tangent-cats} and later generalized by Cockett and Cruttwell in~\cite{cockett:tangent-cats}, a tangent category consists of a collection of objects, interpreted as generalized geometric spaces, a collection of morphisms, intepreted as smooth functions, and an assignment $\T$, which sends each object $M$ to a new object $\T M$ and each morphism $f\colon M\to N$ to a new morphism $\T f\colon\T M\to\T N$.

The object $\T M$ should be regarded as the \textbf{tangent bundle} of $M$, while $\T f$ is interpreted as the smooth function which sends each tangent vector $u\in\T M$ of $M$ at $x$ to the differential $\d_xf$ of $f$ at $x$ applied to $u$, that is, to $\d_xf(u)$. The assignment $\T$ is functorial, that is, $\T$ preserves identities and composition of morphisms.

The tangent bundle $\T M$ comes equipped with an \textbf{additive structure}, given by the following natural transformations:
\begin{align*}
&p_M\colon\T M\to M     &&z_M\colon M\to\T M        &&s_M\colon\T_2M\to\T M
\end{align*}
$p_M$ is interpreted as the map that sends a tangent vector to its base point, $z_M$ is the map that sends each point to its zero tangent vector, and $s_M$ is the sum of two tangent vectors with the same base point, $\T_2M$ denoting the space of pairs of tangent vectors that share the same base point.

A tangent structure has two other natural transformations: the \textbf{vertical lift}, $l_M\colon\T M\to\T\T M$ and the \textbf{canonical flip} $c_M\colon\T\T M\to\T\T M$. The vertical lift satisfies a universal property which is interpreted as follows. Each fibre $\T_xM\=p^{-1}(x)$ of the tangent bundle $\T M$ of $M$ is \textit{linear} in that the tangent bundle $\T\T_xM$ of $\T_xM$ is isomorphic to the Cartesian product $\T_xM\times\T_xM$. This linearity condition generalizes the property enjoyed by Euclidean spaces for which $\T\R^n\cong\R^n\times\R^n$. This expresses the idea that an object in a tangent category is a \emph{locally linear} space and a morphism is a \emph{locally linear} function.

Finally, the canonical flip, $c_M\colon\T\T M\to\T\T M$, is an isomorphism of the double tangent bundle, which encodes the symmetry of partial derivatives, that is, $\partial_i\partial_jf=\partial_j\partial_if$.

The category $\SMan$ is the archetypal example of a tangent category, in which the tangent bundle functor corresponds to the usual notion. However, tangent categories capture geometry well beyond smooth manifolds. In particular, the category $\Aff_R$ of \textbf{affine schemes} carries a canonical tangent structure, as described in~\cite{cruttwell:algebraic-geometry}. By identifying the category $\Aff_R$ with the opposite category $\cAlg_R^\op$ of commutative and unital $R$-algebras, the tangent bundle functor sends each algebra $A$ to the symmetric algebra of the module of K\"ahler differentials, that is, $\T^\Omega A\=\Sym_A\Omega_A$.
\end{content}

\begin{content}[Affine group schemes]
The correspondence between Lie groups and Lie algebras extends beyond differential geometry. In particular, the work of Demazure and Grothendieck~\cite{grothendieck:group-schemes} extended the Lie correspondence to algebraic geometry. An affine group scheme is an internal object in the category $\Aff_R$ of affine schemes. This corresponds to a commutative Hopf algebra $H$. The Lie algebra of an affine group scheme is the Lie algebra of \textbf{left-invariant derivations} of the corresponding Hopf algebra.

Concretely, a left-invariant derivation of $H$ consists of an $R$-linear morphism $\delta\colon H\to H$ subject to the Leibniz rule, that is, $\delta(xy)=x\delta(y)+y\delta(x)$, and invariant under the left-action of the comultiplication $\Delta\colon H\to H\x H$ of $H$.
\end{content}

\vskip .5cm

Therefore, on the one hand, the Lie correspondence is not only a feature of differential geometry, but rather a phenomenon that appears in other geometric contexts, including algebraic geometry. On the other hand, both the category $\SMan$ of smooth manifolds and the category $\Aff_R$ of affine schemes are examples of tangent categories.

Numerous concepts of differential geometry have already been internalized to tangent categories: vector fields and their Lie bracket~\cite{rosicky:tangent-cats,cockett:tangent-cats,cockett:jacobi}, Euclidean spaces and vector bundles~\cite{cockett:differential-bundles}, Koszul and Ehresmann connections~\cite{cockett:connections,cruttwell:ehresmann-connections}, ordinary differential equations~\cite{cockett:differential-equations}, etc.

Thus, one is left to wonder if the tangent structure suffices to construct the Lie correspondence. In this paper, we answer this question affirmatively.

\begin{content}[The external Lie algebra of a Lie group]
Some work in this direction has already been done. Cockett and Schwarz in~\cite{cockett:lie-groups-tangent-cats} studied group objects in tangent categories and constructed for such a group $G$, the \emph{external} Lie algebra of left-invariant vector fields, which is the \emph{set} of vector fields $X\colon G\to\T G$ which are invariant under the left action of the group. They also showed that, when $G$ admits the tangent space $\T_\e G$ at the unit, defined as the (tangent) pullback of the tangent bundle $p_G\colon\T G\to G$ along the unit map $\e\colon\*\to G$, left-invariant vector fields are in bijective correspondence with the global elements $\*\to\T_\e G$ of the tangent space.

However, Cockett and Schwarz's paper does not provide an \emph{internal} Lie algebra, that is, a Lie algebra object internal to the tangent category itself. In fact, the Lie algebra of left-invariant vector fields lives \emph{outside} the tangent category, since it is defined as a set of vector fields. To fully compare the Lie algebra of Cockett and Schwarz with the tangent space $\T_\e G$, one needs to \emph{internalize} this construction.
\end{content}

\begin{content}[Main goal of the paper]
In this paper, we provide a good notion of \textbf{internal Lie algebra}. There are a few good reasons that push us to pursue this objective.

The first reason is purely theoretical: achieving a full Lie correspondence in the context of tangent categories, which generalizes both the classical correspondence of differential geometry and the one of algebraic geometry. This will prove that the Lie correspondence is fully determined by the tangent structure and not by the intrinsic properties of the geometric spaces underlying the Lie groups. This will allow us to explore the Lie correspondence in new territories, beyond the classical contexts of differential and algebraic geometry.

The second reason is more practical. This is an important stepping stone towards an internal language of Lie groups in tangent categories that can be used in future developments to study other geometric constructions, such as principal bundles, principal connections, and their curvature form. This will also extend the Cartan calculus, developed for tangent categories in~\cite{aintablian:cartan-calculus-tangent-cats} by Aintablian and Blohmann.
\end{content}

\begin{content}[Lie groups in tangent categories]
We define a \textbf{Lie group} in a Cartesian tangent category as a group object $G$ which admits the tangent space at the unit (Definition~\ref{definition:lie-group}), that is, for which the tangent pullback of the tangent bundle $p_G\colon\T G\to G$ along the unit map $\e\colon\*\to G$ exists.

At first glance, this definition does not give an explicit description of the Lie algebra of $G$. To make it explicit, we first construct the \textbf{Lie functor} of a group object with negatives (Definition~\ref{definition:lie-functor}). This functor sends each object $A$ of the tangent category to the Lie algebra of $A$-parametrized left-invariant vector fields of $G$ (Definition~\ref{definition:parametrized-left-invariant-vector-field}). These are morphisms $X\colon G\times A\to\T G$ invariant under the left action of the group and such that, composed with the projection $p_G\colon\T G\to G$, give $\pi_1\colon G\times M\to G$.

To prove that parametrized left-invariant vector fields form in fact a Lie algebra, we characterize them as left-invariant vector fields in a new tangent category, denoted by $\Simple_A[\X,\TT]$, which is the fibre over $A$ of the simple tangent fibration of $(\X,\TT)$ (Propositions~\ref{proposition:simple-tangent-fibration} and~\ref{proposition:simple-fibration-fibres}).

With Proposition~\ref{proposition:lie-functor-is-representable}, we prove that the Lie functor $\g(-)$ of a Lie group is \textbf{tangently representable}, that is, there is a natural isomorphism $\varphi\colon\g(A)\to\X(A,\T_\e G)$ between the Lie functor and the Hom-functor $\X(-,\T_\e G)$. Furthermore, we show that the representability of $\g(-)$ is compatible with the tangent structure (Definition~\ref{definition:tangently-representable}).

With Proposition~\ref{proposition:representable-implies-lie}, we also prove the converse: the tangent representability of the Lie functor of a group object with negatives $G$ implies that $G$ is a Lie group. We also prove that a Lie functor is tangently representable if and only if it is \textbf{differentiably representable}, which means that the representing object $\g$ is a differential object and the representing isomorphism $\varphi$ preserves multilinearity (Definition~\ref{definition:differentiably-representable}).

This establishes a full characterization of Lie groups, as stated by Theorem~\ref{theorem:first-fundamental-lie}, as an equivalence between the following conditions for a group object $G$:
\begin{enumerate}
\item $G$ is a Lie group object;

\item $G$ admits negatives and its Lie functor is tangently representable;

\item $G$ admits negatives and its Lie functor is differentiably representable;

\item $G$ is parallelizable.
\end{enumerate}
\end{content}

\begin{content}[The differential Lie algebra of a Lie group]
In Section~\ref{section:lie-algebras}, we exploit the representability of the Lie functor $\g(-)$ to construct the internal Lie algebra of a Lie group (Theorem~\ref{theorem:lie-algebra-of-lie-group}). This is defined as the representing object $\g$ of the Lie functor, that is, the tangent space $\T_\e G$ at the unit, equipped with a Lie algebra structure $0_\g\colon\*\to\g$, $+_\g\colon\g\times\g\to\g$, $[,]_\g\colon\g\times\g\to\g$, induced by the Lie algebra structure of $\g(A)$ via the Yoneda lemma.

We also show that the Lie algebra of a Lie group is in fact a \textbf{differential Lie algebra}, that is, a Lie algebra object $\g$ over a differential object, whose additive structure coincides with the one of the underlying differential object and whose Lie bracket $[,]_\g\colon\g\times\g\to\g$ is bilinear (Definition~\ref{definition:differential-lie-algebra} and Theorem~\ref{theorem:second-fundamental-lie}).

Finally, we provide a fully internal formula for the Lie algebra structure of $\g$ as the differential $\ad_\g\colon\g\times\g\to\g$ of the adjoint representation $\Ad_G\colon G\times\g\to\g$, extending to the general settings of tangent categories the well-known formula $\ad(x)(y)=[x,y]$ of differential geometry (Theorem~\ref{theorem:adjoint-formula}).
\end{content}

\begin{content}[The Lie correspondence in tangent categories]
We use this formula to prove in Theorem~\ref{theorem:functoriality} that the assignment for a Lie group $G$ of its internal differential Lie algebra $\g$ extends to a functor:
\begin{align*}
&\Lie\colon\LieGr(\X,\TT)\to\LieAlg(\X,\TT)
\end{align*}
With this functor, we establish that the classical Lie correspondence generalizes to an adjunction $\Gamma\dashv\Lie$ between the functor $\Lie$ and a left adjoint $\Gamma\colon\LieAlg(\X,\TT)\to\LieGr(\X,\TT)$. A tangent category which exhibits $\Lie$ as a right adjoint is called a \textbf{Lie tangent category} (Definition~\ref{definition:lie-tangent-category}).

In Section~\ref{section:functoriality}, we prove that our Lie correspondence generalizes both the classic differential geometry one and the one of algebraic geometry (Theorem~\ref{theorem:lie-algebra-of-affine-group-schemes}).
\end{content}

\begin{conventions}
Throughout the paper, we denote the composition of two morphisms $f\colon A\to B$ and $g\colon B\to C$ using the diagrammatic convention, that is, we denote by $fg$ the composition of $f$ followed by $g$. This aligns with the tangent categorical literature. We will use functional notation for the application of functions and functors to elements, objects, and morphisms. So, for example, we write $\T\T M$ for the application of the functor $\T$ to $\T$ to $M$. We do the same for the application of a functor to a morphism, so that $\T f$ denotes the application of the functor $\T$ to the morphism $f$. Similarly, in the examples, we often refer to actual functions. In that case, we write $f(g(x))$ for the application of $f$ to $g$ to $x$.

The Hom-set of a category $\X$ between two objects $A$ and $B$ is denoted by $\X(A,B)$.
\end{conventions}


\section{Lie Groups in tangent categories}
\label{section:lie-groups}
In this section, we introduce Lie groups in tangent categories. We begin by recalling the definition of a group object in an arbitrary \textbf{Cartesian category}, that is, a category equipped with a terminal object $\*$ and binary products, denoted by $A\times B$, whose projections are denoted by $\pi_1\colon A\times B\to A$ and $\pi_2\colon A\times B\to B$.

\subsection{Background}
\label{subsection:background}
A group object in a Cartesian category consists of an object $G$ equipped with three morphisms
\begin{align*}
&\e_G\colon\*\to G     &&\i_G\colon G\to G        &&\m_G\colon G\times G\to G
\end{align*}
respectively called the \textbf{unit}, the \textbf{inverse map}, and the \textbf{multiplication}, subject to the familiar conditions satisfied by a group, that is, unitality, associativity, and the inverse axiom:
\begin{equation*}
\adjustbox{width=\linewidth}{
\begin{tikzcd}
G & {G\times G} & G \\
& G
\arrow["{\<\id_G,\e\>}", from=1-1, to=1-2]
\arrow[equals, from=1-1, to=2-2]
\arrow["\m", from=1-2, to=2-2]
\arrow["{\<\e,\id_G\>}"', from=1-3, to=1-2]
\arrow[equals, from=1-3, to=2-2]
\end{tikzcd}\quad
\begin{tikzcd}
{G\times G\times G} & {G\times G} \\
{G\times G} & G
\arrow["{\m\times\id_G}", from=1-1, to=1-2]
\arrow["{\id_G\times\m}"', from=1-1, to=2-1]
\arrow["\m", from=1-2, to=2-2]
\arrow["\m"', from=2-1, to=2-2]
\end{tikzcd}\quad
\begin{tikzcd}
G & {G\times G} & G \\
\* & G & \*
\arrow["{\<\id_G,\i\>}", from=1-1, to=1-2]
\arrow[from=1-1, to=2-1]
\arrow["\m", from=1-2, to=2-2]
\arrow["{\<\i,\id_G\>}"', from=1-3, to=1-2]
\arrow[from=1-3, to=2-3]
\arrow["\e"', from=2-1, to=2-2]
\arrow["\e", from=2-3, to=2-2]
\end{tikzcd}
}
\end{equation*}

\begin{example}
\label{example:ordinary-groups}
Ordinary groups are group objects in the Cartesian category of sets.
\end{example}

\begin{example}
\label{example:abelian-groups}
The category $\Gr$ of groups is also Cartesian, with terminal object the trivial group $\{\e\}$ containing a single element, and Cartesian products are products of groups. Therefore, one can look at group objects in $\Gr$. By an Eckmann-Hilton argument, group objects in $\Gr$ are Abelian groups.
\end{example}

\begin{example}
\label{example:lie-groups}
The category $\SMan$ of finitely-dimensional smooth manifolds is a Cartesian category and group objects of $\SMan$ correspond to ordinary Lie groups.
\end{example}

\begin{example}
\label{example:affine-group-schemes}
The category $\Aff_R$ of affine schemes over a commutative and unital ring $R$ is dual to the category $\cAlg_R$ of commutative and unital $R$-algebras, via the classical equivalence given by the $\mathsf{Spec}$ functor. Thus, $\Aff_R$ is complete and cocomplete and therefore Cartesian. Under this equivalence, group objects in $\Aff_R$ correspond to commutative Hopf algebras. A commutative Hopf algebra consists of a commutative and unital $R$-algebra $H$, equipped with a coassociative and counital comultiplication $\Delta\colon H\to H\x H$ and a counit $\epsilon\colon H\to R$ making $H$ into a bialgebra. Furthermore, a commutative Hopf algebra carries an $R$-linear morphism $S\colon H\to H$, known as the antipode, subject to the following conditions. $\Delta(\id_H\x S)\m=\Delta(S\x\id_H)\m=\epsilon\eta$, where $\m$ denotes the multiplication of $H$ and $\eta\colon R\to H$ sends $1$ to $1$.

As a concrete example of an affine group scheme, we consider the group $\SL_2$. For starters, consider the functor $\SL_2(-)\colon\cAlg_R^\op\to\Mod_R$ which sends each $R$-algebra $A$ to the special linear group $\SL_2(A)$ with coefficients in $A$. This is the group of invertible $(2\times 2)$-matrices $K\in\SL_2(A)$ with coefficients in $A$ and whose determinant is equal to $1$, that is, $\det(K)=1$. It turns out that the functor $\SL_2(-)$ is in fact representable, represented by the $R$-algebra:
\begin{align*}
&\SL_2\=\frac{R[x_{1,1},x_{1,2},x_{2,1},x_{2,2}]}{\left(\det(x_{i,j})-1\right)}
\end{align*}
where $\det(x_{i,j})=x_{1,1}x_{2,2}-x_{1,2}x_{2,1}$. To see why $\SL_2$ represents the functor $\SL_2(-)$, consider a morphism $A\to\SL_2$ in $\cAlg_R^\op$, which corresponds to a morphism $K\colon\SL_2\to A$ in $\cAlg_R$. However, such a morphism is fully determined by its value at the variables $x_{i,j}$, that is, $K(x_{i,j})$. However, since in $\SL_2$, $\det(x_{i,j})-1=0$, $K$ is an algebra map if and only if $\det(K(x_{i,j}))=1$.

By representability, the group structure of each $\SL_2(A)$ induces a group structure on $\SL_2$ in $\Aff_R$. Concretely, the counit sends $x_{1,1}$ and $x_{2,2}$ to $1$ and $x_{1,2}$ and $x_{2,1}$ to $0$ and the comultiplication $\Delta$ is defined on the generators as follows:
\begin{align*}
&\Delta(x_{1,1})=x_{1,1}\x x_{1,1}+x_{1,2}\x x_{2,1}    &&\Delta(x_{1,2})=x_{1,1}\x x_{1,2}+x_{1,2}\x x_{2,2}\\
&\Delta(x_{2,1})=x_{2,1}\x x_{11}+x_{2,2}\x x_{2,1}      &&\Delta(x_{2,2})=x_{2,1}\x x_{1,2}+x_{2,2}\x x_{2,2}
\end{align*}
Finally, the antipode $S$ is defined by:
\begin{align*}
&S(x_{1,1})=x_{2,2}        &&S(x_{1,2})=-x_{1,2}       &&S(x_{2,1})=-x_{2,1}       &&S(x_{2,2})=x_{1,1}
\end{align*}
We suggest the reader to see~\cite{waterhouse:affine-group-schemes} for an introduction on affine group schemes and more examples.
\end{example}

In this paper, we study group objects in a \emph{Cartesian tangent category}. Tangent categories were introduced by Rosick\'y in~\cite{rosicky:tangent-cats} and later revisited and generalized by Cockett and Cruttwell~\cite{cockett:tangent-cats}. Tangent categories are now a well-established field of research which offers a categorical framework for differential geometry that extends well beyond standard differential geometry.

We begin by recalling the definition of a tangent category as given by Cockett and Cruttwell. We start with the notion of additive bundles.

\begin{definition}[{\cite[Definitions~2.1 and~2.2]{cockett:tangent-cats}}]
\label{definition:additive-bundles}
In a category $\X$, an \textbf{additive bundle} consists of a triple of morphisms
\begin{align*}
&q\colon E\to M     &&z_q\colon M\to E       &&s_q\colon E_2\to E
\end{align*}
respectively called the \textbf{projection}, the \textbf{zero}, and the \textbf{sum}, where $E_n$ denotes the $n$-fold pullback
\begin{equation*}
\begin{tikzcd}
{E_n} & E \\
E & M
\arrow["{\pi_n}", from=1-1, to=1-2]
\arrow["{\pi_1}"', from=1-1, to=2-1]
\arrow["\lrcorner"{anchor=center, pos=0.125}, draw=none, from=1-1, to=2-2]
\arrow["q", from=1-2, to=2-2]
\arrow["\dots"{marking, allow upside down}, shift left=4, draw=none, from=2-1, to=1-2]
\arrow["q"', from=2-1, to=2-2]
\end{tikzcd}
\end{equation*}
of $q$ along itself. Furthermore, these morphisms are subject to the following conditions:
\begin{equation*}
\begin{tikzcd}
E & \\
M & M
\arrow["q", from=1-1, to=2-2]
\arrow["{z_q}", from=2-1, to=1-1]
\arrow[equals, nfold, from=2-1, to=2-2]
\end{tikzcd}\qquad
\begin{tikzcd}
{E_2} & E \\
E & M
\arrow["{s_q}", from=1-1, to=1-2]
\arrow["{\pi_1}"', from=1-1, to=2-1]
\arrow["q", from=1-2, to=2-2]
\arrow["q"', from=2-1, to=2-2]
\end{tikzcd}
\end{equation*}
\begin{equation*}
\begin{tikzcd}
{E_2} & \\
{E_2} & E
\arrow["{s_q}", from=1-1, to=2-2]
\arrow["\tau", from=2-1, to=1-1]
\arrow["{s_q}"', from=2-1, to=2-2]
\end{tikzcd}\qquad
\begin{tikzcd}
{E_3} & {E_2} \\
{E_2} & E
\arrow["{s_q\times\id_E}", from=1-1, to=1-2]
\arrow["{\id_E\times s_q}"', from=1-1, to=2-1]
\arrow["{s_q}", from=1-2, to=2-2]
\arrow["{s_q}"', from=2-1, to=2-2]
\end{tikzcd}\qquad
\begin{tikzcd}
{E_2} & E \\
E & E
\arrow["{s_q}", from=1-1, to=2-2]
\arrow["{\<\id_E,qz_q\>}"', from=1-2, to=1-1]
\arrow[equals, nfold, from=1-2, to=2-2]
\arrow["{\<qz_q,\id_E\>}", from=2-1, to=1-1]
\arrow[equals, nfold, from=2-1, to=2-2]
\end{tikzcd}
\end{equation*}
establishing that $z_q$ and $s_q$ are bundle morphisms and encoding commutativity, associativity, and unitality of $s_q$, respectively, where $\tau\colon E_2\to E_2$ denotes the canonical symmetry.

Furthermore, an \textbf{Abelian bundle} consists of an additive bundle $q\colon E\to M$ equipped with an extra morphism
\begin{align*}
&n_q\colon E\to E
\end{align*}
called \textbf{negation}, which is an inverse map for the sum, in that it satisfies the following conditions:
\begin{equation*}
\begin{tikzcd}
E & E \\
& M
\arrow["{n_q}", from=1-1, to=1-2]
\arrow["q"', from=1-1, to=2-2]
\arrow["q", from=1-2, to=2-2]
\end{tikzcd}\qquad
\begin{tikzcd}
E & {E_2} \\
M & E
\arrow["{\<\id_E,n_q\>}", from=1-1, to=1-2]
\arrow["q"', from=1-1, to=2-1]
\arrow["{s_q}", from=1-2, to=2-2]
\arrow["{z_q}"', from=2-1, to=2-2]
\end{tikzcd}
\end{equation*}

An \textbf{additive bundle morphism} consists of a pair $(f,g)\colon q\to q'$ of morphisms $f\colon M\to M'$ and $g\colon E\to E'$ that commute with the projections, the zero, and the sum, as follows:
\begin{equation*}
\begin{tikzcd}
E & {E'} \\
M & {M'}
\arrow["g", from=1-1, to=1-2]
\arrow["q"', from=1-1, to=2-1]
\arrow["{q'}", from=1-2, to=2-2]
\arrow["f"', from=2-1, to=2-2]
\end{tikzcd}\qquad
\begin{tikzcd}
E & {E'} \\
M & {M'}
\arrow["g", from=1-1, to=1-2]
\arrow["{z_q}", from=2-1, to=1-1]
\arrow["f"', from=2-1, to=2-2]
\arrow["{z_q'}"', from=2-2, to=1-2]
\end{tikzcd}\qquad
\begin{tikzcd}
{E_2} & {E_2'} \\
E & {E'}
\arrow["{g\times_fg}", from=1-1, to=1-2]
\arrow["{s_q}"', from=1-1, to=2-1]
\arrow["{s_q'}", from=1-2, to=2-2]
\arrow["g"', from=2-1, to=2-2]
\end{tikzcd}
\end{equation*}
\end{definition}

We can now recall the definition of a tangent category.

\begin{definition}[{\cite[Definitions~2.3 and~2.8]{cockett:tangent-cats}}]
\label{definition:tangent-category}
A \textbf{tangent structure} on a category $\X$ consists of a tuple $\TT\=(\T,p,z,s,l,c)$ formed by an endofunctor
\begin{align*}
&\T\colon\X\to\X
\end{align*}
called the \textbf{tangent bundle functor} and by five natural transformations
\begin{align*}
&p_M\colon\T M\to M     &&z_M\colon M\to\T M    &&s_M\colon\T_2M\to\T M\\
&l_M\colon\T M\to\T\T M                         &&&&c_M\colon\T\T M\to\T\T M
\end{align*}
natural in $M$, respectively called the \textbf{projection}, the \textbf{zero}, the \textbf{sum}, the \textbf{vertical lift}, and the \textbf{canonical flip}, where $\T_nM$ denotes the $n$-fold pullback
\begin{equation}
\label{equation:n-fold-pullback-Tn}
\begin{tikzcd}
{\T_nM} & {\T M} \\
{\T M} & M
\arrow["{\pi_n}", from=1-1, to=1-2]
\arrow["{\pi_1}"', from=1-1, to=2-1]
\arrow["\lrcorner"{anchor=center, pos=0.125}, draw=none, from=1-1, to=2-2]
\arrow["{p_M}", from=1-2, to=2-2]
\arrow["\dots"{marking, allow upside down}, shift left=4, draw=none, from=2-1, to=1-2]
\arrow["{p_M}"', from=2-1, to=2-2]
\end{tikzcd}
\end{equation}
of the projection along itself, subject to the following conditions:
\begin{description}
\item[TNG.1] The $n$-fold pullback of Equation~\ref{equation:n-fold-pullback-Tn} is a \textbf{tangent $n$-fold pullback}, in that it is preserved by all iterates $\T^n\=\T\T{\dots}\T$ of the tangent bundle functor;

\item[TNG.2] For each $M$, the triple $\p_M\=(p_M,z_M,s_M)$ is an additive bundle;

\item[TNG.3] The pair $(z_M,l_M)\colon(p_M,z_M,s_M)\to(\T p_M,\T z_M,\T s_M)$ is an additive bundle morphism: 
\begin{equation*}
\begin{tikzcd}
{\T M} & {\T^2M} \\
M & {\T M}
\arrow["l_M", from=1-1, to=1-2]
\arrow["p_M"', from=1-1, to=2-1]
\arrow["{\T p_M}", from=1-2, to=2-2]
\arrow["z_M"', from=2-1, to=2-2]
\end{tikzcd}\qquad
\begin{tikzcd}
{\T M} & {\T^2M} \\
M & {\T M}
\arrow["l_M", from=1-1, to=1-2]
\arrow["z_M", from=2-1, to=1-1]
\arrow["z_M"', from=2-1, to=2-2]
\arrow["{\T z_M}"', from=2-2, to=1-2]
\end{tikzcd}\qquad
\begin{tikzcd}
{\T_2M} & {\T\T_2M} \\
{\T M} & {\T^2M}
\arrow["{l_M\times_{z_M}l_M}", from=1-1, to=1-2]
\arrow["s_M"', from=1-1, to=2-1]
\arrow["{\T s_M}", from=1-2, to=2-2]
\arrow["l_M"', from=2-1, to=2-2]
\end{tikzcd}
\end{equation*}

\item[TNG.4] The pair $(\id_M,c_M)\colon(p_M,z_M,s_M)\to(\T p_M,\T z_M,\T s_M)$ is an additive bundle morphism:
\begin{equation*}
\begin{tikzcd}
{\T^2M} & {\T^2M} \\
{\T M} & {\T M}
\arrow["c_M", from=1-1, to=1-2]
\arrow["{p_{\T M}}"', from=1-1, to=2-1]
\arrow["{\T p_M}", from=1-2, to=2-2]
\arrow[equals, from=2-1, to=2-2]
\end{tikzcd}\quad
\begin{tikzcd}
{\T^2M} & {\T^2M} \\
{\T M} & {\T M}
\arrow["c_M", from=1-1, to=1-2]
\arrow["{z_{\T M}}", from=2-1, to=1-1]
\arrow[equals, from=2-1, to=2-2]
\arrow["{\T z_M}"', from=2-2, to=1-2]
\end{tikzcd}\quad
\begin{tikzcd}
{\T_2\T M} && {\T\T_2M} \\
{\T^2M} && {\T^2M}
\arrow["{c_M\times_{\T M}c_M}", from=1-1, to=1-3]
\arrow["{s_{\T M}}"', from=1-1, to=2-1]
\arrow["{\T s_M}", from=1-3, to=2-3]
\arrow["c_M"', from=2-1, to=2-3]
\end{tikzcd}
\end{equation*}

\item[TNG.5] The following diagrams commute:
\begin{equation*}
\begin{tikzcd}
{\T M} & {\T^2M} \\
{\T^2M} & {\T^3M}
\arrow["l_M", from=1-1, to=1-2]
\arrow["l_M"', from=1-1, to=2-1]
\arrow["{\T l_M}", from=1-2, to=2-2]
\arrow["{l_{\T M}}"', from=2-1, to=2-2]
\end{tikzcd}\quad
\begin{tikzcd}
{\T^2M} & {\T^2M} \\
& {\T^2M}
\arrow["c_M", from=1-1, to=1-2]
\arrow[equals, from=1-1, to=2-2]
\arrow["c_M", from=1-2, to=2-2]
\end{tikzcd}\quad
\begin{tikzcd}
{\T^3M} & {\T^3M} & {\T^3M} \\
{\T^3M} & {\T^3M} & {\T^3M}
\arrow["{\T c_M}", from=1-1, to=1-2]
\arrow["{c_{\T M}}"', from=1-1, to=2-1]
\arrow["{c_{\T M}}", from=1-2, to=1-3]
\arrow["{\T c_M}", from=1-3, to=2-3]
\arrow["{\T c_M}"', from=2-1, to=2-2]
\arrow["{c_{\T M}}"', from=2-2, to=2-3]
\end{tikzcd}
\end{equation*}

\item[TNG.6] The following diagrams commute:
\begin{equation*}
\begin{tikzcd}
{\T M} & {\T^2M} \\
& {\T^2M}
\arrow["l_M", from=1-1, to=1-2]
\arrow["l_M"', from=1-1, to=2-2]
\arrow["c_M", from=1-2, to=2-2]
\end{tikzcd}\qquad\hfill
\begin{tikzcd}
{\T^2M} & {\T^3M} & {\T^3M} \\
{\T^2M} && {\T^3M}
\arrow["{l_{\T M}}", from=1-1, to=1-2]
\arrow["c_M"', from=1-1, to=2-1]
\arrow["{\T c_M}", from=1-2, to=1-3]
\arrow["{c_{\T M}}", from=1-3, to=2-3]
\arrow["{\T l_M}"', from=2-1, to=2-3]
\end{tikzcd}
\end{equation*}

\item[TNG.7] The tangent bundle is locally linear, that is, the following diagram
\begin{equation*}
\begin{tikzcd}
{\T_2M} & {\T\T M} \\
M & {\T M}
\arrow["{\xi_M}", from=1-1, to=1-2]
\arrow["{\pi_1p_M}"', from=1-1, to=2-1]
\arrow["{\T p_M}", from=1-2, to=2-2]
\arrow["{z_M}"', from=2-1, to=2-2]
\end{tikzcd}
\end{equation*}
is a pullback diagram, where\footnote{In~\cite{cockett:tangent-cats}, tangent bundle $\T M$ of $M$ is interpreted as the space of pairs $(u,x)$ formed by a vector $u$ and by a point $x$. We use the opposite convention, in that we think of the elements of $\T M$ as pairs $(x,u)$ formed by a point and a vector. In particular, in the original paper the map $\xi_M$ was defined as $(l_M\times z_{\T M})\T s_M$.}:
\begin{align*}
&\xi_M\colon\T_2M\xrightarrow{z_{\T M}\times l_M}\T\T_2M\xrightarrow{\T s_M}\T\T M
\end{align*}
\end{description}
A \textbf{tangent category} is a category $\X$ equipped with a tangent structure $\TT$. Furthermore, a \textbf{Cartesian tangent category} is a Cartesian category $\X$ equipped with a tangent structure $\TT$ whose tangent bundle functor $\T$ preserves the terminal object and the binary products. Finally, in a tangent category, an object $M$ \textbf{admits negatives} if there is a \textbf{negation} morphism $n_M\colon\T M\to\T M$, natural in $M$, which makes each $\p_M$ into an Abelian bundle. 
\end{definition}

The objects of a tangent category should be regarded as \emph{generalized smooth spaces} and morphisms as \emph{smooth functions}. The object $\T M$, called the \textbf{tangent bundle} of $M$, should be interpreted as the space of tangent vectors of $M$ at every point, so, a generic element of $\T M$ should be thought as a pair $(x,u)$ formed by a generic point $x$ of $M$ and a tangent vector $u$ at $x$.

The projection $p_M$ is regarded as the map that sends each tangent vector $(x,u)$ to its base point $x$, the zero, as sending each $x$ to its zero tangent vector $(x,0)$, the sum, as summing two tangent vectors $(x,u)$ and $(x,v)$ with same base point to their sum $(x,u+v)$.

While the notation $(x,u)$ for the generic element of $\T M$ is suggestive, it only makes sense \emph{locally}. In fact, in general, the tangent bundle $\T M$ fails globally to be the Cartesian product $M\times M$. Even though $\T M$ is not \emph{globally} trivial, it exhibits a \emph{local linear} behaviour. This is expressed through the universal property of the vertical lift (\textbf{[TNG.7]}). This axiom establishes that the fibres $\T_xM$ of the tangent bundle are \emph{linear spaces} in that their tangent bundle $\T\T_xM$ is isomorphic to $\T_xM\times\T_xM$ in a way that preserves the additivity structure. Finally, the canonical flip $c_M$ encodes the symmetry of partial derivatives, that is, $\partial_i\partial_jf=\partial_j\partial_if$, which is often referred to in the literature as the symmetry of the Hessian matrix.

\begin{example}
\label{example:tangent-categories-trivial}
Every category comes with a trivial tangent structure, whose tangent bundle functor is the identity functor and whose structural natural transformations are also the identity.
\end{example}

\begin{example}
\label{example:tangent-categories-sman}
The category $\SMan$ of finitely-dimensional smooth manifolds is the archetypal example of a tangent category, whose tangent bundle functor coincides with the usual notion, whose additive structure is provided by the additive structure of the tangent spaces, whose vertical lift is induced by the local linearity of the spaces, and whose canonical flip is induced by the symmetry of partial derivatives. $\SMan$ is also Cartesian.
\end{example}

\begin{example}
\label{example:tangent-categories-c-alg}
The category $\cAlg_R$ of commutative and unital algebras over a unital and commutative base ring $R$ also comes with a tangent structure denoted by $\TT^\epsilon$. The tangent bundle $\T^\epsilon A$ of an algebra $A$ is the algebra $A[\epsilon]$, where $\epsilon^2=0$. The projection $p^\epsilon_A\colon A[\epsilon]\to A$ sends $\epsilon$ to $0$, the zero $z^\epsilon_A\colon A\to A[\epsilon]$ sends $1$ to $1$, the sum $s^\epsilon_A\colon A[\epsilon_1,\epsilon_2]\to A[\epsilon]$, where $\epsilon_1^2=\epsilon_1\epsilon_2=\epsilon_2^2=0$, sends both $\epsilon_1$ and $\epsilon_2$ to $\epsilon$, the lift $l_A^\epsilon\colon A[\epsilon]\to A[\epsilon][\epsilon']$ sends $\epsilon$ to $\epsilon\epsilon'$ and finally, the canonical flip $c_A^\epsilon\colon A[\epsilon][\epsilon']\to A[\epsilon][\epsilon']$ sends $\epsilon$ to $\epsilon'$ and viceversa. This example was initially introduced by Cockett and Cruttwell in~\cite[Section~5.4]{cockett:tangent-cats} and studied in detail by Cruttwell and Lemay in~\cite{cruttwell:algebraic-geometry}. $(\cAlg_R,\TT^\epsilon)$ is also Cartesian.
\end{example}

\begin{example}
\label{example:tangent-categories-affine}
The category $\Aff_R$ of affine schemes over $R$ is also a Cartesian tangent category. Via the $\mathsf{Spec}$ functor, $\Aff_R$ is in fact equivalent to the category $\cAlg_R^\op$. By identifying these two categories, the tangent bundle of a commutative and unital $R$-algebra $A$ is the symmetric algebra of the module of K\"alher differentials of $A$, that is, $\T^\Omega A\=\Sym_A\Omega_A$. Concretely, this algebra is generated by the elements $a$ of $A$ and by symbols $\d a$, for each $a\in A$, subject to the following relations:
\begin{align*}
&\d1=0          &&\d(ra+sb)=r\d a+s\d b     &&\d(ab)=a\d b+b\d a
\end{align*}
for each $r,s\in R$, $a,b\in A$. Furthermore, the unit of $\T^\Omega A$ coincides with $1\in A$, and the multiplication of two elements of $A$ is as in $A$. The structural natural transformations, regarded as algebra morphisms, are defined on generators as follows. The projection $p_A^\Omega\colon A\to\T^\Omega A$ sends each $a$ to itself, the zero $z_A^\Omega\colon\T^\Omega A\to A$ sends $a$ to itself and $\d a$ to $0$, the sum $s_A^\Omega\colon\T^\Omega A\to\T^\Omega A\x\T^\Omega A$ sends $a$ to itself and $\d a$ to $\d a\x 1+1\x\d a$, the vertical lift $l_A^\Omega\colon\T^\Omega\T^\Omega A\to\T^\Omega A$ sends $a$ to itself, $\d a$ and $\d' a$ to $0$ and $\d'\d a$ to $\d a$. Finally, the canonical flip $c_A^\Omega\colon\T^\Omega\T^\Omega A\to\T^\Omega\T^\Omega A$ sends $a$ and $\d'\d a$ to themselves and it exchanges $\d a$ with $\d'a$ and viceversa. This tangent category was initially introduced by Cockett and Cruttwell in~\cite[Section~5.4]{cockett:tangent-cats} and studied in detail by Cruttwell and Lemay in~\cite{cruttwell:algebraic-geometry}.
\end{example}

\begin{example}
\label{example:tangent-categories-operadic-affine-schemes}
The category $\Alg_\P$ of algebras of a symmetric operad $\P$ over the symmetric monoidal category $\Mod_R$ of $R$-modules is also a Cartesian tangent category. The tangent structure generalizes the one of Example~\ref{example:tangent-categories-c-alg}. In particular, it coincides with that one for $\P$, the operad that encodes commutative and unital $R$-algebras. Furthermore, the category $\Aff_\P$ of \emph{operadic affine schemes}, opposite to $\Alg_\P^\op$, is also a tangent category. This construction generalizes Example~\ref{example:tangent-categories-affine} for arbitrary operads over $\Mod_R$. We refer to the original paper~\cite{ikonicoff:operadic-algebras-tangent-cats} for details on these constructions.
\end{example}

\subsection{Lie groups and their linear Generalized Cartesian Differential Category}
\label{subsection:lie-groups}
The classic notions of Lie groups of differential geometry (Example~\ref{example:lie-groups}) and of affine group schemes (Example~\ref{example:affine-group-schemes}) are both examples of group objects in Cartesian tangent categories (Examples~\ref{example:tangent-categories-sman} and~\ref{example:tangent-categories-affine}).

The main goal of this paper is to extend the construction of the Lie algebra of a Lie group to tangent categories, which is defined as the tangent space $\T_\e G$ at the unit of a Lie group $G$. In arbitrary tangent categories, tangent spaces are a derived notion, while the tangent bundle is taken as primitive.

\begin{definition}[{\cite[Definition~]{cockett:differential-bundles}}]
\label{definition:tangent-space}
We say that an object $M$ is a Cartesian tangent category, equipped with a global element $x\colon\*\to M$, \textbf{admits the tangent space} at $x$ provided that the pullback
\begin{equation*}
\begin{tikzcd}
{\T_xM} & {\T M} \\
{\*} & M
\arrow["{\iota_x}", from=1-1, to=1-2]
\arrow[from=1-1, to=2-1]
\arrow["\lrcorner"{anchor=center, pos=0.125}, draw=none, from=1-1, to=2-2]
\arrow["{p_M}", from=1-2, to=2-2]
\arrow["x"', from=2-1, to=2-2]
\end{tikzcd}
\end{equation*}
of the tangent bundle $p_M\colon\T M\to M$ along $x$ exists and is a \textbf{tangent pullback}, in that it is preserved by all iterates $\T^n$ of the tangent bundle functor. In this case, the object $\T_xM$ is called the \textbf{tangent space} of $M$ at $x$.
\end{definition}

Cockett and Schwarz in~\cite{cockett:lie-groups-tangent-cats} already considered group objects in a Cartesian tangent category which admit the tangent space $\T_\e G$ at the unit $\e\colon\*\to G$. In this paper, we show that these are not only a reasonably well-behaved class of group objects but they are the correct notion of Lie groups in a Cartesian tangent category. Moved by this consideration, we suggest the following definition.

\begin{definition}
\label{definition:lie-group}
A \textbf{Lie group} in a Cartesian tangent category is a group object $G$ which admits a tangent space at the unit.
\end{definition}

An important achievement of Cockett and Schwarz's paper~\cite{cockett:lie-groups-tangent-cats} was to prove that the tangent bundle $\T G$ of a Lie group object $G$ is always trivial, in that $\T G\cong G\times\T_\e G$.

In this section, we review this result through a structural lens: after proving that arbitrary group objects form a Cartesian tangent category, we use Cockett and Schwarz's result to prove that the assignment $G\mapsto\T_\e G$, which sends each Lie group object $G$ to its tangent space at the unit, is in fact a linear assignment.

Linear assignments were introduced and studied by Ikonicoff, Lemay, and Van Der Linden in~\cite{ikonicoff:abelianization-tangent-cats}. Here, we recall this definition.

\begin{definition}[{\cite[Definition~2.1]{ikonicoff:abelianization-tangent-cats}}]
\label{definition:linear-assignment}
A \textbf{linear assignment} consists of an endofunctor $\L\colon\X\to\X$ of a Cartesian category $\X$ that preserves the terminal object and binary products, equipped with natural transformations
\begin{align*}
&0_A\colon\*\to\L A     &&+_A\colon\L A\times\L A\to\L A        &&\nu_A\colon\L A\to\L\L A
\end{align*}
such that each $(\L A,0_A,+_A)$ becomes a commutative monoid of $\X$, $\nu_A$ is a natural isomorphism, and $\nu_{\L A}=\L\nu_A$.
\end{definition}

As proved by~\cite[Theorem~3.5]{ikonicoff:abelianization-tangent-cats}, a Cartesian category equipped with a linear assignment becomes a Cartesian tangent category whose tangent bundle functor $\T^\L$ is defined by:
\begin{align}
\label{equation:linear-GCDC}
&\T^\L A\= A\times\L A      &&\T^\L(f\colon A\to B)\=f\times\L f
\end{align}
The projection $p^\L_A\colon A\times\L A\to A$ is the first projection, the zero and the sum are induced by the commutative monoid structure $(\L A,0_A,+_A)$, the vertical lift $l^\L_A\colon A\times\L A\to(A\times\L A)\times(\L A\times\L\L A)$ sends the component $A$ to itself and $\L A$ to $\L\L A$ using $\nu_A$. Finally, the canonical flip $c^\L_A\colon(A\times\L A)\times(\L A\times\L\L A)\to(A\times\L A)\times(\L A\times\L\L A)$ exchanges the internal $\L A$ components. We suggest the reader consult~\cite{ikonicoff:abelianization-tangent-cats} for details.

\begin{definition}[Ikonicoff, Lemay, Van Der Linden\footnote{This definition was made implicit in~\cite{ikonicoff:abelianization-tangent-cats}.}]
\label{definition:linear-GCDC}
A \textbf{linear Generalized Cartesian Differential Category} (\textbf{linear GCDC}) consists of a Cartesian tangent category whose tangent structure is fully determined by a linear assignment $\L$ as in Equation~\eqref{equation:linear-GCDC}.
\end{definition}

\begin{remark}
\label{remark:linear-GCDC}
Generalized Cartesian Differential Categories (GCDCs, for short) were introduced by Cruttwell in~\cite{cruttwell:generalized-CDC} as a generalization of Cartesian Differential Categories (CDCs). A linear GCDC is in fact a special case of a GCDC.
\end{remark}

Following Cockett and Schwarz's construction, for a Lie group $G$ we define the \textbf{Maurer-Cartan form} as the unique morphism $\mc_G\colon\T G\to\T_\e G$ which makes the diagram
\begin{equation*}
\begin{tikzcd}
{\T G} && {\T(G\times G)} & {\T(G\times G)} \\
&& {\T_\e G} & {\T G} \\
&& \1 & G
\arrow["{\<p_Gz_G,\id_{\T G}\>}", from=1-1, to=1-3]
\arrow["{\mc_G}", dashed, from=1-1, to=2-3]
\arrow[curve={height=24pt}, from=1-1, to=3-3]
\arrow["{\T(\i_G\times\id_G)}", from=1-3, to=1-4]
\arrow["{\T\m_G}", from=1-4, to=2-4]
\arrow["{\iota_\e}", from=2-3, to=2-4]
\arrow[from=2-3, to=3-3]
\arrow["\lrcorner"{anchor=center, pos=0.125}, draw=none, from=2-3, to=3-4]
\arrow["{p_G}", from=2-4, to=3-4]
\arrow["{\e_G}"', from=3-3, to=3-4]
\end{tikzcd}
\end{equation*}
commutative. In the category of smooth manifolds, $\mc_G$ is precisely the Maurer-Cartan form, which, in local coordinates, sends $(g,x)\in\T G$ to $\d_gL_{g^{-1}}(x)$, where $L_h\colon G\to G$ denotes the left action, that is, $L_h(a)=ha$ and $\d_gL_{g^{-1}}\colon\T_gG\to\T_\e G$ is the differential of $L_{g^{-1}}$ at $g$.

Intuitively, $\mc_G$ moves the tangent vectors at a point $g$ over the unit. Crucially, this operation is invertible, that is, tangent vectors at a point $g$ are in bijective correspondence with the tangent vectors at the unit. This is expressed by Cockett and Schwarz's theorem, which we recall here.

\begin{theorem}[{\cite[Theorem~1]{cockett:lie-groups-tangent-cats}}]
\label{theorem:trivial-tangent-bundle-groups}
The tangent bundle $\T G$ of a Lie group object $G$ is trivial, in that the following diagram
\begin{equation*}
\begin{tikzcd}
G & {\T G} & {\T_\e G}
\arrow["{p_G}"', from=1-2, to=1-1]
\arrow["\mc_G", from=1-2, to=1-3]
\end{tikzcd}
\end{equation*}
is a product diagram.
\end{theorem}

Structurally, Theorem~\ref{theorem:trivial-tangent-bundle-groups} tells us that the Cartesian tangent category of Lie group objects is a linear GCDC. To make this statement precise, we begin by showing that the category of group objects forms a Cartesian tangent category.

Let $\Gr(\X,\TT)$ denote the category of group objects internal to $(\X,\TT)$ and group homomorphisms. Furthermore, let $\LieGr(\X,\TT)$ denote the full subcategory of $\Gr(\X,\TT)$ spanned by Lie group objects of $(\X,\TT)$.

\begin{lemma}
\label{lemma:groups-to-base-creates-limits}
The forgetful functor $\Gr(\X,\TT)\to(\X,\TT)$ which forgets the group structure reflects limits. 
\end{lemma}
\begin{proof}
Consider a diagram $D\colon I\to\Gr(\X,\TT)$ and a cone $\lambda_c\colon L\to D(c)$ of $D$ which is sent to a limit cone in $(\X,\TT)$ by the forgetful functor. Since the components of the cone $\lambda_c$ are group homomorphisms, $\e_L\lambda_c=\e_{D(c)}$, for each $c\in I$. Moreover, $\m_L\lambda_c=(\lambda_c\times\lambda_c)\m_{D(c)}$. Now, consider a cone $\gamma_c\colon X\to D(c)$ of $D$ in $\Gr(\X,\TT)$. Since $\lambda_c$ is a limit cone in $(\X,\TT)$, there exists a unique morphism $\theta\colon X\to L$ of cones in $(\X,\TT)$. Our goal is to show that $\theta$ lifts to $\Gr(\X,\TT)$. However, $\e_X\theta\lambda_c=\e_X\gamma_c=\e_{D(c)}$. Therefore, by the universal property of $\lambda_c$, $\e_X\theta=\e_L$. Finally, $\m_X\theta\lambda_c=\m_X\gamma_c=(\gamma_c\times\gamma_c)\m_{D(c)}=(\theta_c\times\theta_c)(\lambda_c\times\lambda_c)\m_{D(c)}=(\theta_c\times\theta_c)\m_L\lambda_c$. Using the universal property of $\lambda_c$ again, we conclude that $\theta$ is in fact a group homomorphism, proving that the forgetful functor reflects limits.
\end{proof}

\begin{lemma}
\label{lemma:group-objects-form-tangent-category}
Group objects of a Cartesian tangent category $(\X,\TT)$ form a Cartesian tangent category $\Gr(\X,\TT)$. Moreover, $\LieGr(\X,\TT)$ becomes a Cartesian tangent subcategory of $\Gr(\X,\TT)$.
\end{lemma}
\begin{proof}
For starters, notice that, since the tangent bundle functor of the base Cartesian tangent category is a Cartesian functor, it preserves group objects and group homomorphisms. Furthermore, the structural natural transformations are group homomorphisms because of naturality. Therefore, to prove the statement, it remains to show the universal property of the $n$-fold pullback $\T_nM$ of the tangent bundle and the one of the vertical lift. However, these universal properties are a consequence of the forgetful functor $\Gr(\X,\TT)\to(\X,\TT)$ reflecting limits. To prove that $\LieGr(\X,\TT)$ is a Cartesian tangent subcategory, notice that if $G$ is a Lie group object, then since the pullback that defines $\T_\e G$ is preserved by the tangent bundle functor, $\T G$ has tangent space at the unit given by $\T\T_\e G$.
\end{proof}

Our goal is to use Theorem~\ref{theorem:trivial-tangent-bundle-groups} to show that the tangent category $\LieGr(\X,\TT)$ is in fact a linear GCDC.

\begin{lemma}
\label{lemma:tangent-space-linear-assignment}
There is a Cartesian endofunctor $\T_\e$ of $\LieGr(\X,\TT)$ which sends a group object $G$ to its tangent space $\T_\e G$ at the unit.
\end{lemma}
\begin{proof}
Consider a group homomorphism $f\colon G\to H$. By using the universal property of the tangent pullback that defines $\T_\e G$ and by using that $f$ preserves the units of the groups, it is not hard to construct $\T_\e f\colon\T_\e G\to\T_{\e_H}H$ as the unique morphism such that $\T_\e f\iota_\e=\iota_\e\T f$. By uniqueness of $\T_\e f$, it is also easy to see that $\T_\e\id_G=\id_{\T_\e G}$ and that $\T_\e(fg)=\T_\e f\T_\e g$, making $\T_\e$ into a functor $\T_\e\colon\LieGr(\X,\TT)\to(\X,\TT)$. Furthermore, since $\T_\e G$ is defined via a pullback and limits commute with other limits, it is also immediate that $\T_\e$ preserves finite products. This proves that $\T_\e$ is Cartesian. Thus, we can apply $\T_\e$ to the structure maps $\e_G,\m_G,\i_G$ of a group object $G$ in $(\X,\TT)$ and obtain the structure maps:
\begin{align*}
&\*\cong\T_\e\*\xrightarrow{\T_\e\e_G}\T_\e G   &&\T_\e G\times\T_\e G\cong\T_\e(G\times G)\xrightarrow{\T_\e\m_G}\T_\e G&&\T_\e G\xrightarrow{\T_\e\i_G}\T_\e G
\end{align*}
Since functors preserve equations, $\T_\e G$ is a new group object.
\end{proof}

Tangent spaces are always differential objects (see~\cite[Corollary~3.5]{cockett:differential-bundles}). Since it will appear a few times throughout the paper, we may recall this definition.

\begin{definition}[{\cite[Defitinition~3.1]{cockett:tangent-cats}}]
A \textbf{differential object} in a Cartesian tangent category consists of an object $A$ equipped with three morphisms
\begin{align*}
&0_A\colon\*\to A       &&+_A\colon A\times A\to A      &&\hat p_A\colon\T A\to A
\end{align*}
respectively called, the \textbf{zero}, the \textbf{sum}, and the \textbf{differential projection} of $A$, subject to the following conditions:
\begin{description}
\item[DO.1] The triple $(A,0_A,+_A)$ is a commutative monoid, that is, the terminal map $!_A\colon A\to\*$ is an additive bundle;

\item[DO.2] The pair $(!_A,\hat p_A)\colon\p_A\to!_A$ is a additive bundle morphism:
\begin{equation*}
\begin{tikzcd}
{\T A} & A \\
A & {*}
\arrow["{\hat p_A}", from=1-1, to=1-2]
\arrow["{z_A}", from=2-1, to=1-1]
\arrow["{!_A}"', from=2-1, to=2-2]
\arrow["{0_A}"', from=2-2, to=1-2]
\end{tikzcd}\qquad
\begin{tikzcd}
{\T_2A} & {A\times A} \\
{\T A} & A
\arrow["{\hat p_A\times_{!_A}\hat p_A}", from=1-1, to=1-2]
\arrow["{s_A}"', from=1-1, to=2-1]
\arrow["{+_A}", from=1-2, to=2-2]
\arrow["{\hat p_A}"', from=2-1, to=2-2]
\end{tikzcd}
\end{equation*}

\item[DO.3] The pair $(!_{\T A},\hat p_A)\colon(\T!_A,\T0_A,\T+_A)\to(!_A,0_A,+_A)$ is an additive bundle morphism:
\begin{equation*}
\begin{tikzcd}
{\T A} & A \\
{\T\*} & {*}
\arrow["{\hat p_A}", from=1-1, to=1-2]
\arrow["{\T0_A}", from=2-1, to=1-1]
\arrow["\cong"', from=2-1, to=2-2]
\arrow["{0_A}"', from=2-2, to=1-2]
\end{tikzcd}\qquad
\begin{tikzcd}
{\T(A\times A)} & {A\times A} \\
{\T A} & A
\arrow["{\hat p_A\times\hat p_A}", from=1-1, to=1-2]
\arrow["{\T+_A}"', from=1-1, to=2-1]
\arrow["{+_A}", from=1-2, to=2-2]
\arrow["{\hat p_A}"', from=2-1, to=2-2]
\end{tikzcd}
\end{equation*}

\item[DO.4] The differential projection is linear, that is:
\begin{equation*}
\begin{tikzcd}
{\T\T A} & {\T A} \\
{\T A} & A
\arrow["{\T\hat p_A}", from=1-1, to=1-2]
\arrow["{\hat p_A}", from=1-2, to=2-2]
\arrow["{l_A}", from=2-1, to=1-1]
\arrow["{\hat p_A}"', from=2-1, to=2-2]
\end{tikzcd}
\end{equation*}

\item[DO.5] The differential projection exhibit $\T A$ as the Cartesian product $A\times A$ that is, the following diagram
\begin{equation*}
\begin{tikzcd}
A & {\T A} & A
\arrow["{p_A}"', from=1-2, to=1-1]
\arrow["{\hat p_A}", from=1-2, to=1-3]
\end{tikzcd}
\end{equation*}
is a product diagram.
\end{description}
A \textbf{linear morphism} of differential objects consists of a morphism $f\colon A\to B$ that commutes with the differential projections, that is:
\begin{equation*}
\begin{tikzcd}
{\T A} & {\T B} \\
A & B
\arrow["{\T f}", from=1-1, to=1-2]
\arrow["{\hat p_A}"', from=1-1, to=2-1]
\arrow["{\hat p_B}", from=1-2, to=2-2]
\arrow["f"', from=2-1, to=2-2]
\end{tikzcd}
\end{equation*}
\end{definition}

\begin{example}
\label{example:differential-objects-sman}
In $(\SMan,\TT)$ (Example~\ref{example:tangent-categories-sman}), differential objects are precisely the Euclidean spaces, that is, $\R^n$, for each $n$.
\end{example}

\begin{example}
\label{example:differential-objects-c-alg}
In $(\cAlg_R,\TT^\epsilon)$ (Example~\ref{example:tangent-categories-c-alg}), the only differential object is the terminal object, that is, the zero algebra $\{0\}$ (\cite[Corollary~3.12]{cruttwell:algebraic-geometry}).
\end{example}

\begin{example}
\label{example:differential-objects-affine}
In $(\Aff_R,\TT^\Omega)$ (Example~\ref{example:tangent-categories-affine}), differential objects are equivalent to $R$-modules. In particular, there is an equivalence between $\DO[\Aff_R]$ and $\Mod_R^\op$ (\cite[Corollary~4.18]{cruttwell:algebraic-geometry}\footnote{In the paper, they look at the case in which $R=\Z$ but then, in Remark~4.20 they clarify that the same proof applies for the arbitrary case with $R$ any commutative and unital ring.}).
\end{example}

We will make recurrent use of an equivalent presentation for differential objects, established by~\cite[Proposition~3.4]{cockett:differential-bundles}. According to this presentation, a differential object consists of a commutative monoid $(A,0_A,+_A)$ equipped with a \textbf{vertical lift} $\lambda_A\colon A\to\T A$, compatible with the additive structure and the vertical lift, and subject to the following condition. The morphism
\begin{align*}
&\hat\xi_A\colon A\times A\xrightarrow{z_A\times\lambda_A}\T(A\times A)\xrightarrow{\T+_A}\T A
\end{align*}
is an isomorphism. In particular, $\lambda_A$ is uniquely defined as the morphism that satisfies $\lambda_Ap_A=!_A0_A$ and $\lambda_A\hat p_A=\id_A$. Furthermore, $\lambda_A$ satisfies the following identity:
\begin{align}
\label{equation:lambda-hat-p}
&\hat p_A=\T\lambda_Ac_A\T\hat p_A
\end{align}

By applying $\T_\e$ to the structure maps of a group object $G$, we deduce that $\T_\e G$ comes equipped with a group structure. Cockett and Schwarz showed in~\cite[Theorem~2]{cockett:lie-groups-tangent-cats} that the unit and the multiplication of this group structure on $\T_\e G$ coincide with the zero $0_{\T_\e G}$ and the sum $+_{\T_\e G}$ of the commutative monoid structure of $\T_\e G$, as a differential object.

\begin{lemma}[{\cite[Theorem~2]{cockett:lie-groups-tangent-cats}}]
\label{lemma:additive-structure-tangent-spaces}
A Lie group object $G$ admits negatives. Moreover, the group structure of $\T_\e G$ coincides with its additive structure induced by $\T_\e G$ being a differential object.
\end{lemma}

\begin{remark}
\label{remark:negatives-group-objects}
While it is immediate for Lie groups, in general one should not expect a group object to have negatives. In the following, we will always assume that a group object has negatives.
\end{remark}

\begin{theorem}
\label{theorem:Grp-linear-GCDC}
The Cartesian tangent category $\LieGr(\X,\TT)$ of Lie group objects is a linear GCDC whose linear assignment is given by the Cartesian endofunctor $\T_\e$ of Lemma~\ref{lemma:tangent-space-linear-assignment}. Furthermore, $\LieGr(\X,\TT)$ admits negatives.
\end{theorem}
\begin{proof}
Theorem~\ref{theorem:trivial-tangent-bundle-groups} shows that $\T G\cong G\times\T_\e G$ and Lemma~\ref{lemma:tangent-space-linear-assignment} proves that $\T_\e$ is functorial and Cartesian. It remains to prove that $\T_\e$ has the structure of a linear assignment. For starters, the additive structure of $\T_\e G$ is given by the additive structure of differential objects. Since $\T_\e G$ is a tangent space, it inherits such a structure. However, by Lemma~\ref{lemma:additive-structure-tangent-spaces}, the additive structure coincides with the group structure of $\T_\e G$. Therefore, it lifts to $\LieGr(\X,\TT)$. Finally, to construct the isomorphism $\nu_G\colon\T_\e G\to\T_\e\T_\e G$, we use again that $\T_\e G$ is a differential object and that the tangent space at zero of a differential object is isomorphic to itself, since $\T\T_\e G\cong\T_\e G\times\T_\e G$. Negatives are induced by the inverse map $\T_\e\i_G$.
\end{proof}

Thanks to~\cite[Theorem~4.13]{ikonicoff:abelianization-tangent-cats}, we can also classify differential objects of $\LieGr(\X,\TT)$ as \textbf{linear algebras}. A linear algebra consists of an object $A$ together with an isomorphism $\upsilon_A\colon A\to\L A$ such that $\L(\upsilon_A)=\nu_A$ (\cite[Definition~4.1]{ikonicoff:abelianization-tangent-cats}).

Let $\DO(\X,\TT)$ denote the Cartesian differential category of differential objects and linear morphisms of $(\X,\TT)$.

\begin{corollary}
\label{corollary:differential-objects-group-objects}
The category of differential objects $\DO(\LieGr(\X,\TT))$ of $\LieGr(\X,\TT)$ is equivalent to the category of $\T_\e$-algebras. In particular, a group object $A$ is a differential object in $\LieGr(\X,\TT)$ if and only if $\upsilon_A\colon A\cong\T_\e A$ and $\T_\e\upsilon_A=\nu_A$.
\end{corollary}


\section{The Lie functor of a group object}
\label{section:lie-functors}
The first step to construct the internal Lie algebra of a Lie group is to associate to each group object $G$ (with negatives) a \emph{Lie functor}, that is, a functor $\g(-)\colon(\X,\TT)\to\LieAlg^\op$. In Section~\ref{section:representability}, we will show that the Lie functor of a Lie group object is always representable. To construct this functor, we rely on a fundamental result of tangent category theory: the set $\VF_M(\X,\TT)$ of vector fields over an object $M$ forms a Lie algebra. This result was initially mentioned by Rosick\'y in his seminal paper~\cite{rosicky:tangent-cats}. Later, a complete proof was given by Cockett and Cruttwell in~\cite{cockett:jacobi}.

We begin by recalling the notion of a vector field in a tangent category.

\begin{definition}[{\cite{rosicky:tangent-cats}}]
\label{definition:vector-field}
A \textbf{vector field} in a tangent category over an object $M$ consists of a section $X\colon M\to\T M$ of the projection, in that $Xp_M=\id_M$.
\end{definition}

The zero of the Lie algebra $\VF_M(\X,\TT)$ is the zero vector field $0_M\=z_M\colon M\to\T M$. Given two vector fields $X,Y\colon M\to\T M$, their sum is defined by:
\begin{align*}
&X+Y\colon M\xrightarrow{\<X,Y\>}\T_2M\xrightarrow{s_M}\T M
\end{align*}
When $M$ has negatives, $\VF_M(\X,\TT)$ becomes a Lie algebra. The following formula defines the Lie bracket of two vector fields $X,Y$:
\begin{align*}
&[X,Y]\=\{X\T Y-Y\T Xc_M\}
\end{align*}
where $-X\=Xn_M$ and for a given morphism $f\colon X\to\T\T M$ satisfying $f\T p_M=fp_{\T M}p_Mz_M$, $\{f\}\colon X\to\T M$ is the unique morphism, defined by the universal property of the vertical lift, such that $f=\<fp_{\T M}z_{\T M},\{f\}l_M\>\T s_M$ (see~\cite[Lemma~2.12]{cockett:tangent-cats}).

Throughout the section, we shall assume that group objects have negatives (see Lemma~\ref{lemma:tangent-space-linear-assignment} and Remark~\ref{remark:negatives-group-objects}).

\subsection{Parametrized vector fields}
\label{subsection:parametrized-vector-fields}
The first step towards constructing the Lie functor $\g(-)$ of a group object with negatives $G$ is to introduce parametrized vector fields.

\begin{definition}
\label{definition:parametrized-vector-fields}
For a given object $A$, an \textbf{$A$-parametrized vector field} in a Cartesian tangent category consists a morphism $X\colon M\times A\to\T M$ satisfying the following commutativity condition:
\begin{equation*}
\begin{tikzcd}
{M\times A} & {\T M} \\
& M
\arrow["X", from=1-1, to=1-2]
\arrow["{\pi_1}"', from=1-1, to=2-2]
\arrow["{p_M}", from=1-2, to=2-2]
\end{tikzcd}
\end{equation*}
\end{definition}

Intuitively, a parametrized vector field can be seen as a family of vector fields $X_a\colon M\to\T M$, parametrized by the parameters $a\in A$. We shall prove that parametrized vector fields are in fact proper vector fields in another tangent category. Since parametrized vector fields play an important role in the story of this paper, we spend a bit of time exploring this characterization.

It is a well-known fact of category theory that the objects $A$ of a Cartesian category $\X$ come equipped with a cocommutative comonoid whose counit is given by the terminal map $!_A\colon A\to\*$ and whose comultiplication is the diagonal map $\Delta_A\colon A\to A\times A$. It is also well-known that each comonoid $A$ induces a comonad $S_A\=(-)\times A$ on $\X$. We may refer to $S_A$ as the \textbf{simple comonad} of $A$. The coKleisli category $\KL[S_A]$ of $S_A$ is the category defined as follows. The objects of $\KL[S_A]$ are all the objects of $\X$, however, a morphism $f\colon M\nto N$ in $\KL[S_A]$ corresponds to a morphism $\SQ{f}\colon M\times A\to N$. The identity $\id_M$ is the projection $\pi_1\colon M\times A\to M$ and composition of two morphisms $f\colon M\nto N$ and $g\colon N\nto P$ is defined by the coKleisli formula:
\begin{align*}
\SQ{fg}&\colon M\times A\xrightarrow{\id_M\times\Delta_A}M\times A\times A\xrightarrow{\SQ{f}\times\id_A}N\times A\xrightarrow{\SQ g}P
\end{align*}

In~\cite[Definition~19]{cockett:tangent-monads}, Cockett, Lemay, and Lucyshyn-Wright introduced tangent monads, which are monads $S$ over a tangent category $(\X,\TT)$ together with a lax distributive law $\alpha_M\colon S\T M\to\T SM$, compatible with the monad structure and the tangent structure. This turns out to be precisely the same as a monad in the $2$-category $\TngCat$ of tangent categories, lax tangent morphisms, and tangent natural transformations.

Similarly, one can introduce tangent comonads.

\begin{definition}
\label{definition:lax-tangent-comonad}
A (\textbf{lax}) \textbf{tangent comonad} is a comonad in the $2$-category $\TngCat$.
\end{definition}

By unwrapping this definition, a tangent comonad consists of a comonad $S$ with counit denoted by $\epsilon_M\colon SM\to M$ and comultiplication by $\delta_M\colon SM\to SSM$, together with a lax distributive law $\alpha_M\colon S\T M\to\T SM$, compatible with the comonad structure
\begin{equation*}
\begin{tikzcd}
{S\T M} & {\T SM} \\
& {\T M}
\arrow["{\alpha_M}", from=1-1, to=1-2]
\arrow["{\epsilon_{\T M}}"', from=1-1, to=2-2]
\arrow["{\T\epsilon_M}", from=1-2, to=2-2]
\end{tikzcd}\qquad
\begin{tikzcd}
{S\T M} && {\T SM} \\
{SS\T M} & {S\T SM} & {\T SSM}
\arrow["{\alpha_M}", from=1-1, to=1-3]
\arrow["{\delta_{\T M}}"', from=1-1, to=2-1]
\arrow["{\T\delta_M}", from=1-3, to=2-3]
\arrow["{S\alpha_M}"', from=2-1, to=2-2]
\arrow["{\alpha_{SM}}"', from=2-2, to=2-3]
\end{tikzcd}
\end{equation*}
and the tangent structure:
\begin{equation*}
\begin{tikzcd}
{S\T M} & {\T SM} \\
& SM
\arrow["{\alpha_M}", from=1-1, to=1-2]
\arrow["{Sp_M}"', from=1-1, to=2-2]
\arrow["{p_{SM}}", from=1-2, to=2-2]
\end{tikzcd}\qquad
\begin{tikzcd}
{S\T M} & {\T SM} \\
& SM
\arrow["{\alpha_M}", from=1-1, to=1-2]
\arrow["{Sz_M}", from=2-2, to=1-1]
\arrow["{z_{SM}}"', from=2-2, to=1-2]
\end{tikzcd}\qquad
\begin{tikzcd}
{S\T_2M} & {\T_2SM} \\
{S\T M} & {\T SM}
\arrow["{\<\pi_1\alpha_M,\pi_2\alpha_M\>}", from=1-1, to=1-2]
\arrow["{Ss_M}"', from=1-1, to=2-1]
\arrow["{s_{SM}}", from=1-2, to=2-2]
\arrow["{\alpha_M}"', from=2-1, to=2-2]
\end{tikzcd}
\end{equation*}
\begin{equation*}
\begin{tikzcd}
{S\T M} && {\T SM} \\
{S\T\T M} & {\T S\T M} & {\T\T SM}
\arrow["{\alpha_M}", from=1-1, to=1-3]
\arrow["{Sl_M}"', from=1-1, to=2-1]
\arrow["{l_{SM}}", from=1-3, to=2-3]
\arrow["{\alpha_{\T M}}"', from=2-1, to=2-2]
\arrow["{\T\alpha_M}"', from=2-2, to=2-3]
\end{tikzcd}\qquad
\begin{tikzcd}
{S\T\T M} & {\T S\T M} & {\T\T SM} \\
{S\T\T M} & {\T S\T M} & {\T\T SM}
\arrow["{\alpha_{\T M}}", from=1-1, to=1-2]
\arrow["{Sc_M}"', from=1-1, to=2-1]
\arrow["{\T\alpha_M}", from=1-2, to=1-3]
\arrow["{c_{SM}}", from=1-3, to=2-3]
\arrow["{\alpha_{\T M}}"', from=2-1, to=2-2]
\arrow["{\T\alpha_M}"', from=2-2, to=2-3]
\end{tikzcd}
\end{equation*}

\begin{lemma}
\label{lemma:simple-comonad}
The simple comonad $S_A$ of an object $A$ in a Cartesian tangent category $(\X,\TT)$ becomes a tangent comonad over $(\X,\TT)$.
\end{lemma}
\begin{proof}
To make $S_A$ into a tangent comonad, we need to provide a distributive law between $S_A$ and the tangent bundle functor. This can be done with the zero morphism $z_A\colon A\to\T A$ as follows:
\begin{align*}
&(\alpha_A)_M\colon S_A(\T M)=\T M\times A\xrightarrow{\id_{\T M}\times z_A}\T(M\times A)=\T S_A(M)
\end{align*}
It is not hard to show that $\alpha_A$ is in fact compatible with the comonad structure of $S_A$ and with the tangent structure.
\end{proof}

As shown by~\cite[Proposition~20]{cockett:tangent-monads}, the Eilenberg-Moore category of a tangent monad comes equipped with a tangent structure strictly preserved by the forgetful functor. In~\cite[Corollary~4.6]{lanfranchi:tangentads-I}, the author also showed that this tangent category is in fact the Eilenberg-Moore object of a tangent monad seen as a monad internal to $\TngCat$.

Now, we prove that the coKleisli category of a tangent comonad is also a tangent category. Unfortunately, this result does not come from a formal argument of dualization, since the dual of a tangent category fails, in general, to be a tangent category.

\begin{proposition}
\label{proposition:cokleisli-tangent-category}
The coKleisli category $\KL[S,\alpha]$ of a tangent comonad $(S,\alpha)$ comes equipped with a tangent structure. Moreover, the inclusion functor $(\X,\TT)\to\KL[S,\alpha]$ becomes a strict tangent morphism.
\end{proposition}
\begin{proof}
Define the tangent bundle functor $\T_S\colon\KL[S]\to\KL[S]$ as the functor that sends an object $M$ to $\T M$ and a morphism $f\colon M\nto N$ to:
\begin{align*}
&\SQ{\T_Sf}\colon S\T M\xrightarrow{\alpha_M}\T SM\xrightarrow{\T\SQ f}\T N
\end{align*}
The structural natural transformations are defined by applying the inclusion functor $\X\to\KL[S]$ to their components. Therefore, by functoriality of this inclusion functor, all the equational axioms of a tangent category automatically hold. Finally, since the inclusion functor is a right adjoint, it preserves limits. This proves that also the universal axioms are satisfied.
\end{proof}

Thanks to Proposition~\ref{proposition:cokleisli-tangent-category}, the base tangent structure lifts to the coKleisli category of the simple tangent comonad $(S_A,\alpha_A)$. We may denote this tangent category by $\Simple_A[\X,\TT]$. This notation will be clear in a moment.

\begin{proposition}
\label{proposition:parametrized-vector-fields}
$A$-parametrized vector fields are equivalent to vector fields in $\Simple_A[\X,\TT]$.
\end{proposition}
\begin{proof}
A vector field in $\Simple_A[\X,\TT]$ on an object $M$ consists of a coKleisli morphism $X\colon M\nto\T M$, such that, $Xp_M=\id_{M}$ in $\Simple_A[\X,\TT]$. This data corresponds to a morphism $\SQ X\colon M\times A\to\T M$, such that:
\begin{align*}
&\pi_1=(\id_M\times\Delta_A)(\SQ X\times\id_A)\pi_1p_M=\SQ Xp_M
\end{align*}
However, this is precisely the condition satisfied by an $A$-parametrized vector field.
\end{proof}

There is an alternative way to construct the tangent category $\Simple_A[\X,\TT]$, which is worth mentioning. For starters, recall that the \textbf{simple fibration} of a Cartesian category $\X$ consists of a category $\Simple[\X]$ whose objects are pairs $(A,M)$ of objects of $\X$ and morphisms $(f,\varphi)\colon(A,M)\to(B,N)$ are pairs of morphisms $f\colon A\to B$ and $\varphi\colon M\times A\to N$ of $\X$. Identities are given by $(\id_A,\pi_1)$ and composition is similar to the coKleisli composition. The fibration functor sends a pair $(A,M)$ to $A$ and a morphism $(f,\varphi)$ to $f$.

It turns out that the simple fibration of a Cartesian tangent category is a tangent fibration. The subject of tangent fibrations is not central to this paper, so we leave it to the reader to consult~\cite[Section~5]{cockett:differential-bundles} for details. In a nutshell, a (cloven) tangent fibration consists of a (cloven) fibration $\Pi\colon(\X',\TT')\to(\X,\TT)$ between two tangent categories, whose underlying functor is a strict tangent morphism and such that the tangent bundle functors $\T,\T'$ preserve the (cloven) Cartesian lifts.

\begin{proposition}
\label{proposition:simple-tangent-fibration}
The simple fibration of a Cartesian tangent category is a cloven tangent fibration.
\end{proposition}
\begin{proof}
We begin by turning the total category $\Simple[\X,\TT]$ of the simple fibration into a tangent category. The tangent bundle functor $\T^\Simple$ sends a pair $(A,M)$ to $(\T A,\T M)$ and a morphism $(f,\varphi)$ to $(\T f,\T\varphi)$. The structural natural transformations are constructed by using the same morphism in the first component and by precomposing with $\pi_1$ for the second one, e.g., the projection is defined by $(p_A,\pi_1p_M)$. The proof that this defines a tangent structure is a straightforward exercise that we leave to the reader to spell out. It is also immediate to see that the fibration preserves the tangent structure strictly. To conclude, it is left to show that the tangent bundle functors preserve the cloven Cartesian lifts. First, recall that the cloven Cartesian lift of a morphism $f\colon A\to B$ along an object $(B,N)$ over $B$ is given by the morphism $(f,\pi_1)\colon(A,N)\to(B,N)$. From this, using that the tangent bundle functor preserves finite products, it follows immediately that $\T^\Simple(f,\pi_1)=(\T f,\T\pi_1)=(\T f,\pi_1)$.
\end{proof}

\begin{remark}
\label{remark:simple-tangent-fibration}
In~\cite[Example~5.6(i)]{cockett:differential-bundles}, Cockett and Cruttwell already showed that the simple fibration of a Cartesian differential category is a tangent fibration. Proposition~\ref{proposition:simple-tangent-fibration} extends this result to any Cartesian tangent category.
\end{remark}

Cockett and Cruttwell's~\cite[Theorem~5.3]{cockett:differential-bundles} proved that the fibres of a (cloven) tangent fibration are in fact tangent categories, whose tangent bundle functor is constructed by pulling back the tangent bundle functor of the total tangent category along the zero morphism $z_M\colon M\to\T M$ of the base.

\begin{proposition}
\label{proposition:simple-fibration-fibres}
The coKleisli tangent category $\Simple_A[\X,\TT]$ of the simple tangent comonad $(S_A,\alpha_A)$ is isomorphic to the fibre tangent category of the simple tangent fibration $\Pi\colon\Simple[\X,\TT]\to(\X,\TT)$ over $A$.
\end{proposition}
\begin{proof}
This is a simple exercise. In particular, notice that the tangent bundle functor on the fibre is obtained by pulling back the total tangent bundle functor along the zero morphism. The zero in fact appears in the distributive law $\alpha_A$ of the tangent comonad $S_A$. We leave it to the reader to complete the details of this comparison.
\end{proof}

Thanks to Proposition~\ref{proposition:parametrized-vector-fields}, parametrized vector fields over an object $M$ whose tangent bundle $p_M$ has negatives form a Lie algebra.

\begin{corollary}
\label{corollary:parametrized-lie-algebra}
If $M$ admits negatives, $A$-parametrized vector fields over $M$ form a Lie algebra $\VF^A_M(\X,\TT)$.
\end{corollary}

\subsection{Left-invariant vector fields}
\label{subsection:left-invariant-vector-fields}
The next step of our construction is to consider left-invariant vector fields. These are vector fields over a group object $G$ which are invariant under the left action of the group. They were previously introduced and studied by Cockett and Schwarz in the context of tangent categories in~\cite[Section~4.4]{cockett:lie-groups-tangent-cats}. In this section, we briefly recall this notion and prove some results.

\begin{definition}[{\cite[Section~4.4]{cockett:lie-groups-tangent-cats}}]
\label{definition:left-invariant-vector-fields}
A \textbf{left-invariant vector field} on a group object $G$ in a Cartesian tangent category consists of a vector field $X\colon G\to\T G$ subject to the following condition:
\begin{equation*}
\begin{tikzcd}
{G\times G} & G \\
{\T(G\times G)} & {\T G}
\arrow["{\m_G}", from=1-1, to=1-2]
\arrow["{z_G\times X}"', from=1-1, to=2-1]
\arrow["X", from=1-2, to=2-2]
\arrow["{\T\m_G}"', from=2-1, to=2-2]
\end{tikzcd}
\end{equation*}
\end{definition}

An important property of left-invariant vector fields is that they are fully determined by their value at the unit. This statement is made precise by the next lemma that was already proved by Cockett and Schwarz.

\begin{lemma}[{\cite[Proposition~8]{cockett:lie-groups-tangent-cats}}]
\label{lemma:correspondence-livf-global-elements}
There is a bijective correspondence between left-invariant vector fields of a Lie group object $G$ and global elements $\*\to\T_\e G$ of the tangent space of $G$ at the unit, as follows:
\begin{enumerate}
\item Each global element $x\colon\*\to\T_\e G$ corresponds to the left-invariant vector field:
\begin{align*}
&X_x\colon G\xrightarrow{\<\id_G,x\>}G\times\T_\e G\cong\T G
\end{align*}

\item Each left-invariant vector field $X\colon G\to\T G$ corresponds to the global element $X_\e$ defined by:
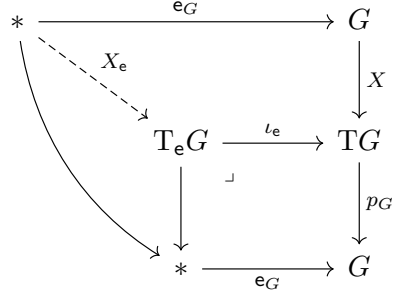
\begin{equation*}
\begin{tikzcd}
{\*} && G \\
& {\T_\e G} & {\T G} \\
& {\*} & G
\arrow["{\e_G}", from=1-1, to=1-3]
\arrow["{X_\e}", dashed, from=1-1, to=2-2]
\arrow[curve={height=18pt}, from=1-1, to=3-2]
\arrow["X", from=1-3, to=2-3]
\arrow["{\iota_\e}", from=2-2, to=2-3]
\arrow[from=2-2, to=3-2]
\arrow["\lrcorner"{anchor=center, pos=0.125}, draw=none, from=2-2, to=3-3]
\arrow["{p_G}", from=2-3, to=3-3]
\arrow["{\e_G}"', from=3-2, to=3-3]
\end{tikzcd}
\end{equation*}
\end{enumerate}
\end{lemma}
\begin{proof}
For completeness, we give an explicit proof of this statement. We start by considering a global element $x\colon\*\to\T_\e G$. So, the corresponding vector field is $X_x\=\<\id_G,x\>\gamma_G$, where $\gamma_G\colon G\times\T_\e G\cong\T G$ is defined by the following formula:
\begin{align*}
&\gamma_G\colon G\times\T_\e G\xrightarrow{z_G\times\iota_\e}\T(G\times G)\xrightarrow{\T\m_G}\T G
\end{align*}
We shall prove that $X_x$ is in fact a left-invariant vector field of $G$. For starters, let us show that $X_x$ is a vector field:
\begin{align*}
X_xp_G&=~\<\id_G,x\>\gamma_Gp_G                                \tag{Definition of $X_x$}\\
&=~\<\id_G,x\>(z_G\times\iota_\e)\T\m_Gp_G                   \tag{Definition of $\gamma_G$}\\
&=~\<\id_G,x\>(z_Gp_G\times\iota_\e p_G)\m_G                  \tag{Naturality of $p$}\\
&=~\<\id_G,x\>(\id_G\times!\e)\m_G                        \tag{$z_Gp_G=\id_G$, $\iota_\e p_G=!\e$}\\
&=~\<\id_G,!\e\>\m_G\\
&=~\id_G                                                    \tag{Unitality}
\end{align*}
Now, let us show that $X_x$ is left-invariant:
\begin{align*}
(z_G\times X_x)\T\m_G&=~(z_G\times\<z_G,x\iota_\e\>\T\m_G)\T\m_G     \tag{Definition of $X_x$}\\
&=~\<(z_G\times z_G)\T\m_G,x\iota_\e\>\T\m_G                        \tag{Associativity}\\
&=~\<\m_Gz_G,x\iota_\e\>\T\m_G                                       \tag{Naturality of $z$}\\
&=~\m_G\<z_G,x\iota_\e\>\T\m_G\\
&=~\m_GX_x                                                          \tag{Definition of $X_x$}
\end{align*}
This proves that $X_x$ is a left-invariant vector field. Now, let us prove that the corresponding global element $(X_x)_\e$ is again $x$. To see this, let us compute:
\begin{align*}
\e X_x&=~\e\<z_G,x\iota_\e\>\T\m_G                                \tag{Definition of $X_x$}\\
&=~\<z_G\T\e,x\iota_\e\>\T\m_G                                     \tag{Naturality of $z$}\\
&=~\<z_G,x\iota_\e\>\T((\e\times\id_G)\m_G)\\
&=~\<z_G,x\iota_\e\>\pi_2                                            \tag{Unitality}\\
&=~x\iota_\e
\end{align*}
Therefore, by definition, $(X_x)_\e=x$. Conversely, consider a left-invariant vector field $X\colon G\to\T G$. The corresponding global element $X_\e$ is uniquely determined by the equation $X_\e\iota_\e=\e X$. So, the corresponding left-invariant vector field is $X_{X_\e}$ defined as follows:
\begin{align*}
X_{X_\e}&=~\<z_G,X_\e\iota_\e\>\T\m_G                            \tag{Definition of $X_{X_\e}$}\\
&=~\<z_G,\e X\>\T\m_G                                           \tag{Definition of $X_\e$}\\
&=~\<\id_G,!\e\>(z_G\times X)\T\m_G\\
&=~\<\id_G,!\e\>\m_GX                                           \tag{Left-invariance}\\
&=~X                                                            \tag{Unitality}
\end{align*}
This proves the statement.
\end{proof}

By Lemma~\ref{lemma:additive-structure-tangent-spaces}, a group object $G$ always admits negatives. Therefore, vector fields over a group object form a Lie algebra $\VF_G(\X,\TT)$. Cockett and Schwarz showed that left-invariant vector fields form a Lie subalgebra of the Lie algebra of vector fields over a group object.

\begin{proposition}[{\cite[Proposition~9 and Lemma~8]{cockett:lie-groups-tangent-cats}}]
\label{proposition:left-invariant-vector-fields}
Left-invariant vector fields of a group object $G$ form a Lie subalgebra $\LIVF_G(\X,\TT)$ of the Lie algebra $\VF_G(\X,\TT)$ of vector fields over $G$.
\end{proposition}

Cockett and Schwarz proved this result directly. Instead, we would like to give a more structural perspective. In particular, we shall prove that left-invariant vector fields are proper vector fields in another tangent category. To this end, fix a group object $G$ in a Cartesian tangent category. For each object $A$ of a Cartesian tangent category, we have previously constructed an associated tangent comonad $(S_A,\alpha_A)$. Therefore, in particular, we can consider $(S_G,\alpha_G)$. Consider also the following map:
\begin{align*}
&\theta_G\colon G\times G\xrightarrow{\tau}G\times G\xrightarrow{\m_G}G
\end{align*}
where $\tau$ is the canonical symmetry. $\theta_G$ can be regarded as a coKleisli map $\theta_G\colon G\nto G$, that is, a morphism in $\Simple_G[\X,\TT]$. Now, consider a vector field $X\colon G\to\T G$ in the base tangent category. By Proposition~\ref{proposition:cokleisli-tangent-category}, the inclusion functor $(\X,\TT)\to\Simple_G[\X,\TT]$ is a strict tangent morphism. Therefore, it sends vector fields to vector fields. Let $\dot X\colon G\nto\T_GG$ denote the corresponding vector field in $\Simple_G[\X,\TT]$. Concretely, $\SQ{\dot X}=\pi_1X$.

Now that we have made both the vector field $X$ and the multiplication map $\m_G$ into coKleisli maps, we can state the left-invariance condition as a commutative diagram in $\Simple_G[\X,\TT]$.

\begin{lemma}
\label{lemma:left-invariance-cokleisli}
A vector field $X\colon G\to\T G$ on a group object $G$ is left-invariant if and only if the lifted vector field $\dot X$ in $\Simple_G[\X,\TT]$ commutes with $\theta_G$ as follows:
\begin{equation*}
\begin{tikzcd}
G & G \\
{\T_GG} & {\T_GG}
\arrow["{\theta_G}"{inner sep=.8ex}, "\shortmid"{marking}, from=1-1, to=1-2]
\arrow["{\dot X}"'{inner sep=.8ex}, "\shortmid"{marking}, from=1-1, to=2-1]
\arrow["{\dot X}"{inner sep=.8ex}, "\shortmid"{marking}, from=1-2, to=2-2]
\arrow["{\T_G\theta_G}"'{inner sep=.8ex}, "\shortmid"{marking}, from=2-1, to=2-2]
\end{tikzcd}
\end{equation*}
\end{lemma}
\begin{proof}
In the base category $\X$, the commutativity condition expressed by the above diagram reads as follows:
\begin{align*}
&(\id_G\times\Delta_G)(\tau\times\id_G)(\m_G\times\id_G)(\pi_1X)=(\id_G\times\Delta)(\pi_1X\times\id_G)(\id_{\T G}\times z_G)\T\tau\T\m_G
\end{align*}
However, by rearranging both the left and the right terms, we obtain the equation:
\begin{align*}
&\tau\m_GX=(X\times z_G)\T\tau\T\m_G=\tau(z_G\times X)\T\m_G
\end{align*}
However, since $\tau$ is a symmetry, this equation is equivalent to $\m_GX=(z_G\times X)\T\m_G$, that is, the left-invariance condition of $X$.
\end{proof}

This lemma allows us to construct a tangent category whose vector fields over a given object are precisely the left-invariant vector fields of a group object. To this end, recall that for each tangent category $(\X,\TT)$, the arrow category $\X^\to$ comes with a tangent structure, obtained by applying the base tangent structure pointwise. Let us denote this tangent category by $(\X,\TT)^\to$. There is a tangent subcategory $\End(\X,\TT)$ of $(\X,\TT)^\to$ whose objects are pairs $(M,m)$ formed by an object $M$ of $(\X,\TT)$ and an endomorphim $m\colon M\to M$ and whose morphisms $f\colon(M,m)\to(N,n)$ are morphisms that commute with the endomorphisms, that is, $fn=mf$. Notice that, $(G,\theta_G)$ becomes an object of $\End(\Simple_G[\X,\TT])$.

Another way to construct $\End(\X,\TT)$ is to look at inserters in the $2$-category $\TngCat_{\cong}$ of tangent categories, strong tangent morphisms, and tangent natural transformations. Then, $\End(\X,\TT)$ is precisely the inserter between the identity tangent morphisms over $(\X,\TT)$. We suggest the reader consult~\cite[Section~4]{lanfranchi:tangentads-II} for details on inserters and equifiers in the $2$-category $\TngCat_{\cong}$.

Now, we do not want to consider every morphism of $\End(\Simple_G[\X,\TT])$, but only those morphisms that come from $(\X,\TT)$. So, we consider the subcategory denoted by $\LI_G[\X,\TT]$ of $\End(\Simple_G[\X,\TT])$ with the same objects but morphisms are those morphisms $f\colon(M,m)\nto(N,n)$ whose underlying morphism $f\colon M\nto N$ of $\Simple_G[\X,\TT]$ is given by $\SQ f=\pi_1f$.

\begin{proposition}
\label{proposition:left-invariant-vector-fields-characterization}
Given a group object $G$ in $(\X,\TT)$, there is a tangent category $\LI_G[\X,\TT]$ with an object $(G,\theta_G)$, whose vector fields over $(G,\theta_G)$ are precisely all the left-invariant vector fields of $G$ in $(\X,\TT)$.
\end{proposition}
\begin{proof}
For starters, since the tangent bundle functor preserves Cartesian products, if $\SQ f=\pi_1f$, then $\SQ{\T_Gf}=(\id_{\T M}\times z_G)\T(\pi_1f)=\pi_1\T f$. Therefore, $\LI_G[\X,\TT]$ becomes a tangent subcategory of $\End(\Simple_G[\X,\TT])$. By Lemma~\ref{lemma:left-invariance-cokleisli}, a map $X\colon G\to\T G$ is a left-invariant vector field of $G$ if and only if the corresponding vector field $\dot X$ in $\Simple_G[\X,\TT]$ commutes with $\theta_G$. However, if $M$ is an object in a tangent category $(\X,\TT)$, $X$ is a vector field of $M$, and $m\colon M\to M$ is an endomorphism, then $X$ becomes a vector field $X\colon(M,m)\to\T(M,m)$ in $\End(\X,\TT)$ if and only if $X$ commutes with $m$, that is, if $mX=X\T m$. Thus, $X$ is a left-invariant vector field if and only if $\dot X$ becomes a vector field of $(G,\theta_G)$ in $\End(\Simple_G[\X,\TT])$. However, a vector field $Y$ of this tangent category is equal to $\dot X$ for some $X\colon G\to\T G$ if and only if its interpretation $\SQ{Y}=\pi_1X$.
\end{proof}

We can now give a structural proof of Proposition~\ref{proposition:left-invariant-vector-fields}.

\begin{proof}[Proof of Proposition~\ref{proposition:left-invariant-vector-fields}]
There is a forgetful functor $\LI_G[\X,\TT]\to(\X,\TT)$ that sends an object $(M,m)$ to $M$ and each morphism $\SQ{f}=\pi_1f$ to $f=\<\id_H,\e_G\>\SQ{f}$. Furthermore, this functor preserves the tangent structure strictly. Thus, it sends vector fields of $\LI_G[\X,\TT]$ over $(M,m)$ to vector fields of $(\X,\TT)$ over $M$, and it preserves the Lie algebra structure. So, in particular, it sends vector fields over $(G,\theta_G)$ to vector fields over $G$. However, by Proposition~\ref{proposition:left-invariant-vector-fields-characterization}, vector fields over $(G,\theta_G)$ in $\LI_G[\X,\TT]$ are precisely left-invariant vector fields which form the Lie algebra $\LIVF_G(\X,\TT)$. Thus, this functor becomes an inclusion of $\LIVF_G(\X,\TT)$ into the Lie algebra $\VF_G(\X,\TT)$ of ordinary vector fields of $G$.
\end{proof}

\begin{proposition}
\label{proposition:left-invariant-vector-fields-tangent-morphisms}
A lax tangent morphism $(F,\alpha)\colon(\X,\TT)\to(\X',\TT')$ of Cartesian tangent categories that preserves finite products preserves left-invariant vector fields.
\end{proposition}
\begin{proof}
If $G$ is a group object of $(\X,\TT)$ and $F$ preserves finite products, then $FG$ becomes a group object in $(\X',\TT')$. Furthermore, using the coKleisli construction, we can lift $(F,\alpha)$ to a lax tangent morphism $\KL[F,\alpha]\colon\Simple_G[\X,\TT]\to\KL[S_{FG},\alpha_{FG}]$. Moreover, it is also easy to see that $\KL[F,\alpha]$ commutes with the inclusion functors from the base tangent categories to the coKleisli. This implies that if $X\colon G\to\T G$ is a vector field in $(\X,\TT)$, then $\KL[F,\alpha](\dot X)=\dot{FX}$. Next, we can apply the $2$-functor $\End$ and obtain another lax tangent morphism $\End(\KL[F,\alpha])\colon\End(\Simple_G[\X,\TT])\to\End(\KL[S_{FG},\alpha_{FG}])$. However, since lax tangent morphisms preserve vector fields, $\End(\KL[F,\alpha])$ sends vector fields over $(G,\theta_G)$ to vector fields over $(FG,\theta_{FG})$. Finally, since $\KL[F,\alpha](\dot X)=\dot{FX}$ for each vector field $X$, we conclude that $(F,\alpha)$ preserves left-invariant vector fields.
\end{proof}

\subsection{The Lie functor of a group object}
\label{subsection:lie-functor}
It is time to construct the Lie functor of a group object. To this end, consider a group object $G$ which admits negatives in a Cartesian tangent category $(\X,\TT)$. Notice that $G$ is not assumed to have a tangent space at the unit. Given an object $A$ of $(\X,\TT)$, we can consider the coKleisli tangent category $\Simple_A[\X,\TT]$.

\begin{lemma}
\label{lemma:group-objects-in-cokleisli}
Every group object $G$ becomes a group object in the coKleisli category $\Simple_A[\X,\TT]$.
\end{lemma}
\begin{proof}
Since the inclusion functor preserves limits, it sends group objects to group objects.
\end{proof}

We can use this simple result to define the next concept.

\begin{definition}
\label{definition:parametrized-left-invariant-vector-field}
An \textbf{$A$-Parametrized Left-Invariant Vector Field} (\textbf{$A$-\PLIVF} for short) on a group object $G$ consists of a left-invariant vector field over $G$ in $\Simple_A[\X,\TT]$.
\end{definition}

By unpacking this definition, an $A$-\PLIVF over $G$ consists of a map $X\colon G\times A\to\T G$ subject to the following two conditions:
\begin{equation*}
\begin{tikzcd}
{G\times A} & {\T G} \\
& G
\arrow["X", from=1-1, to=1-2]
\arrow["{\pi_1}"', from=1-1, to=2-2]
\arrow["{p_G}", from=1-2, to=2-2]
\end{tikzcd}\qquad
\begin{tikzcd}
{G\times G\times A} & {G\times A} \\
{\T(G\times G)} & {\T G}
\arrow["{\m_G\times\id_A}", from=1-1, to=1-2]
\arrow["{z_G\times X}"', from=1-1, to=2-1]
\arrow["X", from=1-2, to=2-2]
\arrow["{\T\m_G}"', from=2-1, to=2-2]
\end{tikzcd}
\end{equation*}

By definition, $A$-\PLIVFs are left-invariant vector fields in the coKleisli category. Therefore, by Proposition~\ref{proposition:left-invariant-vector-fields}, they form a Lie algebra. We denote this Lie algebra by $\g(A)$.

Furthermore, thanks to Proposition~\ref{proposition:left-invariant-vector-fields-tangent-morphisms}, lax tangent morphisms preserve left-invariant vector fields. Therefore, given a morphism $f\colon A\to B$, the induced lax tangent morphism $\Simple_B[\X,\TT]\to\Simple_A[\X,\TT]$ sends $B$-\PLIVFs over $G$ to $A$-\PLIVFs over $G$. The induced Lie algebra homomorphism is denoted by $\g(f)\colon\g(B)\to\g(A)$.

\begin{lemma}
\label{lemma:lie-functor}
For each group object $G$ with negatives in a Cartesian tangent category, there is a functor
\begin{align*}
&\g(-)\colon\X\to\LieAlg^\op
\end{align*}
which sends every object $A$ to the Lie algebra $\g(A)$ of $A$-\PLIVFs over $G$ and each morphism $f\colon A\to B$ to the Lie algebra homomorphism $\g(f)\colon\g(B)\to\g(A)$.
\end{lemma}

\begin{definition}
\label{definition:lie-functor}
The \textbf{Lie functor} of a group object $G$ in a Cartesian tangent category is the functor $\g(-)\colon\X\to\Set^\op$ of Lemma~\ref{lemma:lie-functor}, where we forget the Lie algebra structure.
\end{definition}

Recall that, by Proposition~\ref{proposition:simple-fibration-fibres}, the coKleisli tangent category $\Simple_A[\X,\TT]$ coincides with the fibre tangent category over $A$ of the simple tangent fibration $\Pi\colon\Simple[\X,\TT]\to(\X,\TT)$. We also defined $\g(A)$ as the Lie algebra of left-invariant vector fields over $G$ in $\Simple_A[\X,\TT]$.

It is immediate to see that a group object $G$ in the base tangent category is also a group object in $\Simple[\X,\TT]$ and that left-invariant vector fields over $G$ in $\Simple[\X,\TT]$ are precisely parametrized left-invariant vector fields over $G$. One of the advantages of this point of view is that the parameter space $A$ is now left free to change.

The tangent bundle functor $\T^\Simple\colon\Simple[\X,\TT]\to\Simple[\X,\TT]$ becomes a strong tangent morphism with the canonical flip $(\T^\Simple,c^\Simple)$ and, as we mentioned earlier, (lax and in particular strong) tangent morphisms preserve vector fields. In particular, an $A$-parametrized vector field $X\colon G\times A\to\T G$ over $G$ is sent to the $\T A$-parametrized vector field:
\begin{align*}
&X_\T\colon\T G\times\T A\xrightarrow{\cong}\T(G\times A)\xrightarrow{\T X}\T\T G\xrightarrow{c_G}\T\T G
\end{align*}
Moreover, by Proposition~\ref{proposition:left-invariant-vector-fields-tangent-morphisms}, tangent morphisms also preserve left invariance and the Lie algebra structure of vector fields. Therefore, we obtain a Lie algebra homomorphism:
\begin{align*}
&\tau_A^G\colon\g(A)\to(\T\g)(\T A)     &&(X\colon G\times A\to\T G)\mapsto(X_\T\colon\T G\times\T A\to\T\T G)
\end{align*}
where $(\T\g)(-)$ denotes the Lie functor of $\T G$.

\begin{lemma}
\label{lemma:tau-lie-algebra-homomorphism}
For a group object $G$, there is a natural transformation
\begin{align*}
&\tau_A^G\colon\g(A)\to(\T\g)(\T A)
\end{align*}
natural in $A$, whose components are Lie algebra homomorphisms.
\end{lemma}
\begin{proof}
It remains to prove that $\tau_A^G$ is natural. To this end, consider a morphism $f\colon A\to B$. We can compute:
\begin{align*}
&((\T\g)(\T f))(\tau_B^G(X))=((\T\g)(\T f))(X_\T)=(\id_{\T G}\times\T f)X_\T=\T(\id_G\times f)(\T Xc_G)=\\
&\qquad=\T((\id_G)\times f)X)c_G=\T(\g(f)(X))c_G=\tau_A^G(\g(f)(X))\qedhere
\end{align*}
\end{proof}

We use the natural transformation $\tau$ to construct an infinite sequence of Lie algebras as follows:
\begin{align*}
&\g(A)\xrightarrow{\tau_A^G}(\T\g)(\T A)\xrightarrow{\tau_{\T A}^{\T G}}(\T^2\g)(\T^2A)\to{\dots}\to(\T^n\g)(\T^nA)\xrightarrow{\tau_{\T^nA}^{\T^nG}}(\T^{n+1}\g)(\T^{n+1}G)\to{\dots}
\end{align*}
We will refer to this as the \textbf{Lie sequence} of a group object $G$ with coefficients in $A$.


\section{The representability of the Lie functor}
\label{section:representability}
Proposition~\ref{proposition:left-invariant-vector-fields} establishes that the set $\LIVF_G(\X,\TT)$ of left-invariant vector fields of a group object $G$ in a Cartesian tangent category $(\X,\TT)$ carries the structure of a Lie algebra. This construction defines the \emph{external} Lie algebra of a group object $G$, namely, a Lie algebra which, however, does not live in the same category as the group object $G$. Our goal is to \emph{internalize} this Lie algebra by constructing the \emph{internal} Lie algebra of a Lie group.

Our approach is based on the \emph{representability} of the Lie functor. Recall that a functor $F\colon\X\to\Set^\op$ is \textbf{representable} if there is an object $\mathfrak{f}$ of $\X$ and a natural isomorphism $F(A)\to\X(A,\mathfrak{f})$, natural in $A$.

We begin by showing that the Lie functor of a Lie group object is always representable. This means that there is an object $\g$ of the tangent category and a natural isomorphism:
\begin{align*}
&\varphi_A^G\colon\g(A)\xrightarrow{\cong}\X(A,\g)
\end{align*}
Furthermore, we show that $\g$ is $\T_\e G$ and that the representability condition is compatible with the tangent structure in that the Lie functor of each $\T^nG$ is also representable, represented by $\T^n\g$, and the representing isomorphism is fully determined by the one of $\g(-)$.

We also prove the converse: each group object $G$ whose Lie functor is representable and whose representability is compatible with the tangent bundle functor becomes a Lie group object. Finally, we give a third equivalent characterization of Lie groups as group objects $G$ whose Lie functor is representable, the representing object $\g$ is a differential object, and the representing isomorphism preserves multilinearity.

In the next section, we will use the representability of the Lie functor of a Lie group $G$ to construct the internal Lie algebra of $G$.

\subsection{Lie implies representability}
\label{subsection:lie-to-representability}
We begin by showing that the Lie functor $\g(-)$ of a Lie group object is always representable. Recall that, as pointed out in Remark~\ref{remark:negatives-group-objects}, a Lie group object always has negatives and thus, the Lie functor of $G$ is well-defined.

\begin{proposition}
\label{proposition:lie-functor-is-representable}
The Lie functor $\g(-)$ of a Lie group object $G$ is representable, and is represented by the tangent space $\T_\e G$ at the unit.
\end{proposition}
\begin{proof}
We need to construct a natural isomorphism $\varphi_A^G\colon\g(A)\to\X(A,\T_\e G)$. However, by definition $\g(A)$ is the Lie algebra of left-invariant vector fields over $G$ in $\Simple_A[\X,\TT]$. Thus, we can apply Lemma~\ref{lemma:correspondence-livf-global-elements} and obtain a bijection:
\begin{align*}
&\g(A)=\LIVF_G(\Simple_A[\X,\TT])\xrightarrow{\cong}\Simple_A[\X,\TT](\*,\T_\e G)
\end{align*}
However, a morphism from $\*$ to $\T_\e G$ in $\Simple_A[\X,\TT]$ consists of a morphism $\*\times A\to\T_\e G$ in the base category. So, morphisms $\*\to\T_\e G$ in $\Simple_A[\X,\TT]$ are in bijection with morphisms $A\to\T_\e G$ in $\X$. Therefore, we have a bijection:
\begin{align*}
&\varphi\colon\g(A)\xrightarrow{\cong}\X(A,\T_\e G)
\end{align*}
In particular, a morphism $f\colon A\to\T_eG$ corresponds to the $A$-\PLIVF:
\begin{align*}
&\varphi^{-1}(f)=X_f\colon G\times A\xrightarrow{\id_G\times f}G\times\T_\e G\xrightarrow{\gamma_G}\T G
\end{align*}
To prove that this bijection is natural in $A$, consider a morphism $g\colon A\to B$ in $\X$ and a morphism $f\colon A\to\T_\e G$. Thus:
\begin{align*}
&\g(g)[\varphi^{-1}(f)]=\g(g)[X_f]=(\id_G\times g)X_f=(\id_G\times g)(\id_G\times f)\gamma_G\\
&\qquad=(\id_G\times(gf))\gamma_G=X_{gf}=\varphi^{-1}(gf)=\varphi^{-1}(\X(g,\T_\e G)[f])
\end{align*}
where $\gamma_G\colon G\times\T_\e G\cong\T G$. This proves that $\varphi^{-1}$ is natural and that $\varphi$ is a natural isomorphism.
\end{proof}

The representability of the Lie functor of a Lie group is also compatible with the tangent structure. In fact, since $\LieGr(\X,\TT)$ is a tangent category, the tangent bundle $\T G$ is a Lie group is again a Lie group. By Proposition~\ref{proposition:lie-functor-is-representable}, this also means that the Lie functor $(\T\g)(-)$ of $\T G$ is representable and represented by the tangent space of $\T G$ at $\T\e$, that is, $\T\T_\e G$. Moreover, since $\T_\e G$ is a differential object, $\T\T_\e G\cong\T_\e G\times\T_\e G$.

\begin{proposition}
\label{proposition:lie-functor-representable-but-tangent}
If $G$ is a Lie group then, for each $n\geq0$, the Lie functor $(\T^n\g)(-)$ of $\T^nG$ is representable, represented by $\T^n\T_\e G$. Furthermore, for each $n\geq0$, the following diagram
\begin{equation*}
\begin{tikzcd}
{(\T^n\g)(A)} & {\X(A,\T^n\T_\e G)} \\
{(\T^{n+1}\g)(\T A)} & {\X(\T A,\T^{n+1}\T_\e G)}
\arrow["{\varphi_A^{\T^nG}}", from=1-1, to=1-2]
\arrow["{\tau_A^{\T^nG}}"', from=1-1, to=2-1]
\arrow["{\T(-)}", from=1-2, to=2-2]
\arrow["{\varphi_{\T A}^{\T^{n+1}G}}"', from=2-1, to=2-2]
\end{tikzcd}
\end{equation*}
commutes, where $\varphi_A^{\T^nG}$ is the representing natural isomorphism of $(\T^n\g)(-)$ and $\tau_A^{\T^nG}$ is the natural transformation of Lemma~\ref{lemma:tau-lie-algebra-homomorphism}.
\end{proposition}
\begin{proof}
We already proved that the Lie functor of a Lie group is representable, and we also commented that if $G$ is a Lie group, so is each $\T^nG$. It remains to show the commutativity of the representing isomorphisms with the tangent bundle functor. Recall that the inverse $\varphi^{-1}\colon\X(A,\T_\e G)\to\g(A)$ sends a morphism $f\colon A\to\T_\e G$ to $X_f=(\id_G\times A)\gamma_G$, where $\gamma_G\colon G\times\T_\e G\cong\T G$. Consider a morphism $f\colon A\to\T^n\T_\e G$. Thus, $\tau(X_f)$ is the \PLIVF so defined:
\begin{align*}
&\tau(X_f)=\T(X_f)c_{\T^nG}=(\id_{\T\T^nG}\times\T f)\T\gamma_Gc_{\T^nG}
\end{align*}
However, $\gamma_G=(z_G\times\iota_\e)\T\m_G$. Therefore, for any Lie group $G$:
\begin{align*}
&\T\gamma_Gc_G=(\T z_G\times\T\iota_\e)\T\T\m_Gc_G=(\T z_Gc_G\times\T\iota_\e c_G)\T\T\m_G=(z_{\T G}\times\T\iota_\e c_G)\T\T\m_G
\end{align*}
Lastly, $\iota_{\T\e}=\T\iota_\e c_G$, since $p_{\T G}=c_G\T p_G$. Therefore, $\T\gamma_Gc_G=\gamma_{\T G}$. From this, we conclude that $(\id_{\T\T^nG}\times\T f)\gamma_{\T\T^nG}=X_{\T f}=\varphi^{-1}(\T f)$. Therefore, $\varphi^{-1}\tau=\T(-)\varphi^{-1}$. By applying $\varphi$ on both sides, we finally obtain that $\tau\varphi=\varphi\T(-)$.
\end{proof}

\subsection{Representability implies Lie}
\label{subsection:representability-to-lie}
Previously, we showed that the Lie functor of a Lie group is representable and that the representability is compatible with the tangent structure, as established by Proposition~\ref{proposition:lie-functor-representable-but-tangent}. In this section, we prove that also the converse holds, in that if the Lie group of a group object with negatives $G$ is representable and the representability is compatible with the tangent structure, then $G$ is necessarily a Lie group. We also prove that the representing object $\g$ of the Lie functor is the tangent space at the unit of $G$.

In order to prove this result, we first reinterpret the representability of the Lie functor via the existence of a universal \PLIVF\!.

\begin{proposition}
\label{proposition:universal-vector-field}
Consider a group object with negatives $G$. The Lie functor $\g(-)$ of $G$ is representable if and only if there exists an object $\g\in(\X,\TT)$ together with a $\g$-\PLIVF $\omega_G\colon G\times\g\to\T G$ over $G$, subject to the following universal property. For any object $A$ and any $A$-\PLIVF $X\colon G\times A\to\T G$ there exists a unique morphism $\hat X\colon A\to\g$ such that, the following diagram
\begin{equation*}
\begin{tikzcd}
{G\times A} & {\T G} \\
{G\times\g}
\arrow["X", from=1-1, to=1-2]
\arrow["{\id_G\times\hat X}"', from=1-1, to=2-1]
\arrow["{\omega_G}"', from=2-1, to=1-2]
\end{tikzcd}
\end{equation*}
commutes.
\end{proposition}
\begin{proof}
Let us start by assuming that $\g(-)$ is representable and let $\g$ be the representing object. Via the isomorphism $\varphi\colon\g(\g)\to\X(\g,\g)$, the identity over $\g$ corresponds to a $\g$-\PLIVF $\omega_G\colon G\times\g\to\T G$ over $G$. Now, consider an $A$-\PLIVF $X\colon G\times A\to\T G$ over $G$. By the representability of $\g(-)$, it corresponds to a morphism $\hat X\=\varphi_A(X)\colon A\to\g$. However, since $\varphi_A$ is natural, the following diagram also commutes:
\begin{equation*}
\begin{tikzcd}
{\X(\g,\g)} & {\g(\g)} \\
{\X(A,\g)} & {\g(A)}
\arrow["{\varphi_\g^{-1}}", from=1-1, to=1-2]
\arrow["{\X(\hat X,\g)}"', from=1-1, to=2-1]
\arrow["{\g(\hat X)}", from=1-2, to=2-2]
\arrow["{\varphi_A^{-1}}"', from=2-1, to=2-2]
\end{tikzcd}
\end{equation*}
On the one hand, for a given morphism $f\colon A\to B$, $\g(f)\colon\g(B)\to\g(A)$ sends a $B$-\PLIVF $Y\colon G\times B\to\T G$ of $G$ to $(f\times\id_G)Y\colon G\times A\to\T G$. Therefore, the identity $\id_\g$, is sent by $\varphi^{-1}_\g$ to $\omega_G$, which is sent by $\g(\hat X)$ to $(\hat X\times\id_G)\omega_G$. On the other hand, the identity $\id_\g$ is sent to $\hat X$ by $\X(\hat X,\g)$, which is sent to $X$ by $\varphi_A^{-1}$, by definition of $\hat X$. This proves that $X=(\hat X\times\id_G)\omega_G$. Now, suppose that $f\colon A\to\g$ satisfies the same equation, that is, $(f\times\id_G)\omega_G=X$. Then, $\g(f)[\varphi_\g^{-1}(\id_\g)]=\g(f)[\omega_G]=X$. Using naturality again, we can rewrite:
\begin{align*}
&\varphi_A^{-1}(\hat X)=X=\g(f)[\varphi_\g^{-1}(\id_\g)]=\varphi_A^{-1}(f)
\end{align*}
which implies that $f=\hat X$.

Conversely, let us assume that $G$ comes with an object $\g$ and a $\g$-\PLIVF $\omega_G$ over $G$ satisfying the above universal property. This universal property induces a function $\varphi_A\colon\g(A)\to\X(A,\g)$, which sends each $A$-\PLIVF $X$ over $G$ to the unique morphism $\hat X\colon A\to\g$. Furthermore, every morphism $f\colon A\to\g$ defines an $A$-\PLIVF $X_f\=\g(f)[\omega_G]$, since $\g(f)\colon\g(\g)\to\g(A)$. Let us prove that these two operations invert each other. Take a morphism $f\colon A\to\g$. Thus, $\hat X_f$ is the unique morphism satisfying $(f\times\id_G)\omega_G=X_f$. However, by definition of $X_f=\g(f)[\omega_G]=(f\times\id_G)\omega_G$. Therefore, $\hat X_f=f$. Conversely, for an $A$-\PLIVF $X$, $X_{\hat X}=\g(\hat X)[\omega_G]=(\hat X\times\id_G)\omega_G=X$. Finally, let us prove that the transformation $\psi_A\colon\X(A,\g)\to\g(A)$, that sends each $f\in\X(A,\g)$ to $X_f=\g(f)[\omega_G]$ is natural. Consider $h\colon A\to B$. Thus, $\X(h,\g)\psi_A$ sends each $g\colon B\to\g$ to:
\begin{align*}
&\g(hg)[\omega_G]=(hg\times\id_G)\omega_G=(h\times\id_G)(g\times\id_G)\omega_G=\g(h)\left(\g(g)[\omega_G]\right)=\g(h)\psi_A(g)
\end{align*}
This proves that $\varphi_A$ is a natural isomorphism and therefore that the Lie functor of $G$ is representable.
\end{proof}

Thanks to this characterization, we can prove the following lemma.

\begin{lemma}
\label{lemma:representable-implies-pullback}
Suppose that $G$ is a group object with negatives whose Lie functor $\g(-)$ is representable, represented by $\g$. Then, the following diagram
\begin{equation*}
\begin{tikzcd}
\g & {\T G} \\
{\*} & G
\arrow["{\iota_\e}", from=1-1, to=1-2]
\arrow[from=1-1, to=2-1]
\arrow["{p_G}", from=1-2, to=2-2]
\arrow["\e_G"', from=2-1, to=2-2]
\end{tikzcd}
\end{equation*}
is a pullback diagram, where
\begin{align*}
&\iota_\e\colon\g\xrightarrow{\<!\e,\id_\g\>}G\times\g\xrightarrow{\omega}\T G
\end{align*}
and $\omega$ is the universal \PLIVF of $G$ of Proposition~\ref{proposition:universal-vector-field}.
\end{lemma}
\begin{proof}
Consider a morphism $f\colon A\to\T G$, subject to the equation, $fp_G=!\e$. Let us construct the following map:
\begin{align*}
&X_f\colon G\times A\xrightarrow{z_G\times f}\T(G\times G)\xrightarrow{\T\m_G}\T G
\end{align*}
We want to prove that $X_f$ is an $A$-\PLIVF of $G$. First, let us compute the following:
\begin{align*}
X_f p_G&=~(z_G\times f)\T\m_G p_G          \\
&=~(z_G\times f)(p_G\times p_G)\m_G         \tag{Naturality of $p$}\\
&=~(z_Gp_G\times fp_G)\m_G                  \\
&=~(\id_G\times!\e)\m_G                     \Tag{z_Gp_G=\id_G, fp_G=!\e}\\
&=~\pi_1                                    \tag{Unitality}
\end{align*}
This proves that $X_f$ is an $A$-parametrized vector field. Now, let us prove that it is also left-invariant:
\begin{align*}
(\m_G\times\id_A)X_f&=~(\m_G\times\id_A)(z_G\times f)\T\m_G     \tag{Definition of $X_f$}\\
&=~(z_G\times z_G\times f)\T(\m_G\times\id_G)\T\m_G             \tag{Naturality of $z$}\\
&=~(z_G\times z_G\times f)\T(\id_G\times\m_G)\T\m_G             \tag{Associativity of $\m_G$}\\
&=~(z_G\times X_f)\T\m_G                                        \tag{Definition of $X_f$}
\end{align*}
This proves that $X_f$ is an $A$-\PLIVF of $G$, therefore, by invoking the universal property of $\omega$, there exists a unique morphism $\hat X_f\colon A\to\g$ such that, $(\id_G\times\hat X_f)\omega=X_f$. We want to prove that $\hat X_f\iota_\e=f$:
\begin{align*}
\hat X_f\iota_\e&=~\hat X_f\<!\e,\id_\g\>\omega      \tag{Definition of $\iota_\e$}\\
&=~\<!\e,\hat X_f\>\omega\\
&=~\<!\e,\id_\g\>(\id_G\times\hat X_f)\omega       \\
&=~\<!\e,\id_\g\>X_f                               \tag{Definition of $\hat X_f$}\\
&=~\<!\e,\id_\g\>(z_G\times f)\T\m_G                \tag{Definition of $X_f$}\\
&=~\<!\T\e,f\>\T\m_G                                \tag{Naturality of $z$}\\
&=f                                                 \tag{Unitality of $\m_G$}
\end{align*}
Now, suppose that $g\colon A\to\g$ satisfies the same equation, that is, $g\iota_\e=f$. This means that $f=g\<!\e,\id_\g\>\omega$. Therefore, we can compute:
\begin{align*}
X_f&=~(z_G\times f)\T\m_G                          \tag{Definition of $X_f$}\\
&=~(z_G\times g\<!\e,\id_\g\>\omega)\T\m_G          \Tag{g\iota_\e=f}\\
&=~(\id_G\times g\<!\e,\id_\g\>)(z_G\times\omega)\T\m_G\\
&=~(\id_G\times g)(\id_G\times\<!\e,\id_\g\>)(\m_G\times\id_\g)\omega    \tag{Left-invariance of $\omega$}\\
&=~(\id_G\times g)\omega                            \tag{Unitality of $\m_G$}
\end{align*}
However, $\hat X_f$ is the unique morphism satisfying $X_f=(\id_G\times\hat X_f)\omega$. Therefore, $g=\hat X_f$. This proves that $\g$ is the pullback of the tangent bundle $p_G$ along $\e$.
\end{proof}

The map $\iota_\e\colon\g\to\T G$ is defined by $\omega$. It is useful to see that also the converse is true, in that $\omega$ is fully determined by $\iota_\e$.

\begin{lemma}
\label{lemma:universal-vector-field-via-iota}
If $G$ is a Lie group and $\omega$ denotes its universal \PLIVF\!, the following diagram
\begin{equation*}
\begin{tikzcd}
{G\times\T_\e G} & {\T G} \\
{\T(G\times G)}
\arrow["\omega", from=1-1, to=1-2]
\arrow["{z_G\times\iota_\e}"', from=1-1, to=2-1]
\arrow["{\T\m_G}"', from=2-1, to=1-2]
\end{tikzcd}
\end{equation*}
\end{lemma}
\begin{proof}
To prove this, we use that $\iota_\e=\<!\e,\id_\g\>\omega$ together with the left invariance of $\omega$, that is, $(z_G\times\omega)\T\m_G=(\m_G\times\id_\g)\omega$:
\begin{align*}
&(z_G\times\iota_\e)\T\m_G=(z_G\times\<!\e,\id_\g\>\omega)\T\m_G=(\id_G\times\<!\e,\id_\g\>)(z_G\times\omega)\T\m_G=\\
&\qquad=(\id_G\times\<!\e,\id_\g\>)(\m_G\times\id_\g)\omega=\omega
\end{align*}
where in the last step, we used the unitality of $\m_G$.
\end{proof}

The representability of the Lie functor implies that $\g$ is the pullback of the tangent bundle along the unit map; however, it does not imply that the pullback is a tangent pullback, which is a necessary condition for a group to be a Lie group. To solve this difficulty, we shall require the compatibility between the representability of the Lie functor and the tangent structure established in Proposition~\ref{proposition:lie-functor-representable-but-tangent}. First, we characterize this compatibility in terms of the universal \PLIVF $\omega$.

If $G$ is a group object and $X$ is an $A$-\PLIVF of $G$, by plugging in $X$ into the Lie sequence of $G$:
\begin{align*}
&\g(A)\xrightarrow{\tau_A^G}(\T\g)(\T A)\xrightarrow{\tau_{\T A}^{\T G}}(\T^2\g)(\T^2A)\to{\dots}
\end{align*}
we can construct a sequence $\{X^n\}_n$ of \PLIVFs where, for each $n\geq 0$, $X^n$ is a $\T^nA$-\PLIVF of $\T^nG$.

\begin{proposition}
\label{proposition:tangent-universal-vector-field}
Let $G$ be a group object with negatives. The following conditions are equivalent:
\begin{enumerate}
\item\label{item:tangent-universal-vector-field-1} For each $n\geq0$, the Lie functor $(\T^n\g)(-)$ of each $\T^nG$ is representable, represented by $\T^n\g$, where $\g$ represents $\g(-)$, and such that the following diagram
\begin{equation*}
\begin{tikzcd}
{(\T^n\g)(A)} & {\X(A,\T^n\g)} \\
{(\T^{n+1}\g)(\T A)} & {\X(\T A,\T^{n+1}\g)}
\arrow["{\varphi_A^{\T^nG}}", from=1-1, to=1-2]
\arrow["{\tau_A^{\T^nG}}"', from=1-1, to=2-1]
\arrow["{\T(-)}", from=1-2, to=2-2]
\arrow["{\varphi_{\T A}^{\T^{n+1}G}}"', from=2-1, to=2-2]
\end{tikzcd}
\end{equation*}
commutes;

\item\label{item:tangent-universal-vector-field-2} There exists an object $\g$ and a $\g$-\PLIVF $\omega$ of $G$ such that for each $A$-\PLIVF $X\colon A\times\T^nG\to\T\T^nG$ of $\T^nG$ there exists a unique morphism $\hat X\colon A\to\T^n\g$ such that:
\begin{equation*}
\begin{tikzcd}
{\T^nG\times A} & {\T\T^nG} \\
{\T^nG\times\T^n\g}
\arrow["X", from=1-1, to=1-2]
\arrow["{\id_{\T^nG}\times\hat X}"', from=1-1, to=2-1]
\arrow["{\omega^n}"', from=2-1, to=1-2]
\end{tikzcd}
\end{equation*}
where $\{\omega^n\}$ is the associated sequence of \PLIVFs\!.
\end{enumerate}
\end{proposition}
\begin{proof}
By Proposition~\ref{proposition:universal-vector-field}, we already know that the Lie functor $\g(-)$ is representable if and only if there exists an object $\g$ and a $\g$-\PLIVF $\omega$, satisfying the universal property of Condition~\ref{item:tangent-universal-vector-field-2}. Therefore, if Condition~\ref{item:tangent-universal-vector-field-1} holds, then there exists an object $\g$ and for each $n\geq0$ a $\T^n\g$-\PLIVF $\omega^n$, which satisfies the universal property of Condition~\ref{item:tangent-universal-vector-field-2}. So, to prove Condition~\ref{item:tangent-universal-vector-field-2}, it remains to show that the resulting sequence $\omega^n$ of \PLIVFs is entirely generated by $\omega$ via the Lie sequence. However, by assumption, the isomorphisms $\varphi_A^{\T^nG}$ commute with the natural transformations $\tau_A^{\T^nG}$. This implies that also the inverses commute with the $\tau$, as follows:
\begin{equation*}
\begin{tikzcd}
{\X(\T^n\g,\T^n\g)} & {(\T^n\g)(\T^n\g)} \\
{\X(\T\T^n\g,\T^{n+1}\g)} & {(\T^{n+1}\g)(\T\T^n\g)}
\arrow["{{\varphi_{\T^n\g}^{\T^nG}}^{-1}}", from=1-1, to=1-2]
\arrow["{\T(-)}"', from=1-1, to=2-1]
\arrow["{\tau_\T^n\g^{\T^nG}}", from=1-2, to=2-2]
\arrow["{{\varphi_{\T\T^n\g}^{\T^{n+1}G}}^{-1}}"', from=2-1, to=2-2]
\end{tikzcd}
\end{equation*}
On the one hand, the identity morphism over $\T^n\g$ is sent by $\varphi^{-1}$ to the universal \PLIVF $\omega^n$ of $\T^nG$, which is sent by $\tau$ to $\tau^{\T^nG}_{\T^n\g}(\omega^n)$. On the other hand, by functoriality, the identity over $\T^n\g$ is sent by $\T(-)$ to the identity, which is sent by $\varphi^{-1}$ to the universal \PLIVF $\omega^{n+1}$ of $\T^{n+1}G$. Therefore, $\omega^{n+1}=\tau^{\T^nG}_{\T^n\g}(\omega^n)$, for each $n\geq0$. By induction, we conclude that the sequence of universal \PLIVFs $\{\omega^n\}$ is fully generated by $\omega$ via the Lie sequence of $G$.

Conversely, suppose Condition~\ref{item:tangent-universal-vector-field-2} holds, and let us prove Condition~\ref{item:tangent-universal-vector-field-1}. By Proposition~\ref{proposition:universal-vector-field}, for every $n\geq0$, the Lie functor of $\T^nG$  is representable and is represented by $\T^n\g$. To conclude, we need to show that the isomorphisms $\varphi$ commute with the $\tau$. To this end, consider an $A$-\PLIVF $X\colon\T^nG\times A\to\T\T^nG$ of $\T^nG$. $X$ is sent by $\tau$ to the $\T A$-\PLIVF of $\T^{n+1}G$ so defined:
\begin{align*}
&\T\T^nG\times\T A\xrightarrow{\cong}\T(\T^nG\times A)\xrightarrow{\T X}\T\T\T^nG\xrightarrow{c_{\T^nG}}\T\T\T^nG
\end{align*}
However, since the universal vector field of $\T^{n+1}G$ is $\tau(\omega^n)=\T\omega^nc_{\T^nG}$, it follows that:
\begin{align*}
&(\id_{\T^{n+1}G}\times\T\hat X)\omega^{n+1}=(\id_{\T^{n+1}G}\times\T\hat X)\T\omega^nc_{\T^nG}=\T((\id_{\T^nG}\times\hat X)\omega^n)c_{\T^nG}=\T Xc_{\T^nG}=\tau(X)
\end{align*}
Therefore, by the universal property of $\omega^{n+1}$, $\T\hat X\colon\T A\to\T^{n+1}\g$ is the unique morphism, satisying $(\id_{\T^{n+1}G}\times\T\hat X)\omega^{n+1}=\tau(X)$. This precisely means that $\T(\varphi(X))=\T(\hat X)=\varphi(\tau(X))$.
\end{proof}

We will make use of the two equivalent conditions of Proposition~\ref{proposition:tangent-universal-vector-field} a few times; therefore, we decided to coin a useful expression.

\begin{definition}
\label{definition:tangently-representable}
The Lie functor of a group object with negatives is \textbf{tangently representable} provided that either one of the two equivalent conditions of Proposition~\ref{proposition:tangent-universal-vector-field} holds.
\end{definition}

With Lemma~\ref{lemma:representable-implies-pullback}, we proved that if the Lie functor of $G$ is representable, the representing object $\g$ is the \emph{pullback} of the tangent bundle of $G$ at the unit.

\begin{proposition}
\label{proposition:representable-implies-lie}
If the Lie functor of a group object with negatives $G$ is tangently representable, the pullback diagram of Lemma~\ref{lemma:representable-implies-pullback} becomes a tangent pullback.
\end{proposition}
\begin{proof}
The goal is to show that $\T^n$ preserves this pullback. For starters, let us apply $\T$ to the diagram. Notice that $\T\iota_\e$ becomes the following map:
\begin{align*}
&\T\iota_\e=\T(\<!\e_G,\id_\g\>\omega_G)=\<!\e_{\T G},\id_{\T\g}\>\T\omega_G
\end{align*}
where we used that $\T\e_G=\e_{\T G}$, by definition. Therefore, by postcomposing $\T\iota_\e$ by $c_G$, we obtain the map:
\begin{align*}
&\T\iota_\e c_G=\<!\e_{\T G},\id_{\T\g}\>\T\omega_Gc_G=\<!\e_{\T G},\id_{\T\g}\>\omega_{\T G}=\iota_{\T\e}
\end{align*}
where we used that, by Proposition~\ref{proposition:tangent-universal-vector-field}, the universal \PLIVF $\omega_{\T G}$ of $\T G$ is $\tau(\omega_G)=\T\omega c_G$.

Now, let us assume, by induction, that the pullback is preserved by $\T^n$ and let us prove that it is preserved by $\T^{n+1}$. For the sake of simplicity, let us denote $\T^nG$ by $H$ and let $\h$ denote $\T^n\g$. Let us consider the following diagram:
\begin{equation*}
\begin{tikzcd}
{\T\h} & {\T\T H} & {\T\T H} \\
{\*} & {\T H} & {\T H}
\arrow["{\T\iota_\e}", from=1-1, to=1-2]
\arrow["{\iota_{\T\e}}", curve={height=-24pt}, from=1-1, to=1-3]
\arrow[from=1-1, to=2-1]
\arrow["{c_H}", from=1-2, to=1-3]
\arrow["{\T p_H}", from=1-2, to=2-2]
\arrow["{p_{\T H}}", from=1-3, to=2-3]
\arrow["{\T\e_H}"', from=2-1, to=2-2]
\arrow["{\e_{\T H}}"', curve={height=24pt}, from=2-1, to=2-3]
\arrow[equals, nfold, from=2-2, to=2-3]
\end{tikzcd}
\end{equation*}
However, since $\T H$ is a Lie group, by Proposition~\ref{proposition:universal-vector-field}, the outer diagram is a pullback. However, the right diagram is also, trivially, a pullback, since $c_H$ and the identity are isomorphisms. Therefore, by the properties of pullbacks, so is the left diagram. This proves that if $\T^n$ preserves the pullback diagram of Proposition~\ref{proposition:universal-vector-field}, so does $\T^{n+1}$ and therefore, the diagram becomes a tangent pullback, by induction.
\end{proof}

If the representable Lie functor of a group object with negatives is tangently representable, the pullback of Lemma~\ref{lemma:representable-implies-pullback} turns into a tangent pullback. The converse is also true.

\begin{corollary}
\label{corollary:pullback-lie-groups}
A representable Lie functor of a group object with negatives $G$ is tangently representable if and only if the pullback diagram of Lemma~\ref{lemma:representable-implies-pullback} is a tangent pullback.
\end{corollary}
\begin{proof}
If the Lie functor is tangently representable, by Proposition~\ref{proposition:representable-implies-lie}, the pullback diagram of Lemma~\ref{lemma:representable-implies-pullback} is a tangent pullback. Conversely, if the Lie functor of $G$ is representable and the pullback diagram of Lemma~\ref{lemma:representable-implies-pullback} is a tangent pullback, then $G$ is a Lie group. Therefore, by Proposition~\ref{proposition:lie-functor-is-representable}, the Lie functor is tangently representable.
\end{proof}

\subsection{Differential representability}
\label{subsection:differential-representability}
In the previous two sections, we showed that the Lie functor of a Lie group is tangently representable (Proposition~\ref{proposition:lie-functor-is-representable}) and that a group object with negatives whose Lie functor is tangently representable is necessarily a Lie group (by Proposition~\ref{proposition:representable-implies-lie}). In this section, we explore an alternative way to express the compatibility between the representability of the Lie functor and the tangent structure. We call this \emph{differential representability}. We begin with a technical lemma.

\begin{lemma}
\label{lemma:tangent-spaces}
Consider an object $M$ in a Cartesian tangent category and suppose that the pullback diagram
\begin{equation*}
\begin{tikzcd}
{\m_x} & {\T M} \\
{\*} & M
\arrow["{\iota_x}", from=1-1, to=1-2]
\arrow[from=1-1, to=2-1]
\arrow["\lrcorner"{anchor=center, pos=0.125}, draw=none, from=1-1, to=2-2]
\arrow["{p_M}", from=1-2, to=2-2]
\arrow["x"', from=2-1, to=2-2]
\end{tikzcd}
\end{equation*}
of the tangent bundle $p_M$ along a global element $x\colon\*\to M$ exists. Then, this diagram is a tangent pullback diagram if and only if $\mathfrak{m}_x$ is a differential object and the inclusion map $\iota_x\colon\mathfrak{m}_x\to\T M$ is linear.
\end{lemma}
\begin{proof}
By~\cite[Corollary~3.5]{cockett:differential-bundles}, if the above diagram is a tangent pullback, then $\mathfrak{m}_x$ is a differential object. Moreover, since the vertical lift $\lambda_{\mathfrak{m}_x}$ of $\mathfrak{m}_x$ is constructed by pulling back $l_M\colon\T M\to\T\T M$ along $\iota_x$, $\iota_x$ becomes immediately linear. Conversely, suppose that the above diagram is a pullback, that $\mathfrak{m}_x$ is a differential object, and that $\iota_x$ is linear. We want to prove that the pullback is a tangent pullback. Consider a morphism $f\colon A\to\T\T M$ such that $f\T p_M=!\T x$. Thus:
\begin{align*}
fp_{\T M}p_Mz_M&=~f\T p_Mp_Mz_M     \tag{Naturality of $p$}\\
&=~!\T xp_Mz_M                      \Tag{f\T p_M=!\T x}\\
&=~!\T x                            \tag{Naturality of $pz$}\\
&=~f\T p_M                          \Tag{f\T p_M=!\T x}
\end{align*}
This means that $f\colon A\to\T\T M$ equalizes $\T p_M$ and $p_{\T M}p_Mz_M$. By the universal property of the vertical lift, there exists a unique morphism $\{f\}\colon A\to\T M$ such that $f=\<fp_{\T M}z_{\T M},\{f\}l_M\>\T s_M$ (see~\cite[Lemma~2.12]{cockett:tangent-cats}). Furthermore, by~\cite[Lemma~2.14]{cockett:tangent-cats}, $\{f\}p_M=f\T p_Mp_M=!\T xp_M=!x$. Therefore, by the universal property of the pullback diagram that defines $\mathfrak{m}_x$, there exists a unique morphism $f_2\colon A\to\mathfrak{m}_x$ that satisfies $f_2\iota_x=\{f\}$. Furthermore, using again the naturality of $p$, $fp_{\T M}p_M=f\T p_Mp_M=!\T xp_M=!x$. Therefore, again by the universal property of $\mathfrak{m}_x$, there also exists a unique morphism $f_1\colon A\to\mathfrak{m}_x$, such that, $f_1\iota_x=fp_{\T M}$. Now, using that $\mathfrak{m}_x$ is a differential object, there is an isomorphism, $\hat\xi_{\mathfrak{m}_x}\colon\mathfrak{m}_x\times\mathfrak{m}_x\cong\T\mathfrak{m}_x$ given by $(z_{\mathfrak{m}_x}\times\lambda_{\mathfrak{m}_x})\T+_{\mathfrak{m}_x}$. Therefore, we can construct the morphism $\<f_1,f_2\>\hat\xi_{\mathfrak{m}_x}\colon A\to\mathfrak{m}_x\times\mathfrak{m}_x\cong\T\mathfrak{m}_x$. However, since $\iota_x$ is linear, $\hat\xi_{\mathfrak{m}_x}\T\iota_x=(\iota_x\times\iota_x)\xi_M$, where $\xi_M=(z_{\T M}\times l_M)\T s_M$. Therefore, we compute:
\begin{align*}
&\<f_1,f_2\>\hat\xi_{\mathfrak{m}_x}\T\iota_x=\<f_1,f_2\>(\iota_x\times\iota_x)\xi_M=\<f_1\iota_x,f_2\iota_x\>\xi_M\\
&\qquad=\<fp_{\T M},\{f\}\>\xi_M=\<fp_{\T M}z_{\T M},\{f\}l_M\>\T s_M=f
\end{align*}
Now, suppose that $g\colon A\to\T\mathfrak{m}_x$ satisfies $g\T\iota_x=f$. However, since $\mathfrak{m}_x$ is a differential object, $\T\mathfrak{m}_x\cong\mathfrak{m}_x\times\mathfrak{m}_x$. Therefore, we can write $g$ as $\<g_1,g_2\>\hat\xi_{\mathfrak{m}_x}$. Thus:
\begin{align*}
&f=\<f_1\iota_x,f_2\iota_x\>\xi_M=g\T\iota_x=\<g_1,g_2\>\hat\xi_{\mathfrak{m}_x}\T\iota_x=\<g_1\iota_x,\g_2\iota_x\>\xi_M
\end{align*}
However, since $\xi_M$ is monic, implies that $g_1\iota_x=f_1\iota_x$ and that $g_2\iota_x=f_2\iota_x$. However, since $\iota_x$ is a pullback of a section, it is also monic; therefore, $g_1=f_1$ and $g_2=f_2$. Thus, $g=\<f_1,f_2\>\hat\xi_{\mathfrak{m}_x}$. This proves that the pullback is preserved by the tangent bundle functor. Finally, since $\T$ preserves differential objects and linear morphisms, by induction, the pullback turns into a tangent pullback.
\end{proof}

\begin{remark}
\label{remark:tangent-spaces-differential-objects}
Being a differential object is not a property but a structure in that an object can have more than one differential structure attached to it. Thus, a priori there could be multiple choices of the differential structure on $\mathfrak m_x$ in the previous lemma. However, the linearity of the inclusion map $\iota_x$ forces the differential structure to be unique and, in fact, entirely induced by the differential structure of the tangent bundle.
\end{remark}

By Corollary~\ref{corollary:pullback-lie-groups}, if the Lie functor of $G$ is representable, then it is tangently representable if and only if the pullback diagram of Lemma~\ref{lemma:representable-implies-pullback} is a tangent pullback. However, by Lemma~\ref {lemma:tangent-spaces}, this is equivalent to saying that the representing object $\g$ is a differential object and that the morphism $\iota_\e=\<!\e,\id_\g\>\omega$ is linear. This suggests an alternative way to characterize tangent representability.

We begin by reviewing multilinearity. Multilinearity was introduced in the context of Cartesian differential categories in~\cite{cruttwell:sector-forms-in-CDCs}. Differential objects of a Cartesian tangent category form a Cartesian differential category (see~\cite{cockett:tangent-cats}); therefore, we can consider multilinear maps from a list of differential objects to a single one. We adapt here this definition.

\begin{definition}[{\cite[Definition~2.9, modified]{cruttwell:sector-forms-in-CDCs}}]
\label{definition:multilinear-map}
Given a list of differential objects $A,A_1\,A_n$, a morphism $f\colon A_1\times\dots\times A_n\to A$ is \textbf{multilinear} provided that, for each index $k=1\,n$, $f$ is linear each $k$-th entry, that is, the following diagram
\begin{equation*}
\begin{tikzcd}
{\T(A_1\times\dots\times A_n)} & {\T A} \\
{A_1\times\dots\times A_n} & A
\arrow["{\T f}", from=1-1, to=1-2]
\arrow["{\lambda_k}", from=2-1, to=1-1]
\arrow["f"', from=2-1, to=2-2]
\arrow["{\lambda_A}"', from=2-2, to=1-2]
\end{tikzcd}
\end{equation*}
commutes for each $k=1\,n$, where
\begin{align*}
&\lambda_k\=z_{A_1}\times\dots\times z_{A_{k-1}}\times\lambda_{A_k}\times z_{A_{k+1}}\times\dots\times z_{A_n}
\end{align*}
and where $\lambda_{A_k}$ denotes the vertical lift of the differential object $A_k$, for $k=1\,n$. For $n=2$, we call a multilinear map \textbf{bilinear}, while for $n=1$, multilinearity is precisely linearity.
\end{definition}

Differential objects and multilinear maps form a multicategory.

\begin{lemma}
\label{lemma:inclusion-functor-preserves-differential-objects}
The inclusion functor $(\X,\TT)\to\Simple_A[\X,\TT]$ sends differential objects to differential objects.
\end{lemma}
\begin{proof}
By Proposition~\ref{proposition:cokleisli-tangent-category}, the inclusion functor $(\X,\TT)\to\Simple_A[\X,\TT]$ is a strict tangent morphism, and moreover, it preserves limits since it is a right adjoint. Therefore, by~\cite[Proposition~4.22]{cockett:differential-bundles}, it sends differential bundles to differential bundles. However, since it preserves the terminal object $\*$ and since differential objects are differential bundles over $\*$ (\cite[Proposition~3.4]{cockett:differential-bundles}), it also sends differential objects to differential objects.
\end{proof}

This lemma allows us to introduce multilinearity for parametrized vector fields.

\begin{definition}
\label{definition:linear-PLIVF}
Given a list $A_1\,A_n$ of differential objects, a \textbf{multilinear $(A_1\,A_n)$-\PLIVF} of a group object $G$ consists of an $(A_1\times\dots\times A_n)$-\PLIVF $X\colon G\times A_1\times\dots\times A_n\to\T G$ which is multilinear in $\Simple_A[\X,\TT]$. 
A \textbf{linear $A$-\PLIVF} is a multilinear $(A)$-\PLIVF\!.
\end{definition}

Concretely, a multilinear \PLIVF is a \PLIVF $X$ satisfying the following condition
\begin{equation*}
\begin{tikzcd}
{\T(G\times A_1\times\dots\times A_k\times\dots\times A_n)} & {\T\T G} \\
{G\times A_1\times\dots\times A_k\times\dots\times A_n} & {\T G}
\arrow["{\T X}", from=1-1, to=1-2]
\arrow["{z_G\times\lambda_k}", from=2-1, to=1-1]
\arrow["X"', from=2-1, to=2-2]
\arrow["{l_G}"', from=2-2, to=1-2]
\end{tikzcd}
\end{equation*}
for each $k=1\,n$.

\begin{definition}
\label{definition:differentiably-representable}
The Lie functor $\g(-)$ of a group object with negatives $G$ is \textbf{differentiably representable} provided that $\g(-)$ is representable, the representing object $\g$ is a differential object and furthermore, the representing isomorphism
\begin{align*}
&\varphi_{A_1\times\dots\times A_n}^G\colon\g(A_1\times\dots\times A_n)\to\X(A_1\times\dots\times A_n,\g)
\end{align*}
preserves multilinearity, in that it sends each multilinear $(A_1\,A_n)$-\PLIVF $X$ of $G$ to a multilinear morphism $\hat X\colon A_1\times\dots\times A_n\to\g$ of differential objects. Furthermore, its inverse $\varphi^{-1}$ sends multilinear morphisms to multilinear \PLIVFs\!.
\end{definition}

Following the established tradition of unpacking the definitions, let us translate the differentiability condition of Definition~\ref{definition:differentiably-representable} into a condition on the universal vector field $\omega$.

\begin{theorem}
\label{theorem:differential-universal-vector-field}
Suppose that the Lie functor $\g(-)$ of a group object with negatives $G$ is representable. Then, $\g(-)$ is differentiably representable if and only if the representing object $\g$ of the Lie functor is a differential object and its universal \PLIVF $\omega$ is linear.
\end{theorem}
\begin{proof}
If $G$ is a differential Lie group object, by definition, $\g$ is a differential object. Moreover, the natural isomorphism $\varphi^{-1}\colon\X(A,\g)\to\g(A)$ preserves multilinearity. However, since the identity of a differential object is always linear, $\omega=\varphi^{-1}(\id_\g)$ becomes a linear $\g$-\PLIVF\!. Conversely, assume that the Lie functor is representable, that $\g$ is a differential object, and that the universal \PLIVF $\omega$ is linear. If $A_1\,A_n$ is a list of differential objects and $f\colon A_1\times\dots\times A_n\to\g$ is a multilinear morphism, then, since $\omega$ is also linear, we compute:
\begin{align*}
&X_fl_M=(\id_G\times f)\omega l_M=(\id_G\times f)(z_G\times\lambda_\g)\T\omega=(z_G\times(f\lambda_\g))\T\omega=\\
&\qquad=(z_G\times\lambda_k)\T((\id_G\times f)\omega)=(z_G\times\lambda_k)\T X_f
\end{align*}
for each index $k=1\,n$, showing that the associated \PLIVF $X_f$ is multilinear. It is left to show that if $X\colon G\times A_1\times\dots\times A_n\to\T G$ is a multilinear \PLIVF\!, the corresponding morphism $\hat X\colon A_1\times\dots\times A_n\to\g$ is multilinear. Using the linearity of $\omega$, we can compute the following:
\begin{align*}
Xl_G&=~(\id_G\times\hat X)\omega l_G        \tag{Universal property of $\omega$}\\
&=~(\id_G\times\hat X)(z_G\times\lambda_\g)\T\omega \tag{Linearity of $\omega$}\\
&=~(\id_G\times(\hat X\lambda_\g))(z_G\times\id_{\T\g})\T\omega
\end{align*}
However, since $X$ is also multilinear, we can also write:
\begin{align*}
Xl_G&=~(z_G\times\lambda_k)\T X             \tag{Multilinearity of $X$}\\
&=~(z_G\times\lambda_k)\T((\id_G\times\hat X)\omega) \tag{Universal property of $\omega$}\\
&=~(\id_G\times(\lambda_k\T\hat X))(z_G\times\id_{\T\g})\T\omega
\end{align*}
Therefore, by postcomposing on the left and on the right by $c_G$, we have:
\begin{align*}
&(\id_G\times(\hat X\lambda_\g))(z_G\times\id_{\T\g})\omega_\T=(\id_G\times(\hat X\lambda_\g))(z_G\times\id_{\T\g})\T\omega c_G=\\
&\qquad=(\id_G\times(\lambda_k\T\hat X))(z_G\times\id_{\T\g})\T\omega c_G=(\id_G\times(\lambda_k\T\hat X))(z_G\times\id_{\T\g})\omega_\T
\end{align*}
However, we can use the left-invariance of $\omega_\T$ to rewrite $(z_G\times\id_{\T\g})\omega_\T$ as follows:
\begin{align*}
(z_G\times\id_{\T\g})\omega_\T&=~(\id_G\times\<!\e,\id_{\T\g}\>)(\m_G\times\id_{\T\g})(z_G\times\id_{\T\g})\omega_\T    \tag{Unitality}\\
&=~(\id_G\times\<!\e,\id_{\T\g}\>)(z_G\times z_G\times\id_{\T\g})(\T\m_G\times\id_{\T\g})\omega_\T  \tag{Naturality of $z$}\\
&=~(\id_G\times\<!\e,\id_{\T\g}\>)(z_G\times z_G\times\id_{\T\g})(z_{\T G}\times\omega_\T)\T\T\m_G      \tag{Left-invariance}\\
&=~((z_Gz_{\T G})\times(\<!\T\e,\id_{\T\g}\>\omega_\T))\T\T\m_G \tag{Naturality of $z$}\\
&=~((z_Gz_{\T G})\times\iota_{\T\e})\T\T\m_G      \Tag{\<!\T\e,\id_{\T\g}\>\omega_\T=\iota_{\T\e}}
\end{align*}
Thus, by precomposing by $\<!\e,\id_{\T\g}\>$ on both sides and by using the formula $(z_G\times\id_{\T\g})\omega_\T=((z_Gz_{\T G})\times\iota_{\T\e})\T\T\m_G$, we can compute:
\begin{align*}
\hat X\lambda_\g\iota_{\T\e}&=~X\lambda(\<!\e,\id_{\T\g}\>)((z_Gz_{\T\e})\times\iota_{\T\e})\T\T\m_G       \tag{Unitality}\\
&=~X\lambda_\g(\<!\e,\id_{\T\g}\>)(z_G\times\id_{\T\g})\omega_\T\\
&=~\<!\e,\id_{\T\g}\>(\id_G\times(X\lambda_\g))(z_G\times\id_{\T\g})\omega_\T\\
&=~\<!\e,\id_{\T\g}\>(\id_G\times(\lambda_k\T\hat X))(z_G\times\id_{\T\g})\omega_\T\\
&=~(\lambda_k\T\hat X)\<!\e,\id_{\T\g}\>(z_G\times\id_{\T\g})\omega_\T\\
&=~(\lambda_k\T\hat X)\<!\e,\id_{\T\g}\>((z_Gz_{\T G})\times\iota_{\T\e})\T\T\m_G\\
&=~\lambda_k\T\hat X\iota_{\T\e}                        \tag{Unitality}
\end{align*}
To conclude, we need to show that $\iota_{\T\e}$ is monic. If so, then $\hat X\lambda_\g=\lambda_k\T\hat X$, which proves that $\hat X$ is multilinear. To prove that $\iota_{\T\e}$ is monic, we invoke Lemma~\ref{lemma:tangent-spaces}. By Proposition~\ref{proposition:representable-implies-lie}, $\g$ is a pullback of $p_G$ along the unit map $\e\colon\*\to G$. Furthermore, by assumption $\g$ is a differential object and since $\omega$ is linear, $\iota=\<!\e,\id_\g\>\omega$ is also linear. This is an immediate consequence of the linearity of $\omega$. Therefore, the pullback of Proposition~\ref{proposition:representable-implies-lie} is a tangent pullback. Therefore, $\T\iota_\e$ is the pullback of a section, that is, $\T\e$, and therefore, $\T\iota_\e$ is monic. Finally, since $\iota_{\T\e}=\T\iota_\e c_G$, we conclude that $\iota_{\T\e}$ is monic.
\end{proof}

This result gives us the third characterization of Lie groups.

\begin{proposition}
\label{proposition:tangently-differentiably}
The representable Lie functor of a group object $G$ is tangently representable if and only if it is differentiably representable.
\end{proposition}
\begin{proof}
Thanks to Theorem~\ref{theorem:differential-universal-vector-field}, if the Lie functor $\g(-)$ is differentiably representable, then the universal \PLIVF $\omega$ is linear, therefore, we can compute:
\begin{align*}
&\iota_\e l_G=\<!\e,\id_\g\>\omega l_G=\<!\e,\id_\g\>(z_G\times\lambda_\g)\T\omega=\<!\T\e,\lambda_\g\>\T\omega=\lambda_\g\T(\<!\e,\id_\g\>\omega)=\lambda_\g\T\iota
\end{align*}
Therefore, $\iota_\e$ becomes linear. Thus, by Lemma~\ref{lemma:tangent-spaces}, the pullback of Lemma~\ref{lemma:representable-implies-pullback} becomes a tangent pullback and, thus, by Corollary~\ref{corollary:pullback-lie-groups}, the Lie functor becomes tangently representable. Conversely, if the Lie functor is tangently representable, then, by Proposition~\ref{proposition:lie-functor-is-representable}, the pullback of Lemma~\ref{lemma:representable-implies-pullback} is a tangent pullback and therefore, by Lemma~\ref{lemma:tangent-spaces}, $\g$ is a differential object and $\iota_\e$ is linear. However, by Lemma~\ref{lemma:universal-vector-field-via-iota}, $\omega=(z_G\times\iota_\e)\T\m_G$. By the naturality of the vertical lift, each morphism of type $\T f$ is always linear. Therefore, $\omega$ becomes a linear \PLIVF\!.
\end{proof}

Theorem~\ref{theorem:trivial-tangent-bundle-groups} shows that the tangent bundle functor of a Lie group object is trivial. It turns out that, for objects with at least a global element, the triviality of the tangent bundle implies the existence of a tangent space at that element. We start with a definition.

\begin{definition}[Cruttwell, Ikonicoff, Lemay, Van Der Linden\footnote{This definition has been folklore in the tangent category community for some time. It was mentioned in informal conversations by Geoff Cruttwell and was studied in unpublished work by Sacha Ikonicoff, JS Lemay, and Tim Van Der Linden.}]
\label{definition:parallelizable-object}
A \textbf{parallelizable object} in a Cartesian tangent category consists of an object $M$ whose tangent bundle functor $\T M$ is isomorphic to the Cartesian product $M\times\mathfrak{m}$, where $\mathfrak{m}$ is a differential object and the isomorphism $\nu_M\colon\T M\cong M\times\mathfrak{m}$ is linear, that is, $\nu_M(z_M\times\lambda_M)=l_M\T\nu_M$, where $\lambda_M$ denotes the vertical lift of $\mathfrak{m}$.
\end{definition}

\begin{lemma}
\label{lemma:parallelizability}
If a parallelizable object $M$ comes equipped with a global element $x\colon\*\to M$, then $M$ admits a tangent space at $x$. In particular, the following diagram
\begin{equation*}
\begin{tikzcd}
{\mathfrak m} & {\T M} \\
{\*} & M
\arrow["{\iota_x}", from=1-1, to=1-2]
\arrow[from=1-1, to=2-1]
\arrow["{p_M}", from=1-2, to=2-2]
\arrow["x"', from=2-1, to=2-2]
\end{tikzcd}
\end{equation*}
is a tangent pullback, where:
\begin{align*}
&\iota_x\colon\mathfrak{m}\xrightarrow{\<!x,\id_{\mathfrak{m}}\>}M\times\mathfrak m\xrightarrow{\nu_M^{-1}}\T M
\end{align*}
\end{lemma}
\begin{proof}
For starters, $\nu\colon\T M\cong M\times\mathfrak m$ preserves the projection, thus, $\iota_xp_M=\<!x,\id_{\mathfrak{m}}\>\nu_M^{-1}p_M=\<!x,\id_{\mathfrak{m}}\>\pi_1=!x$. Therefore, the diagram commutes. Now consider a morphism $f\colon A\to\T M$ such that $fp_M=!x$. Thus, we can define a morphism $\tilde f\=f\nu\pi_2\colon A\to\mathfrak{m}$. Therefore,
\begin{align*}
&\tilde f\iota_x=\tilde f\<!x,\id_\mathfrak{m}\>\nu_M^{-1}=\<!x,f\nu\pi_2\>\nu_M^{-1}=\<f\nu\pi_1,f\nu\pi_2\>\nu_M^{-1}=f\nu(\pi_1\times\pi_2)\nu_M^{-1}=f
\end{align*}
where we used that $!x=fp_M=f\nu\pi_1$ and that $\pi_1\times\pi_2=\id_{M\times\mathfrak{m}}$. Now, suppose that $g\colon A\to\mathfrak{m}$ satisfies the same equation, that is, $g\iota_x=f$. Thus, $f=g\<!x,\id_{\mathfrak{m}}\>\nu_M^{-1}=\<!x,g\>\nu_M^{-1}$. Therefore:
\begin{align*}
&\tilde f=f\nu\pi_2=\<!x,g\>\nu_M^{-1}\nu\pi_2=\<!x,g\>\pi_2=g
\end{align*}
This proves that the diagram is a pullback. Furthermore, we can show that $\iota_x$ is linear:
\begin{align*}
&\iota_xl_M=\<!x,\id_{\mathfrak{m}}\>\nu_M^{-1}l_M=\<!x,\id_{\mathfrak{m}}\>(z_M\times\lambda_\mathfrak{m})\T\nu_M^{-1}\<!\T x,\lambda_\mathfrak{m}\>\T\nu_M^{-1}=\lambda_\mathfrak{m}\T(\<!x,\id_{\mathfrak{m}}\>\nu_M^{-1})=\lambda_\mathfrak{m}\T\iota_x
\end{align*}
Therefore, by applying Lemma~\ref{lemma:tangent-spaces}, we conclude that the diagram is a tangent pullback.
\end{proof}

\subsection{A full characterization of Lie groups}
\label{subsection:first-fundamental-lie}
We can finally put everything together in a general theorem.

\begin{theorem}
\label{theorem:first-fundamental-lie}
In a Cartesian tangent category, given a group object $G$, the following statements are equivalent:
\begin{enumerate}
\item $G$ is a Lie group;

\item $G$ admits negatives and its Lie functor is tangently representable;

\item $G$ admits negatives and its Lie functor is differentiably representable;

\item $G$ is parallelizable.
\end{enumerate}
\end{theorem}

In most of the cases of interest, the existence of the tangent space at the unit is an immediate consequence of the nice properties of the ambient tangent category. This is the case of well-displayed tangent categories (\cite[Definition~2.12]{cruttwell:tangent-display-maps}), which are tangent categories whose tangent bundles $p_M\colon\T M\to M$ are tangent display maps, in that they satisfy the following condition. For every $n\geq0$, $\T^np_M$ admits all tangent pullbacks along arbitrary maps.

\begin{lemma}
\label{lemma:tangent-display-lie}
In a well-displayed Cartesian tangent category, every group object is a Lie group object.
\end{lemma}
\begin{proof}
Since tangent bundles are tangent display, the tangent pullback of $p_G\colon\T G\to G$ along $\e\colon\*\to G$ exists. Therefore, $G$ has a tangent space at the unit.
\end{proof}

Thanks to Theorem~\ref{theorem:first-fundamental-lie}, since the objects of a GCDC are parallelizable, every group object in a GCDC is a Lie group.

\begin{corollary}
\label{corollary:GCDC-lie}
In a GCDC, every group object is a Lie group object.
\end{corollary}


\section{The differential Lie algebra of a Lie group}
\label{section:lie-algebras}
The main objective of this paper is to internalize the \emph{external} Lie algebra of left-invariant vector fields. In this section, we accomplish this goal by invoking the representability of the Lie functor $\g(-)$ associated to such a group. We start by introducing Lie algebra objects in tangent categories. First, we recall the usual notion of a Lie algebra object in an arbitrary Cartesian category; then we introduce a stronger notion, which is compatible with the tangent structure. We call this notion \emph{differential} Lie algebras. Finally, we prove that every Lie group has an associated differential Lie algebra.

\subsection{Lie algebra objects}
\label{subsection:lie-algebra-objects}
We start by recalling the notion of a Lie algebra in an arbitrary Cartesian category. A Lie algebra object in $\X$ consists of an object $\g$ together with three maps
\begin{align*}
&0\colon\*\to\g         &&+\colon\g\times\g\to\g        &&[,]\colon\g\times\g\to\g
\end{align*}
such that $(\g,0,+)$ is a commutative monoid, the binary operation $[,]$ is additive in the first variable, that is,
\begin{align}
\label{equation:antisymmetry-jacobi}
&\<0,\id_\g\>[,]=0      &&(+\times\id_\g)[,]=\<\pi_1,\pi_2,\pi_3,\pi_2\>([,]\times[,])+
\end{align}
and it satisfies the following two equations:
\begin{align*}
&[,]+\tau[,]=0          &&[[,],]+\sigma[[,],]+\sigma^2[[,],]=0
\end{align*}
where, for two parallel morphisms $f,g\colon A\to\g$ we used the notation $f+g\=\<f,g\>+$ and $0\colon A\to\g$ denotes $0\=!0$. Moreover, $\tau\colon\g\times\g\to\g\times\g$ denotes the symmetry and $\sigma\colon\g\times\g\times\g\to\g\times\g\times\g$ is the ciclic permutation $(1\quad2\quad3)$. The former equation establishes that $[,]$ is antisymmetric, while the latter is known as the \textbf{Jacobi identity} and can be interpreted as a special kind of associativity. The binary operation $[,]$ is known as the Lie bracket.

In a Cartesian tangent category, one would like the Lie algebra structure of a Lie algebra object to be compatible with the tangent structure in a suitable sense. 

\begin{definition}
\label{definition:differential-lie-algebra}
A \textbf{differential Lie algebra} in a Cartesian tangent category, consists of a differential object $\g$ equipped with a bilinear map $[,]\colon\g\times\g\to\g$ such that $(\g,0,+,[,])$ is a Lie algebra object, where $0$ and $+$ are the zero and the sum of the differential object $\g$.
\end{definition}

Since linearity implies additivity and similarly multilinearity implies additivity in each variable, to define a differential Lie algebra is only required to have a differential object $\g$ equipped with a bilinear map $[,]\colon\g\times\g\to\g$ which satisfies antilinearity and the Jacobi identity.

\begin{lemma}
\label{lemma:differential-lie-algebras}
A bilinear morphism $[,]\colon\g\times\g\to\g$ on a differential object $\g$ makes $\g$ into a differential Lie algebra if and only if it satisfies Equations~\eqref{equation:antisymmetry-jacobi} with respect to the additive structure of $\g$.
\end{lemma}

\subsection{The Lie algebra of a Lie group}
\label{subsection:lie-algebra-of-lie-groups}
In this section, we construct the Lie algebra object of a Lie group in a Cartesian tangent category. Consider a group object $G$. For each $A$, $\g(A)$ is a Lie algebra in $\Set$. Moreover, each morphism $f\colon A\to B$ corresponds to a Lie algebra morphism $\g(f)\colon\g(B)\to\g(A)$. This implies that the Lie algebra structure on each $\g(-)$ is natural. This translates into the following lemma.

\begin{lemma}
\label{lemma:lie-algebra-in-presheaf}
The Lie functor $\g(-)\colon\X^\op\to\Set$ of a group object $G$ is a Lie algebra object in the presheaf category $\Cat(\X^\op,\Set)$. 
\end{lemma}

If now $G$ is a Lie group object, so that $\g(-)$ becomes a representable functor, by Yoneda's lemma, we immediately conclude that the representing object $\g$ becomes a Lie algebra object \emph{internal} to the Cartesian tangent category. However, we also would like to have a concrete description of the Lie algebra structure of $\g$. For starters, by Lemma~\ref{lemma:additive-structure-tangent-spaces}, $z_G\colon G\to\T G$ is always a left-invariant vector field of $G$, so in particular, it is a $\*$-\PLIVF of $G$. Therefore, under the representability of the Lie functor, it corresponds to a morphism $0\colon\*\to\g$. This will be the zero of the Lie algebra.

To construct the sum and the Lie bracket, we first consider the two projections, $\pi_1,\pi_2\colon\g\times\g\to\g$. Using that $\g(-)$ is contravariant, we obtain two Lie algebra homomorphisms $\g(\pi_k)\colon\g(\g)\to\g(\g\times\g)$, for $k=1,2$. In particular, the universal vector field $\omega\in\g(\g)$ of $G$ is sent to two $(\g\times\g)$-\PLIVF of $G$, respectively denoted by $\omega_1,\omega_2\colon G\times\g\times\g\to\T G$. Since $\g(\g\times\g)$ is a Lie algebra, we can take the sum of $\omega_1$ and $\omega_2$ and define a new $(\g\times\g)$-\PLIVF $\omega_1+\omega_2$. However, by representability, this will correspond to a morphism $+\colon\g\times\g\to\g$.

Similarly, we can construct the Lie bracket by a similar argument: the $(\g\times\g)$-\PLIVF $[\omega_1,\omega_2]$ corresponds to a morphism $[,]\colon\g\times\g\to\g$.

\begin{theorem}
\label{theorem:lie-algebra-of-lie-group}
The representing object $\g$ of the Lie functor of a Lie group object $G$ is a Lie algebra object internal to the Cartesian tangent category, whose structural morphisms are defined by the following correspondence:
\begin{prooftree}
\AxiomC{$z_G\colon G\to\T G$}
\LeftLabel{\normalfont\textbf{[\:$0$\:]}}
\UnaryInfC{$0\colon\*\to\g$}
\DisplayProof\qquad
\AxiomC{$\omega_1+\omega_2\colon G\times\g\times\g\to\T G$}
\LeftLabel{\normalfont\textbf{[\:$+$\:]}}
\UnaryInfC{$+\colon\g\times\g\to\g$}
\DisplayProof\qquad
\AxiomC{$[\omega_1,\omega_2]\colon G\times\g\times\g\to\T G$}
\LeftLabel{\normalfont\textbf{[\:$[,]$\:]}}
\UnaryInfC{$[,]\colon\g\times\g\to\g$}
\end{prooftree}
\end{theorem}
\begin{proof}
The proof that $(\g,0,+,[,])$ defines a Lie algebra is a direct consequence of the full faithfulness of the Yoneda embedding and that the Hom-functor preserves Cartesian products in the second argument. 
\end{proof}

\begin{definition}
\label{definition:lie-algebra-of-lie-group}
The \textbf{Lie algebra of a Lie group} $G$ in a Cartesian tangent category is the Lie algebra object of Theorem~\ref{theorem:lie-algebra-of-lie-group}.
\end{definition}

\subsection{Lie bracket and the adjoint representation}
\label{subsection:adjoint-representation}
So far, we have constructed a Lie algebra object $\g$ associated to $G$ by invoking the representability of the Lie functor of a Lie group. We will soon prove that such a Lie algebra object becomes a differential Lie algebra. Before doing so, we would like to have a more concrete understanding of the Lie bracket.

By the correspondence of Theorem~\ref{theorem:first-fundamental-lie}, the underlying object of $\g$ is the tangent space $\T_\e G$ at the unit of $G$. In this section, we provide an \emph{intrinsic formula} for the Lie bracket of $\g$, which is directly derived from the group structure. This extends a well-known formula of classical Lie group theory.

Unfortunately, the classical proof of this formula involves pointwise derivatives and smooth paths, all techniques that do not fit within the abstract world of tangent categories. Despite the challenge, the formula still holds.

We start by recalling the definition of the adjoint representation of a group. Throughout this section, we shall assume that $G$ is a Lie group object and we will identify the tangent space $\T_\e G$ with $\g$.

We shall denote by $\AD\colon G\times G\to G$ the action of $G$ on iteself defined by the following formula:
\begin{align*}
&\AD\colon G\times G\xrightarrow{\Delta\times\id_G}G\times G\times G\xrightarrow{\id_{G\times G}\times\i_G}G\times G\times G\xrightarrow{\id_G\times\m^\op_G}G\times G\xrightarrow{\m_G}G
\end{align*}
where we denoted by $\m^\op_G\=\tau\m_G$ the opposite of the multiplication of $G$. Concretely, when $G$ is a concrete group in the category of sets, $\AD$ is the action that takes a pair $(g,h)$ of elements of $G$ and sends it to $ghg^{-1}$.

With the action $\AD$, we define the \textbf{adjoint representation}\footnote{Cockett and Schwarz already introduced it in~\cite{cockett:lie-groups-tangent-cats}.} of $G$ as follows:
\begin{equation*}
\begin{tikzcd}
{G\times\g} && {\T(G\times G)} \\
& \g & {\T G} \\
& {\*} & G
\arrow["{z_G\times\iota_\e}", from=1-1, to=1-3]
\arrow["\Ad", dashed, from=1-1, to=2-2]
\arrow[curve={height=24pt}, from=1-1, to=3-2]
\arrow["{\T\AD}", from=1-3, to=2-3]
\arrow["\iota_\e", from=2-2, to=2-3]
\arrow[from=2-2, to=3-2]
\arrow["\lrcorner"{anchor=center, pos=0.125}, draw=none, from=2-2, to=3-3]
\arrow["{p_G}", from=2-3, to=3-3]
\arrow["{\e_G}"', from=3-2, to=3-3]
\end{tikzcd}
\end{equation*}
The outer diagram commutes, since $(z_G\times\iota_\e)\T\AD p_G=(z_Gp_G\times\iota_\e p_G)\AD=(\id_G\times!\e_G)\AD=!\e_G$. Therefore, the morphism $\Ad\colon G\times\g\to\g$ is well-defined.

One can differentiate the adjoint representation $\Ad$ in its first variable and obtain the following morphism:
\begin{align*}
&\ad\colon\g\times\g\xrightarrow{\iota_\e\times z_\g}\T(G\times\g)\xrightarrow{\T\Ad}\T\g\xrightarrow{\hat p_\g}\g
\end{align*}
where $\hat p_\g\colon\T\g\to\g$ is the differential projection of $\g$, as a differential object.

Our goal is to show that the Lie bracket $[,]$ of $\g$ is precisely $\ad$. We begin with a few technical lemmas.

\begin{lemma}
\label{lemma:adjoint-formula-1}
The following formula holds:
\begin{align*}
&[\omega_1,\omega_2]_\KL=\left\{\kappa(\id_{\T\T G}\times\T\T\m_G)\T\T\m_G+\kappa\left(\id_{\T\T G}\times(n_{\T G}\times n_{\T G})\T\T\m^\op_G\right)\T\T\m_G\right\}
\end{align*}
where:
\begin{align*}
&\kappa\=(z_Gz_{\T G}\times\iota_\e\T z_G\times\iota_\e z_{\T G})
\end{align*}
\end{lemma}
\begin{proof}
Using the defining formula of the Lie bracket of two vector fields, by direct inspection, it is not hard to see that $[\omega_1,\omega_2]_\KL$ is equal to the following expression:
\begin{align*}
&\left\{(\omega\times z_\g)\T\omega+(\id_G\times\tau)(\omega\times z_\g)\T\omega c_Gn_{\T G}\right\}
\end{align*}
We start by computing the first term:
\begin{align*}
(\omega\times z_\g)\T\omega&=~((z_G\times\iota_\e)\T\m_G\times z_\g)(\T z_G\times\T\iota_\e)\T\T\m_G   \tag{Lemma~\ref{lemma:universal-vector-field-via-iota}}\\
&=~(z_G\times\iota_\e\times z_\g\T\iota_\e)(\T\m_G\T z_G\times\id_{\T\T G})\T\T\m_G\\
&=~(z_Gz_{\T G}\times\iota_\e\T z_G\times z_\g\T\iota_\e)(\T\T\m_G\times\id_{\T\T G})\T\T\m_G       \tag{Naturality of $\T z$}\\
&=~(z_Gz_{\T G}\times\iota_\e\T z_G\times\iota_\e z_{\T G})(\id_{\T\T G}\times\T\T\m_G)\T\T\m_G       \tag{Associativity}\\
&=~\kappa(\id_{\T\T G}\times\T\T\m_G)\T\T\m_G                       \tag{Definition of $\kappa$}
\end{align*}
From this formula, we can also compute the second term:
\begin{align*}
(\id_G\times\tau)&(\omega\times z_\g)\T\omega c_Gn_{\T G}=~(\id_G\times\tau)(z_Gz_{\T G}\times\iota_\e\T z_G\times\iota_\e z_{\T G})(\id_{\T\T G}\times\T\T\m_G)\T\T\m_G c_{G}n_{\T G}\\
&=~(\id_G\times\tau)(z_Gz_{\T G}c_Gn_{\T G}\times\iota_\e\T z_Gc_G\times\iota_\e z_{\T G}c_G)(\id_{\T\T G}\times(n_{\T G}\times n_{\T G})\T\T\m_G)\T\T\m_G \tag{Naturality of $c$ and $n$}\\
&=~(\id_G\times\tau)(z_Gz_{\T G}\times\iota_\e z_{\T G}\times\iota_\e \T z_G)(\id_{\T\T G}\times(n_{\T G}\times n_{\T G})\T\T\m_G)\T\T\m_G    \tag{$z_{\T G}c_G=\T z_G$, $z_Gn_G=z_G$}\\
&=~(z_Gz_{\T G}\times\iota_\e\T z_G\times\iota_\e z_{\T G})(\id_{\T\T G}\times(n_{\T G}\times n_{\T G})\T\T\m_G^\op)\T\T\m_G  \tag{Exchanging $\tau$}\\
&=~\kappa(\id_{\T\T G}\times(n_{\T G}\times n_{\T G})\T\T\m_G^\op)\T\T\m_G \tag{Definition of $\kappa$}
\end{align*}
By putting the two terms together, we obtain the desired formula.
\end{proof}

\begin{lemma}
\label{lemma:AD}
The following formula for the action $\AD$ of any group object holds:
\begin{align*}
&\AD=\<\m_G,\pi_2,(\i_G\times\i_G)\m^\op_G\>(\id_G\times\m_G)\m_G
\end{align*}
\end{lemma}
\begin{proof}
Instead of giving a formal proof, let us see this by reasoning element-wise. From this, it is not hard to see that the formula holds in full generality. By evaluating the right hand side on two generic elements $g,h$ of $G$, we obtain the expression:
\begin{align*}
&(gh)((h)(h^{-1}g^{-1}))=ghhh^{-1}g^{-1}=ghg^{-1}
\end{align*}
where we removed the brackets by associativity. However, this is the evaluation of the $\AD$ action on $g$ and $h$.
\end{proof}

\begin{theorem}
\label{theorem:adjoint-formula}
The Lie bracket $[,]\colon\g\times\g\to\g$ of the Lie algebra object $\g$ of a Lie group $G$ coincides with $\ad$, that is, the following identity holds:
\begin{align*}
&[,]=\ad
\end{align*}
\end{theorem}
\begin{proof}
We start by recalling that, by the universal property of $\omega$, $[,]\colon\g\times\g\to\g$ is the unique morphism satisfying the following equation:
\begin{align*}
&(\id_G\times[,])\omega=[\omega_1,\omega_2]
\end{align*}
Our goal is to prove that $[\omega_1,\omega_2]=(\id_G\times\ad)\omega$. From there, by the universal property of $\omega$, we conclude the desired identity. To prove this equality, we also need to invoke the universal property of the vertical lift. In fact, the Lie bracket of two vector fields uses the $\{-\}$ operation. Recall that, for a morphism $f\colon X\to\T\T M$ which equalizes $\T p_M$ and $p_{\T M}p_Mz_M$, there exists a unique morphism $\{f\}\colon A\to\T M$ such that $f=\<fp_{\T M}z_{\T M},\{f\}l_M\>\T s_M$. Therefore, by Lemma~\ref{lemma:adjoint-formula-1}, we reduce the problem to proving the following equality:
\begin{align*}
&\LHS=\left\<\LHS p_{\T G}z_{\T G},(\id_G\times\ad)\omega l_G\right\>\T s_M
\end{align*}
where we denoted by $\LHS$ the left-hand side term:
\begin{align*}
&\LHS\=\kappa(\id_{\T\T G}\times\T\T\m_G)\T\T\m_G+\kappa\left(\id_{\T\T G}\times(n_{\T G}\times n_{\T G})\T\T\m^\op_G\right)\T\T\m_G
\end{align*}
where, by Lemma~\ref{lemma:adjoint-formula-1}, $\{\LHS\}=[\omega_1,\omega_2]_\KL$. In the following, we will also denote the right-hand term of the equation by:
\begin{align*}
&\RHS\=\left\<\LHS p_{\T G}z_{\T G},(\id_G\times\ad)\omega l_G\right\>\T s_M
\end{align*}
Let us begin by expanding the term $(\id_G\times\ad)\omega l_G$:
\begin{align*}
(\id_G\times\ad)\omega l_G&=~(\id_G\times\ad)(z_G\times\iota_\e)\T\m_G l_G    \tag{Lemma~\ref{lemma:universal-vector-field-via-iota}}\\
&=~(z_Gl_G\times\ad\iota_\e l_G)\T\T\m_G       \tag{Naturality of $l$}\\
&=~(z_Gz_{\T G}\times\ad\iota_\e l_G)\T\T \m_G \tag{$z_Gl_G=z_Gz_{\T G}$}\\
&=~(z_Gz_{\T G}\times(z_\g\times\iota_\e)\T\Ad\hat p_\g\iota_\e l_G)\T\T \m_G \tag{Definition of $\ad$}\\
&=~\kappa'(\id_{\T\T G}\times\T\Ad\hat p_\g\iota_\e l_G)\T\T\m_G
\end{align*}
where:
\begin{align*}
&\kappa'\=(z_Gz_{\T G}\times z_\g\times\iota_\e)
\end{align*}
We can also expand the term $\LHS p_{\T G}z_{\T G}$:
\begin{align*}
\LHS p_{\T G}z_{\T G}&=~\left(\kappa(\id_{\T\T G}\times\T\T\m_G)\T\T\m_G+\kappa\left(\id_{\T\T G}\times(n_{\T G}\times n_{\T G})\T\T\m^\op_G\right)\T\T\m_G\right)p_{\T G}z_{\T G}\\
&=~\kappa(\id_{\T\T G}\times\T\T\m_G)\T\T\m_Gp_{\T G}z_{\T G}       \tag{$s_Gp_G=\pi_1p_G$}\\
&=~(z_Gz_{\T G}\times\iota_\e\T z_G\times\iota_\e z_{\T G})(\id_{\T\T G}\times\T\T\m_G)\T\T\m_Gp_{\T G}z_{\T G}       \tag{Definition of $\kappa$}\\
&=~(z_Gz_{\T G}\times\iota_\e p_Gz_Gz_{\T G}\times\iota_\e z_{\T G})(\id_{\T\T G}\times\T\T\m_G)\T\T\m_G      \tag{Naturality of $pz$ and $z_Gp_G=\id_G$}\\
&=~(z_Gz_{\T G}\times!\T\T\e\times\iota_\e z_{\T G})(\id_{\T\T G}\times\T\T\m_G)\T\T\m_G       \tag{$\iota_\e p_G=!\e$}\\
&=~(z_Gz_{\T G}\times\pi_2\iota_\e z_{\T G})\T\T\m_G                       \tag{Unitality}
\end{align*}
However, we can also compute:
\begin{align*}
(\iota_\e\times z_\g)\T\Ad\T\iota_\e p_{\T G}&=~(\iota_\e p_G\times z_Gp_G)\Ad\iota_\e       \tag{Naturality of $p$}\\
&=~(!\e\times\id_G)\Ad\iota_\e                                                     \tag{$\iota_\e p_G=!\e$}\\
&=~\pi_2\iota_\e                                                                   \tag{Unitality of the action $\Ad$}
\end{align*}
Therefore:
\begin{align*}
\LHS p_{\T G}z_{\T G}&=~(z_Gz_{\T G}\times\pi_2\iota_\e z_{\T G})\T\T\m_G\\
&=~(z_Gz_{\T G}\times(\iota_\e\times z_\g)\T\Ad\T\iota_\e p_{\T G}z_{\T G})\T\T\m_G\\
&=~\kappa'(\id_{\T\T G}\times\T\Ad\iota_\e p_{\T G}z_{\T G})\T\T m_G
\end{align*}
Thus, we can rewrite the $\RHS$ term as follows:
\begin{align*}
\RHS&=~\left\<\kappa'(\id_{\T\T G}\times\T\Ad\T\iota_\e p_{\T G}z_{\T G})\T\T\m_G,\kappa'(\id_{\T\T G}\times\T\Ad\hat p_\g\iota_\e l_G)\T\T\m_G\right\>\T s_G\\
&=~\kappa'(\id_{\T\T G}\times\T\Ad)\left\<\T\iota_\e p_{\T G}z_{\T G},\hat p_\g\iota_\e l_G\right\>\T s_G\T\T\m_G \tag{Naturality of $s$}\\
&=~\kappa'(\id_{\T\T G}\times\T\Ad\T\iota_\e)\T\T\m_G             \tag{See below}\\
&=~\kappa'(\id_{\T\T G}\times(\T\iota_\e\times\T z_G)\T\T\AD)\T\T\m_G      \tag{Definition of $\Ad$}\\
&=~(z_Gz_{\T G}\times \iota_\e z_{\T G}\times\iota_\e\T z_G)(\id_{\T\T G}\times\T\T\AD)\T\T\m_G \tag{Definition of $\kappa'$}\\
&=~\kappa(\id_{\T\T G}\times\T\T\AD)\T\T\m_G                \tag{Definition of $\kappa$}
\end{align*}
where we used that:
\begin{align*}
\left\<\T\iota_\e p_{\T G}z_{\T G},\hat p_\g\iota_\e l_G\right\>\T s_G&=~\left\<\T\iota_\e p_{\T G}z_{\T G},\{\T\iota_\e\}l_G\right\>\T s_G
\tag{By linearity, $\{\T\iota_\e\}=\hat p_\g\iota_\e$}\\
&=~\T\iota_\e      \tag{Universal property of $l$}
\end{align*}
Thanks to Lemma~\ref{lemma:AD}, we can rewrite $\RHS$ as follows:
\begin{align*}
&\RHS=~\kappa(\id_{\T\T G}\times\T\T\AD)\T\T\m_G\\
&=~\kappa(\id_{\T\T G}\times\<\T\T\m_G,\pi_2,(\T\T\i_G\times\T\T\i_G)\T\T\m_G^\op\>(\id_{\T\T G}\times\T\T\m_G)\T\T\m_G)\T\T\m_G \tag{Lemma~\ref{lemma:AD}}\\
&=~(z_{G}z_{\T G}\times(\iota_\e\T z_G\times\iota_\e z_{\T G})\<\T\T\m_G,\pi_2,(\T\T\i_G\times\T\T\i_G)\T\T\m_G^\op\>(\id_{\T\T G}\times\T\T\m_G)\T\T\m_G)\T\T\m_G
\end{align*}
$\LHS$ and $\RHS$ are morphisms with codomain $\T\T G$, however, $\T\T G\cong G\times\g\times\g\times\g$. Under this isomorphism, $p_{\T G}$ and $\T p_G$ correspond to the projections $\pi_1\times\pi_2$ and $\pi_1\times\pi_3$. It is not hard to see that $\LHS$ and $\RHS$ coincide once postcomposed by $p_{\T G}$ and $\T p_G$. More precisely, we obtain that:
\begin{align*}
&\LHS p_{\T G}=\omega_2=\RHS p_{\T G}           &&\LHS\T p_G=\pi_1z_G=\RHS\T p_G
\end{align*}
So it remains to show that $\LHS$ and $\RHS$ coincide in the last component $\g$ of $\T\T G$. Let us isolate the following term in the expression we found for $\RHS$:
\begin{align*}
&(\iota_\e\T z_G\times\iota_\e z_{\T G})\<\T\T\m_G,\pi_2,(\T\T\i_G\times\T\T\i_G)\T\T\m_G^\op\>(\id_{\T\T G}\times\T\T\m_G)\T\T\m_G\\
&\qquad=~\left\<(\iota_\e\T z_G\times\iota_\e z_{\T G})\T\T\m_G,\pi_2\iota_\e z_{\T G},(\iota_\e\T z_G\times\iota_\e z_{\T G})(\T\T\i_G\times\T\T\i_G)\T\T\m_G^\o\right\>\T\T\m_G
\end{align*}
However, by evaluating $\RHS$ only the last component $\g$ of $\T\T G$, the last term $\T\T\m_G$ becomes $\T\T_\e\m_G=\T+_\g$. However, in the last component, $\T+_\g$ also coincides with $s_{\T G}$. Therefore, by the unitality of $s_{\T G}$ the term $(\iota_\e\times\iota_\e)\pi_2z_{\T G}$ cancels out. Furthermore, since the negation $n_\g$ of $\g$ is induced by the inverse map $\i_G$ of $G$, $\T\T\i_G$, evaluated onto $\g$, becomes $\T\T_\e\i_G=\T n_G$, which, however, on the last component is also equal to $n_{\T G}$. So, we obtain that this term, evaluated on the last component $\g$ of $\T\T G$, becomes:
\begin{align*}
&\left\<(\iota_\e\T z_G\times\iota_\e z_{\T G})\T\T\m_G,\pi_2\iota_\e z_{\T G},(\iota_\e\T z_G\times\iota_\e z_{\T G})(\T\T\i_G\times\T\T\i_G)\T\T\m_G^\op\right\>s_{\T G}\\
&\qquad=~(\iota_\e\T z_G\times\iota_\e z_{\T G})\T\T\m_G+(\iota_\e\T z_Gn_{\T G}\times\iota_\e z_{\T G}n_{\T G})\T\T\m_G^\op
\end{align*}
By putting all things together, the $\RHS$ term, evaluated on the last $\g$, becomes:
\begin{align*}
\RHS_\g&=~\left(z_{G}z_{\T G}\times\left((\iota_\e\T z_G\times\iota_\e z_{\T G})\T\T\m_G+(\iota_\e\T z_Gn_{\T G}\times\iota_\e z_{\T G}n_{\T G})\T\T\m_G^\op\right)\right)\T\T\m_G\\
&=~(z_Gz_{\T G}\times\iota_\e\T z_G\times\iota_\e z_{\T G})\left(((\id_{\T\T G}\times\T\T\m_G)\T\T\m_G+(\id_{\T\T G}\times\T\T\m^\op_Gn_{\T G})\T\T\m_G\right)\\
&=~\LHS_\g
\end{align*}
where we used the naturality of $n$ and of $s$.
\end{proof}

This theorem has an immediate result.

\begin{corollary}
\label{corollary:abelian-lie-groups-lie-algebra}
The Lie algebra object of an Abelian Lie group is trivial, in that $[,]=!0_\g$.
\end{corollary}
\begin{proof}
For an Abelian group object, $\AD$ becomes the second projection, that is, $\AD=\pi_2$. To see this, we use an element-wise argument and leave it to the reader to prove it more formally using the definition. $\AD(g,h)=ghg^{-1}=gg^{-1}h=h$. This implies that also $\Ad$ becomes the second projection $\pi_2\colon G\times\g\to\g$. In fact, $\Ad$ is defined as the unique morphism such that $\Ad\iota_\e=(z_G\times\iota_\e)\T\AD$. However, when $\AD=\pi_2$, this implies that $\Ad\iota_\e=\pi_2\iota_\e$. Therefore, since $\iota_\e$ is monic, $\Ad=\pi_2$. Finally, using the formula of Theorem~\ref{theorem:adjoint-formula}, we can compute:
\begin{align*}
&[,]=\ad=(\iota_\e\times z_G)\T\Ad\hat p_\g=(\iota_\e\times z_G)\pi_2\hat p_\g=\pi_2z_G\hat p_\g=\pi_2!0_\g=!0_\g
\end{align*}
where we used that $z_G\hat p_\g=!0_\g$, which is a property of the differential projection $\hat p_\g$.
\end{proof}

\subsection{The differential Lie algebra of a Lie group}
\label{subsection:differential-lie-algebra-of-tangent-lie-groups}
So far, we proved that every Lie group admits a Lie algebra object internal to the tangent category. Now, we want to show that the Lie algebra is in fact a differential Lie algebra.

We begin with a technical lemma.

\begin{lemma}
\label{lemma:additivity-and-linearity-lie-algebra}
For a differential Lie group $G$, the \PLIVFs $z_G\colon G\to\T G$ and $\omega_1+\omega_2$ are linear.
\end{lemma}
\begin{proof}
To prove that $z_G$ is linear, we only need to realize that the lift of $\*$, as a differential object, coincides with the unique isomorphism $\lambda_\*\colon\*\cong\T\*$, since $\T$ preserves the terminal object. Thus, $z_Gl_G=z_G\T z_G=(z_G\times\lambda_\*)\T z_G$. Therefore, $z_G$ is linear. Now, let us show that also $\omega_1+\omega_2$ is linear. Since $G$ is a differential Lie group, by Theorem~\ref{theorem:differential-universal-vector-field}, $\omega$ is a linear \PLIVF\!. Therefore, we can compute:
\begin{align*}
(\omega_1+\omega_2)l_G&=~\<\omega_1,\omega_2\>s_Gl_G        \tag{Definition of $u+v$}\\
&=~\<\omega_1,\omega_2\>(l_G\times l_G)\T s_G               \tag{Linearity of $s$}\\
&=~\<\omega_1l_G,\omega_2l_G\>\T s_G\\
&=~\<(\id_G\times\pi_1)\omega l_G,(\id_G\times\pi_2)\omega l_G\>\T s_G      \tag{Definition of $\omega_1$ and $\omega_2$}\\
&=~\<(\id_G\times\pi_1)(z_G\times\lambda_\g)\T\omega,(\id_G\times\pi_2)(z_G\times\lambda_\g)\T\omega\>\T s_G    \tag{Linearity of $\omega$}\\
&=~\<(z_G\times\lambda_{\g\times\g})\T(\id_G\times\pi_1)\T\omega,(z_G\times\lambda_\g)\T(\id_G\times\pi_2)\T\omega\>\T s_G\\
&=~(z_G\times\lambda_{\g\times\g})\T\<\omega_1,\omega_2\>\T s_G                       \tag{Definition of $\omega_1$ and $\omega_2$}\\
&=~(z_G\times\lambda_{\g\times\g})\T(\omega_1+\omega_2)         \tag{Definition of $u+v$}
\end{align*}
This proves that $\omega_1+\omega_2$ is linear.
\end{proof}

\begin{remark}
\label{remark:linearity-of-sum}
Note that $\omega_1+\omega_2$ is \emph{linear} and not bilinear.
\end{remark}

\begin{theorem}
\label{theorem:second-fundamental-lie}
The Lie algebra of a Lie group is a differential Lie algebra.
\end{theorem}
\begin{proof}
By Lemma~\ref{lemma:additivity-and-linearity-lie-algebra}, $z_G$ and $\omega_1+\omega_2$ are linear \PLIVF\!. However, by Theorem~\ref{theorem:first-fundamental-lie}, the isomorphisms $\varphi\colon\g(A_1\times\dots\times A_n)\to\X(A_1\times\dots\times A_n,\g)$ and their inverses preserve multilinearity. It immediately follows that $0\colon\*\to\g$ and $+\colon\g\times\g\to\g$ are linear morphisms of differential objects. The linearity of $0$ and $+$ implies that they commute with the zero and the sum morphisms of $\g$ as a differential object. Therefore, by an Eckmann-Hilton argument, they must coincide with the zero and the sum of $\g$ as a differential object. To conclude, it remains to show that $[,]$ is bilinear. Since it satisfies antisymmetry, it suffices to show that it is linear in the first argument. However, using the formula of Theorem~\ref{theorem:adjoint-formula} we can compute:
\begin{align*}
(\lambda_\g\times z_G)\T\ad&=~(\lambda_\g\times z_G)(\T\iota_\e\times\T z_G)\T\T\Ad\T\hat p_\g   \tag{Definition of $\ad$}\\
&=~(\iota_\e l_G\times z_Gl_G)\T\T\Ad\T\hat p_\g                                                 \tag{Linearity of $\iota_\e$ and $z_Gl_G=z_G\T z_G$}\\
&=~(\iota_\e\times z_G)\T\Ad l_G\T\hat p_\g            \tag{Naturality of $l$}\\
&=~(\iota_\e\times z_G)\T\Ad\hat p_\g l_\g              \tag{Compatibility between $l$ and $\hat p_\g$}\\
&=~\ad\lambda_\g                                          \tag{Definition of $\ad$}
\end{align*}
Therefore, by Lemma~\ref{lemma:differential-lie-algebras}, $\g$ becomes a differential Lie algebra.
\end{proof}

\begin{remark}
\label{remark:differential-lie-algebra-proof}
An alternative proof of Theorem~\ref{theorem:second-fundamental-lie} consists of showing that the $(\g,\g)$-\PLIVF $[\omega_1,\omega_2]$ is bilinear and then using the preservation of multilinearity by the isomorphisms $\varphi$. However, this proof requires long and tedious computations.
\end{remark}


\section{The Lie group-Lie algebra correspondence}
\label{section:functoriality}
In this section, we show that the assignment for each Lie group $G$ of its differential Lie algebra extends to a functor $\Lie$. We use this functor to introduce the notion of a \emph{Lie tangent category}, which is a Cartesian tangent category in which to each differential Lie algebra corresponds a Lie group. This generalizes the classic Lie group-Lie algebra correspondence of differential geometry.

We then consider three main examples of Cartesian tangent categories and compute the Lie algebras of the Lie groups in these categories. In particular, we show that our construction fully generalizes both the classic Lie correspondences of differential and algebraic geometry. We conclude with a concrete example on how to compute the Lie algebra of the Hopf algebra $\SL_2$ of Example~\ref{example:affine-group-schemes}.

\subsection{The Lie correspondence}
\label{subsection:functoriality}
In this section we extend the correspondence between Lie groups and differential Lie algebras to a functor $\Lie\colon\LieGr(\X,\TT)\to\LieAlg(\X,\TT)$. We also prove that this functor preserves tangent limits.

\begin{lemma}
\label{lemma:functoriality}
If $f\colon G\to H$ is a group homomorphism between two Lie groups, $\T_\e f\colon\g\to\h$ becomes a Lie algebra homomorphism.
\end{lemma}
\begin{proof}
We already know that $\T_\e f\colon\g\to\h$ is linear and therefore is additive. To conclude, it is left to prove that $\T_\e f$ preserves the Lie bracket. Thankfully, by Theorem~\ref{theorem:adjoint-formula}, we have a concrete description of the Lie bracket in terms of the adjoint representation. We start by showing that $f$ preserves $\Ad$. Since $\AD$ is entirely defined using the multiplication of the group, we leave it to the reader to show that $\AD_Gf=(f\times f)\AD_H$. We shall compute:
\begin{align*}
\Ad_G\T_\e f\iota_\e&=~\Ad_G\iota_\e\T f              \tag{Definition of $\T_\e f$}\\
&=~(z_G\times\iota_\e)\T\AD_G\T f                \tag{Definition of $\Ad$}\\
&=~(z_G\times\iota_\e)(\T f\times\T f)\T\AD_H    \tag{$f$ preserves $\AD$}\\
&=~(fz_H\times\T_\e f\iota_\e)\T\AD_H            \tag{Naturality of $z$ and def'n of $\T_\e f$}\\
&=~(f\times\T_\e f)\Ad_H\iota_\e                 \tag{Definition of $\Ad$}
\end{align*}
Using that $\iota_\e$ is monic, we conclude that $f$ preserves the adjoint representation. Therefore:
\begin{align*}
\ad_\g\T_\e f&=~(\iota_\e\times z_\g)\T\Ad_G\hat p_\g\T_\e f       \tag{Definition of $\ad$}\\
&=~(\iota_\e\times z_\g)\T\Ad_G\T f\hat p_\h                       \tag{Linearity of $\T_\e f$}\\
&=~(\iota_\e\T f\times z_\g\T\T_\e f)\T\Ad_H\hat p_\h              \tag{$f$ preserves $\Ad$}\\
&=~(\T_\e f\times\T_\e f)(\iota_\e\times z_\h)\T\Ad_H\hat p_\h     \tag{Naturality of $z$ and def'n of $\T_\e f$}\\
&=~(\T_\e f\times\T_\e f)\ad_\h                                    \tag{Definition of $\ad$}
\end{align*}
This proves that $\T_\e f$ preserves $\ad$ and therefore that $\T_\e f$ is a Lie algebra homomorphism.
\end{proof}

\begin{theorem}
\label{theorem:functoriality}
There is a functor
\begin{align*}
&\Lie\colon\LieGr(\X,\TT)\to\LieAlg(\X,\TT)
\end{align*}
that sends each Lie group $G$ to its differential Lie algebra $\g$. Furthermore, $\Lie$ is the underlying functor of a strong tangent morphism and preserves tangent limits.
\end{theorem}
\begin{proof}
Functoriality is immediate since $\T_\e$ is functorial. Furthermore, since by Theorem~\ref{theorem:Grp-linear-GCDC}, the tangent category $\LieGr(\X,\TT)$ is a linear GCDC, its tangent structure is entirely determined by the linear assignment $\T_\e$. Therefore, $\T\T_\e G\cong\T_\e G\times\T_\e G$, which is precisely the tangent bundle of $\g$ in $\DO(\X,\TT)$. Finally, since Cartesian functors preserve Lie algebra objects, $\LieAlg(\X,\TT)$ becomes a tangent category and $\T_\e$ a strong tangent morphism. To prove that $\T_\e$ preserves tangent limits, notice that, by construction, $\T_\e G$ is the tangent pullback of $p_G$ along $\e_G$. However, tangent pullbacks commute with other tangent limits. Therefore, it is not hard to see that the functor $\T_\e\colon\LieGr(\X,\TT)\to\DO(\X,\TT)$ preserves tangent limits. To conclude, notice that the forgetful functor $\LieAlg(\X,\TT)\to\DO(\X,\TT)$, which forgets the Lie algebra structure, creates tangent limits. To see this, consider a cone $\lambda_c\colon\g\to D(c)$ in $\LieAlg(\X,\TT)$ of a given diagram $D$, which is made into a limit cone by the forgetful functor. Now, consider a second cone $\theta_c\colon\h\to D(c)$ of $D$. Since $\lambda_c$ is a limit cone in $\DO(\X,\TT)$, there exists a unique morphism of cones $\gamma\colon\h\to\g$ in $\DO(\X,\TT)$. We need to prove that $\gamma$ lifts to $\LieAlg(\X,\TT)$. However, $\gamma$ is already additive, so we only need to show that it preserves the Lie bracket. However, since $\lambda_c$ and $\theta_c$ are cones in $\LieAlg(\X,\TT)$ and that $\gamma$ is a morphism of cones, it follows immediately that $(\gamma\times\gamma)[,]_\g\lambda_c=[,]_\h\gamma\lambda_c$, which, by the universal property of $\lambda_c$ implies that $\gamma$ is a Lie algebra homomorphism. Similarly, one can show that if $\lambda_c$ is a tangent limit cone in $\DO(\X,\TT)$, then it is in $\LieAlg(\X,\TT)$. This proves that $\Lie$ preserves tangent limits.
\end{proof}

Lie's Third Theorem (see~\cite[Theorem~9.13]{etingof:lie-groups}) establishes that every finitely-dimensional Lie algebra $\g$ is the Lie algebra of a simply connected Lie group $\Gamma(\g)$. This construction extends to a functor $\Gamma\colon\LieAlg_\text{fdim}\to\LieGr_\b$ from the category of finitely-dimensional Lie algebras over $\R$ to the category of connected and locally connected Lie groups. Furthermore, $\Gamma$ defines a left adjoint to the $\Lie$ functor. To see this, use Lie's Second Theorem (see~\cite[Theorem~9.12]{etingof:lie-groups}), which establishes that
\begin{align*}
&\Lie\colon\LieGr(G,H)\to\LieAlg(\Lie(G),\Lie(H))
\end{align*}
becomes a bijection if $G$ is simply connected. By replacing $G$ by $\Gamma(\g)$ for a Lie algebra $\g$, we obtain a bijection $\LieGr(\Gamma(\g),H)\cong\LieAlg(\g,\Lie(H))$, establishing that $\Gamma$ is left adjoint to $\Lie$.

The tangent bundle functor of $\SMan$ preserves connected and locally connected manifolds. Therefore, the tangent structure on $\SMan$ restricts to a tangent structure on the subcategory $\SMan_\b$ of connected and locally connected manifolds. Furthermore, the functor $\Lie$ of $\SMan_\b$ is simply a restriction of $\Lie$ of $\SMan$ and the internal Lie algebra objects of $\SMan_\b$ are finite-dimensional Lie algebras over $\R$. Therefore, $\Gamma$ can be defined as the left adjoint of the functor $\Lie$ of $\SMan_\b$.

This suggests the following definition.

\begin{definition}
\label{definition:lie-tangent-category}
A \textbf{Lie tangent category} consists of a Cartesian tangent category whose functor $\Lie$ of Theorem~\ref{theorem:functoriality} admits a left adjoint.
\end{definition}

Therefore, the classic Lie group-Lie algebra correspondence of differential geometry establishes that $\SMan_\b$ is a Lie tangent category.

We now look at some examples of tangent categories and compute the Lie algebras of their Lie groups.

\subsection{The trivial case}
\label{subsection:example-trivial}
We begin by considering the trivial tangent structure on a Cartesian category (Example~\ref{example:tangent-categories-trivial}). Since the projection is simply the identity, every object $M$ admits a tangent space at each point $x\colon\*\to M$, the tangent space $\T_xM$ being simply the terminal object $\*$. Therefore, by Theorem~\ref{theorem:first-fundamental-lie}, every group object is a Lie group whose differential Lie algebra is simply the terminal object equipped with the uniquely defined Lie bracket.

In particular, the functor $\Lie$ becomes the constant functor to the trivial Lie algebra $\*$. This example shows that, in general, we should not expect $\Lie$ to be either a right adjoint or fully faithful.

\subsection{Lie group theory in differential geometry: a tangent-categorical perspective}
\label{subsection:example-sman}
Consider now the Cartesian tangent category $(\SMan,\TT)$ (Example~\ref{example:tangent-categories-sman}). This is a well-displayed tangent category, that is, the projection $p_M\colon\T M\to M$ of every object $M$ is a tangent display map. This is because $p_M$ is a submersion (in fact a fibre bundle) and that, by~\cite[Theorem~2.31]{cruttwell:tangent-display-maps}, submersions in $\SMan$ are equivalent to tangent display maps. Therefore, by Lemma~\ref{lemma:tangent-display-lie}, every group object of $\SMan$ is a Lie group object. This means that our notion of Lie group objects coincides with the classic notion of a Lie group (group internal to $\SMan$).

The differential Lie algebra of a Lie group in $\SMan$ also coincides with the usual construction of the Lie algebra of left-invariant vector fields of a Lie group. This can be easily seen directly from the definition of the Lie functor $\g(-)$ of a group $G$ or by recalling that the Lie algebra of a Lie group can also be defined as the tangent space at the unit, that is, $\T_\e G$, whose Lie bracket is given by $\ad$, as in Theorem~\ref{theorem:adjoint-formula}.

\subsection{Lie group theory in algebraic geometry: a tangent-categorical perspective}
\label{subsection:affine-group-schemes}
Now, consider the Cartesian tangent category $(\Aff_R,\TT^\Omega)$ (Example~\ref{example:tangent-categories-affine}). We have already established in Example~\ref{example:affine-group-schemes} that group objects of $\Aff_R$ are affine group schemes, which are dual to commutative Hopf algebras. Furthermore, the category $\Aff_R$ is complete, and the tangent bundle functor is representable; thus, it preserves all limits. Therefore, $\Aff_R$ is well-displayed, and by Lemma~\ref{lemma:tangent-display-lie}, every group object in $\Aff_R$ becomes a Lie group. Thus, according to Theorem~\ref{theorem:first-fundamental-lie}, every affine group scheme is a Lie group object of $\Aff_R$.

In the literature, the \emph{Lie algebra} of an affine group scheme $H$ is defined as the $R$-module of left-invariant derivations of $H$, which consist of $R$-linear morphisms $\delta\colon H\to H$ satisfying the following two equations:
\begin{align*}
&\delta(xy)=x\delta(y)+y\delta(x)       &&\Delta(\delta(x))=(\id_H\x\delta)(\Delta(x))
\end{align*}
where $\Delta$ is the comultiplication of $H$. \cite[The first theorem of Section~12.2]{milne:algebraic-groups} establishes that left-invariant derivations are in bijection with $\epsilon$-derivations, which consist of $R$-linear morphisms $\delta\colon H\to R$, satisfying the following Leibniz condition:
\begin{align*}
&\delta(xy)=\epsilon(x)\delta(y)+\epsilon(y)\delta(x)
\end{align*}
where $\epsilon$ denotes the counit of $H$. We may denote the Lie algebra of left-invariant derivations of $H$ by $\LIDer(H)$.

Our construction produces for each affine group scheme $H$ a differential Lie algebra internal to the category of affine schemes. Differential objects in $\Aff_R$ are equivalent to $R$-modules (Example~\ref{example:differential-objects-affine}). In particular, there is an equivalence of categories between $\DO(\Aff_R)$ and the opposite category $\Mod_R^\op$ of $R$-modules.

Therefore, the functor $\Lie$ of Theorem~\ref{theorem:functoriality} sends each affine group scheme $H\in\Gr(\Aff_R)$ to a Lie algebra internal to $\Mod_R^\op$, that is, a coLie $R$-algebra $\h$. We now prove that the dual of $\h$, that is, $\Mod_R(\h,R)$, coincides with the Lie algebra of left-invariant derivations of $H$.

\begin{theorem}
\label{theorem:lie-algebra-of-affine-group-schemes}
Consider an affine group scheme $H$. Under the equivalence $\DO(\Aff_R)\simeq\Mod_R^\op$, the dual of the internal Lie algebra $\Lie(H)\in\DO(\Aff_R)$ of $H$ corresponds to the Lie algebra $\LIDer(H)$.
\end{theorem}
\begin{proof}
The internal Lie algebra of an affine group scheme $H$ is the tangent space $\T_\e^\Omega H$ at the unit. Since $\Aff_R$ is dual to the category $\cAlg_R$ of commutative $R$-algebras, the pullback that defines $\T_\e^\Omega H$ becomes a pushout in $\cAlg_R$:
\begin{equation*}
\begin{tikzcd}
H & R \\
{\T^\Omega H} & {\T_\e^\Omega H}
\arrow["\epsilon", from=1-1, to=1-2]
\arrow["{p_H}"', from=1-1, to=2-1]
\arrow[from=1-2, to=2-2]
\arrow[from=2-1, to=2-2]
\arrow["\lrcorner"{anchor=center, pos=0.125, rotate=180}, draw=none, from=2-2, to=1-1]
\end{tikzcd}
\end{equation*}
where $\epsilon$ is the counit of $H$.

\cite[Theorem~4.19]{cruttwell:algebraic-geometry} establishes an equivalence $\DB(\Aff_R)\simeq\Mod^\op$ between the category of all differential bundles of $\Aff_R$ and the opposite category of modules over arbitrary $R$-algebras. Thus, under this equivalence, this pushout diagram corresponds to a pushout diagram of $\Mod$ of the following form:
\begin{equation*}
\begin{tikzcd}
H & R \\
{\Omega_H} & M
\arrow["\epsilon", from=1-1, to=1-2]
\arrow["\d"', from=1-1, to=2-1]
\arrow[from=1-2, to=2-2]
\arrow[from=2-1, to=2-2]
\arrow["\lrcorner"{anchor=center, pos=0.125, rotate=180}, draw=none, from=2-2, to=1-1]
\end{tikzcd}
\end{equation*}
where $\Omega_H$ denotes the modules of K\"ahler differentials of $H$ and $\d$ is the universal derivation. Therefore, the $R$-module $M$ that corresponds to $\T_\e^\Omega H$ is the pushout of $\Omega_H$ along the counit map, that is:
\begin{align*}
&M=\Omega_H\x_HR
\end{align*}
where $R$ is regarded as an $H$-module via the counit. We now show that the map $\d\colon H\to\Omega_H\x_HR$ which sends each $x$ to $\d x\x1$ is the $\epsilon$-universal derivation, in that it classifies $\epsilon$-derivations of $H$. To this end, consider an $\epsilon$-derivation $\delta\colon H\to R$ and let us define $\hat\delta\colon\Omega_H\x_HR\to R$ as the map that sends each generator $\d x$ to $\delta(x)$. This map is well-defined, since $\delta$ is an $\epsilon$-derivation:
\begin{align*}
&\d(xy)=\delta(xy)=\epsilon(x)\delta(y)+\epsilon(y)\delta(x)=\delta(\epsilon(x)y+\epsilon(y)x)\\
&\qquad=\d(\epsilon(x)y+\epsilon(y)x)=\epsilon(x)\d(y)+\epsilon(y)\d x=x\.\d y+y\.\d x
\end{align*}
where in the last step we used the identification between the action of $H$ onto $\Omega_H$ and the action of $H$ on $R$ induced by the counit in $\Omega_H\x_HR$. It is also $R$-linear since $\delta$ and $\d$ are. Furthermore, by definition, $\hat\delta$ splits $\delta$ along $\d\colon H\to\Omega_H\x_HR$. Finally, if $\hat\delta'\colon\Omega_H\x_HR$ was another map of $R$-modules that split $\delta$ along $\d$, then, on generators, $\hat\delta'(\d x)=\delta(x)=\hat\delta(x)$. Therefore, as expected the $R$-module $\Omega_H\x_HR$ classifies $\epsilon$-derivations. In particular, this means that the dual of $\Omega_H\x_HR$, that is, the $R$-module of $R$-linear maps $\Omega_H\x_HR\to R$ is isomorphic to the $R$-module $\Der_\epsilon(H)$ of $\epsilon$-derivations.

Now, consider the Lie algebra $\h(R)$, where $\h(-)$ denotes the Lie functor of the affine group scheme $H$. Since $\h$ is representable, $\h(R)\cong\Aff_R(R,\h)$, where $\h=\Sym_\h(M)$ is the internal Lie algebra of $H$, with $M=\Omega_H\x_HR$. However, since $\Aff_R$ is the opposite category of $\cAlg_R$, then $\h(R)\cong\cAlg_R^\op(R,\h)=\cAlg(\h,R)$. However, since $\h=\Sym_\h(M)$, a morphism of rings $\Sym_\h(M)\to R$ corresponds to a morphism of $R$-modules $M\to R$. Therefore, the underlying $R$-module of $\h(R)$ corresponds to $M^\*$. Furthermore, as we proved before, $M^\*$ is also isomorphic, as an $R$-module, to the $R$-module $\LIDer(H)$ of left-invariant derivations of $H$. The resulting isomorphism $\h(R)\to\LIDer(H)$ of $R$-modules sends a left-invariant vector field $X\colon H\x R\to\T H$ in $\Aff_R$, which, as an algebra morphism is $\tilde X\colon\T H\to H$, to the left-invariant derivation $\delta_X\colon H\to H$ defined as $\delta_X(x)\=\tilde X(\d x)$. However, this is a special case of the general correspondence between vector fields $X\colon\T A\to A$ of $\Aff_R$ and derivations over the underlying algebra $A$. This classification follows from~\cite[Corollary~4.5.3]{ikonicoff:operadic-algebras-tangent-cats} as shown in~\cite[Example~4.5.5]{ikonicoff:operadic-algebras-tangent-cats}. In particular, this correspondence preserves the Lie algebra structure. This proves that the isomorphism $\h(R)\to\LIDer(H)$ is a Lie algebra isomorphism. Finally, the Lie algebra $M^\*$ is isomorphic, as a Lie algebra, to $\h(R)$.
\end{proof}

Since the category $\Gr(\Aff_R)$ is dual to the category of commutative Hopf algebras, we obtain the following corollary.

\begin{corollary}
\label{corollary:affine-group-schemes}
The opposite of the category of commutative Hopf $R$-algebras carries the structure of a linear GCDC. The linear assignment sends each commutative Hopf algebra $H$ to the $R$-module $\Omega_H\x_HR$, whose dual is the $R$-module $\LIDer(H)$ of left-invariant derivations of $H$.
\end{corollary}

We conclude with a concrete example\footnote{We thank Edmund Heng for suggesting this example to us.}.

\begin{example}
\label{example:special-linear-algebra}
In Example~\ref{example:affine-group-schemes}, we mentioned the commutative Hopf algebra $\SL_2$, which is the $R$-algebra so defined:
\begin{align*}
&\SL_2=\frac{R[x_{1,1},x_{1,2},x_{2,1},x_{2,2}]}{(\det(x_{i,j})-1)}
\end{align*}
Now, we want to compute the internal Lie algebra of $\SL_2$. We now know that the Lie algebra is the tangent space at the unit. This means that, in the category $\cAlg_R$, the Lie algebra is realized as the pushout of $p_{\SL_2}\colon\SL_2\to\T^\Omega\SL_2$ with $\e\colon\SL_2\to R$. For starters, we shall unwrap the definition of the tangent bundle of $\SL_2$. Since the tangent bundle functor $\T^\Omega$ is a left adjoint in $\cAlg_R$, it preserves quotients. Furthermore, it also sends a freely generated algebra $R[X]$ to $R[X,\d X]$. So, we have:
\begin{align*}
&\T^\Omega\SL_2=\T^\Omega\left(\frac{R[x_{1,1},x_{1,2},x_{2,1},x_{2,2}]}{(\det(x_{i,j})-1)}\right)=\frac{R[x_{1,1},x_{1,2},x_{2,1},x_{2,2},\d x_{1,1},\d x_{1,2},\d x_{2,1},\d x_{2,2}]}{(\det(x_{i,j})-1,\d(\det(x_{i,j})-1))}
\end{align*}
By the Leibniz rule, we can rewrite $\d(\det(x_{i,j})-1)$ as follows:
\begin{align*}
&\d(\det(x_{i,j})-1)=\d(x_{1,1}x_{2,2}-x_{1,2}x_{2,1}-1)=x_{1,1}\d x_{2,2}+x_{2,2}\d x_{1,1}-x_{1,2}\d x_{2,1}-x_{2,1}\d x_{1,2}
\end{align*}
Therefore, we have:
\begin{align*}
&\T^\Omega\SL_2=\frac{R[x_{1,1},x_{1,2},x_{2,1},x_{2,2},\d x_{1,1},\d x_{1,2},\d x_{2,1},\d x_{2,2}]}{(\det(x_{i,j})-1,x_{1,1}\d x_{2,2}+x_{2,2}\d x_{1,1}-x_{1,2}\d x_{2,1}-x_{2,1}\d x_{1,2})}
\end{align*}
Now, to obtain the Lie algebra, we need to pushout $\T^\Omega\SL_2$ along the counit map $\e\colon\SL_2\to R$. However, $\e$ sends $x_{1,1}$ and $x_{2,2}$ to $1$ and $x_{1,2}$ and $x_{2,1}$ to $0$, so we obtain that:
\begin{align*}
&\T_\e^\Omega\SL_2=\frac{R[\d x_{1,1},\d x_{1,2},\d x_{2,1},\d x_{2,2}]}{(\d x_{2,2}+\d x_{1,1})}=\frac{R[\d x_{1,1},\d x_{1,2},\d x_{2,1},\d x_{2,2}]}{(\tr(x_{i,j}))}
\end{align*}

Recall that $\SL_2$ is the representing object of the functor $\SL_2(-)$ (see Example~\ref{example:affine-group-schemes}). If the base ring $R$ is the ring of real numbers $\R$, then $\SL_2(\R)$ becomes the Lie special linear group of differential geometry, whose Lie algebra is $\sl_2(\R)$, which is the Lie algebra of $(2\times 2)$-matrices $M$ whose trace vanishes, that is, $\tr(M)=0$.

Let us define the functor $\sl_2(-)\colon\cAlg_R^\op\to\Mod_R$ which sends each algebra $A$ to the special linear Lie algebra $\sl_2(A)$ with coefficients in $A$. As per $\SL_2(-)$, also $\sl_2(-)$ is representable and it is represented by $\sl_2$:
\begin{align*}
&\sl_2=\frac{R[x_{1,1},x_{1,2},x_{2,1},x_{2,2}]}{(\tr(x_{i,j}))}
\end{align*}
where $\tr(x_{i,j})=x_{1,1}+x_{2,2}$. In fact, an algebra morphism $M\colon\sl_2\to A$ is fully determined by its value on the variables, that is, $M(x_{i,j})$. However, $\tr(M(x_{i,j}))=0$ since $M$ is a algebra morphism. However, $\sl_2$ coincides precisely with $\T_\e^\Omega\SL_2$. Finally, to compute the Lie bracket, we know that the Lie bracket of $\T_\e G$ will coincide with the one defined in the literature of algebraic geometry. However, for each $A\in\cAlg_R$, the Lie bracket of $\sl_2(A)$ is defined as the commutator between two matrices $M$ and $N$ with vanishing trace. We can work out explicitly the commutator:
\begin{align*}
&[M,N]=MN-NM=\begin{bmatrix}
x_{1,1}&x_{1,2}\\
x_{2,1}&x_{2,2}
\end{bmatrix}
\begin{bmatrix}
y_{1,1}&y_{1,2}\\
y_{2,1}&y_{2,2}
\end{bmatrix}
-\begin{bmatrix}
y_{1,1}&y_{1,2}\\
y_{2,1}&y_{2,2}
\end{bmatrix}
\begin{bmatrix}
x_{1,1}&x_{1,2}\\
x_{2,1}&x_{2,2}
\end{bmatrix}\\
&\qquad=
\begin{bmatrix}
x_{1,2}y_{2,1}-x_{2,1}y_{1,2}&x_{1,2}(y_{2,2}-y_{1,1})-y_{1,2}(x_{2,2}-x_{1,1})\\
x_{2,1}(y_{1,1}-y_{2,2})-y_{2,1}(x_{1,1}-x_{2,2})&x_{2,1}y_{1,2}-x_{1,2}y_{2,1}
\end{bmatrix}
\end{align*}
From this formula, we conclude that the internal Lie bracket is the algebra morphism $[,]\colon\sl_2\to\sl_2\x\sl_2$ fully defined by:
\begin{align*}
&[,](x_{1,1})=x_{1,2}y_{2,1}-x_{2,1}y_{1,2}     &&[,](x_{1,2})=x_{1,2}(y_{2,2}-y_{1,1})-y_{1,2}(x_{2,2}-x_{1,1})\\
&[,](x_{2,1})=x_{2,1}(y_{1,1}-y_{2,2})-y_{2,1}(x_{1,1}-x_{2,2})     &&[,](x_{2,2})=x_{2,1}y_{1,2}-x_{1,2}y_{2,1}
\end{align*}
\end{example}


\printbibliography

\end{document}